\documentclass[11pt, letterpaper]{amsart}
\usepackage[left=1in,right=1in,bottom=1in,top=1in]{geometry}
\usepackage{amsmath,amsthm,amssymb}
\usepackage{calligra,mathrsfs}
\usepackage[all,cmtip]{xy}
\usepackage{cite}
\usepackage{hyperref}
\usepackage{enumerate}
\usepackage{graphicx}
\usepackage{enumitem} \setlist[enumerate]{label={\upshape(\roman*)}}

\newcommand{\on}[1]{\operatorname{#1}}
\newcommand{\mathfont}{\mathbf}

\newcommand{\ZZ}{\mathfont Z}

\newcommand{\NN} {\mathfont N}
\newcommand{\FF}{\mathfont F}

\newcommand{\PP}{\mathfont{P}}

\DeclareFontFamily{OT1}{rsfs}{}
\DeclareFontShape{OT1}{rsfs}{n}{it}{<-> rsfs10}{}
\DeclareMathAlphabet{\mathscr}{OT1}{rsfs}{n}{it}

\newcommand{\Fscr}{\mathscr{F}}
\newcommand{\Gscr}{\mathscr{G}}

\newcommand{\Lscr}{\mathscr{L}}

\newcommand{\Hom}{\on{Hom}}

\newcommand \tensor[1] {\otimes_{#1}}
\renewcommand{\Im}{\on{Im}}

\newcommand{\ord}{\on{ord}}

\newcommand{\into}{\hookrightarrow}

\DeclareMathOperator{\nil}{nil}
\DeclareMathOperator{\bij}{bij}

\newcommand{\D}{\mathbf{D}}
\renewcommand{\O}{\mathcal{O}}

\newcommand{\Spec}{\on{Spec}}

\newcommand{\Ga}{\mathfont{G}_a}

\newcommand{\gr}{\on{gr}}

\theoremstyle{plain}
\newtheorem{lem}{Lemma}
\newtheorem{thm}[lem]{Theorem}
\newtheorem{prop}[lem]{Proposition}
\newtheorem{cor}[lem]{Corollary}

\newtheorem{fact}[lem]{Fact}

\theoremstyle{definition}
\newtheorem{defn}[lem]{Definition}
\newtheorem{example}[lem]{Example}
\newtheorem{remark}[lem]{Remark}

\numberwithin{equation}{section}
\numberwithin{lem}{section}

\newcommand{\spn}{\on{span}}
\newcommand{\HOM}{\underline{\Hom}}
\newcommand{\Ebar}{\overline{E}}
\newcommand{\coef}{\on{coef}}
    
\renewcommand*{\L}{\mathscr{L}}

\newcommand{\n}{\on{n}} 

\title{$a$-Numbers of Curves in Artin-Schreier Covers}
\author{Jeremy Booher and Bryden Cais}
\date{\today}

\email{jeremybooher@math.arizona.edu}
\address{Department of Mathematics\\
  University of Arizona\\
  Tucson, Arizona 85721}

\email{cais@math.arizona.edu}
\address{Department of Mathematics\\
  University of Arizona\\
  Tucson, Arizona 85721}

\begin{document} 

\CompileMatrices
\UseTips

\begin{abstract}
 Let $\pi : Y \to X$ be a branched $\mathbf{Z}/p \mathbf{Z}$-cover of smooth, projective, geometrically connected curves over a perfect field of characteristic $p>0$.  We investigate the relationship between the $a$-numbers of $Y$ and $X$ and the ramification of the map $\pi$.  This is analogous to the relationship between the genus (respectively $p$-rank) of $Y$ and $X$ given the Riemann-Hurwitz (respectively Deuring--Shafarevich) formula.  Except in special situations, the $a$-number of $Y$ is not determined by the $a$-number of $X$ and the ramification of the cover, so we instead give bounds on the $a$-number of $Y$.  We provide examples showing our bounds are sharp.  The bounds come from a detailed analysis of the kernel of the Cartier operator.
 \end{abstract}
 

\maketitle

\section{Introduction}

Let $k$ be a field and 
$\pi:Y\rightarrow X$ a finite morphism of smooth, projective, and geometrically connected curves
over $k$ that is generically Galois with group $G$.  The most fundamental numerical invariant of a curve
is its genus, and the famous {\em Riemann--Hurwitz formula} says that the genus of $Y$
is determined by that of $X$ and the ramification of the cover $\pi$: letting $S \subset X(\overline{k})$ denote the branch locus, 
\begin{equation}
    2g_Y-2 = |G| \cdot (2g_X-2) + \sum_{y\in \pi^{-1}(S)} \sum_{i\ge 0} (|G_i(y)| - 1).\label{RH}
\end{equation}
Here $G_i(y)\leqslant G$ is the $i$-th ramification
group (in the lower numbering) at $y$.

When $k$ is perfect of characteristic $p>0$, which we will assume henceforth, there are 
important numerical invariants of curves beyond the genus coming from the existence of the Frobenius
morphism.  Writing $\sigma$ for the $p$-power Frobenius automorphism of $k$,  the {\em Cartier operator} is a $\sigma^{-1}$-semilinear map $ V: H^0(X,\Omega^1_{X/k})\rightarrow H^0(X,\Omega^1_{X/k})$ which is dual to the pullback by absolute Frobenius on $H^1(X,\O_X)$ using Grothendieck--Serre duality.

The Cartier operator gives 
the $k$-vector space of holomorphic differentials on  $X$ the structure of 
a (left) module of finite length over the (non-commutative in general) polynomial ring $k[V]$.  Fitting's Lemma
provides a canonical direct sum decomposition of $k[V]$-modules
\begin{equation*}
    H^0(X,\Omega^1_{X/k}) = H^0(X,\Omega^1_{X/k})^{\bij} \oplus H^0(X,\Omega^1_{X/k})^{\nil}
\end{equation*}
with $V$ bijective (respectively nilpotent) on $H^0(X,\Omega^1_{X/k})^{\star}$ for $\star=\bij$ (respectively
$\star=\nil$).
Let us write $f_X$ for the $k$-dimension of $H^0(X,\Omega^1_{X/k})^{\bij}$; this integer is 
called the {\em $p$-rank} of $X$, or more properly of the Jacobian $J_X$ of $X$, 
since one also has the description $f_X = \dim_{\FF_p} \Hom(\mu_p,J_X[p])$.
When $\pi:Y\rightarrow X$ is a branched $G$-cover with $G$ a $p$-group,
the {\em Deuring--Shafarevich formula} relates the $p$-ranks of $X$ and $Y$:
\begin{equation}
    f_Y - 1 = |G|\cdot (f_X - 1) + \sum_{y\in \pi^{-1}(S)} (|G_0(y)|-1)\label{DS}
\end{equation}
Like the Riemann--Hurwitz formula, (\ref{DS}) says that the numerical invariant $f_Y$
of $Y$ is determined by $f_X$ and the ramification of $\pi$; unlike the Riemann--Hurwitz formula,
it {\em only} applies when $G$ has $p$-power order, and requires only limited information about the ramification filtration.  As Crew points out in \cite[Remark 1.8.1]{crew84}, there can be no 
version of the Deuring--Shafarevich formula if $G$ is not assumed to be a $p$-group, 
since (for example) if $p> 2$ any elliptic curve $E$ over $k$ is a $\ZZ/2\ZZ$-cover of the projective line
branched at exactly 4 points (necessarily with ramification degree 2), but $f_E$ can be 0 or 1,
so that $f_E$ is {\em not} determined by $f_{\PP^1}=0$ and the ramification of $\pi: E\rightarrow \PP^1$.
Of course, thanks to the solvability of $p$-groups, the essential case of (\ref{DS})
is when $G=\ZZ/p\ZZ$.

Since the $k$-dimension $\delta_X$ of the nilpotent part $H^0(X,\Omega^1_{X/k})^{\nil}$
satisfies $\delta_X= g_X - f_X$, together the Riemann--Hurwitz and Deuring--Shafarevich formulae
provide a similar formula relating $\delta_{X}$, $\delta_{Y}$, and the (wild) ramification of $\pi$ for any $p$-group branched cover $\pi: Y\rightarrow X$.
Beyond this fact, very little seems to be understood about the behavior of the nilpotent part
in $p$-group covers.  

In this paper, we will study the behavior of the {\em $a$-number} of curves in 
branched $\ZZ/p\ZZ$-covers $\pi:Y\rightarrow X$.  By definition, the $a$-number
of a curve $C$ is
\begin{equation}
    a_{C}:=\dim_k \ker \left(V: H^0(C,\Omega^1_{C/k})\rightarrow H^0(C,\Omega^1_{C/k})\right).
\end{equation}
Equivalently, $a_{C}$ is the number of nonzero cyclic direct summands in the invariant factor
decomposition of $H^0(C,\Omega^1_{C/k})^{\nil}$ as a $k[V]$-module.  Yet a third interpretation
is $a_{C}=\dim_{k} \Hom(\alpha_p, J_C[p])$, where $\alpha_p$ denotes the group-scheme $\ker(F: \Ga\rightarrow \Ga)$ over $k$.\footnote{The equivalence of
this description with the given definition follows from Dieudonn\'e theory and theorem of Oda \cite[Corollary 5.11]{Oda},
which provides a canonical isomorphism of $k[V]$-modules $H^0(C,\Omega^1_{C/k})\simeq k\otimes_{k,\sigma^{-1}} \D(J_C[F])$, where $\D(\cdot)$ is the contravariant Dieudonn\'e
module.}
A curve $X$ is said to be \emph{ordinary} if $a_X = 0$.  For elliptic curves, the $a$-number is $0$ or $1$ depending on whether the curve is ordinary or supersingular in the standard senses.

Although this fundamental numerical invariant of curves in positive characteristic has been extensively studied
({\em e.g.} \cite{CMHyper},\cite{Fermat},\cite{ReBound},\cite{ElkinPriesanum1},\cite{Johnston},\cite{ElkinCyclic},\cite{Suzuki},\cite{Dummigan}, \cite{FermatHurwitz},\cite{Frei} \cite{zhou19}
), it remains rather mysterious.  
When $p=2$ and $X$ is ordinary, Voloch \cite{VolochChar2} establishes
an explicit formula for $a_Y$ in terms of the ramification of $\pi$ and the genus of $X$.
If in addition $X=\PP^1$, Elkin and Pries \cite{ElkinPriesChar2} show that 
this data completely determines the Ekedahl-Oort type of $J_Y[p]$;
note that this situation is quite special, 
as every Artin-Schreier cover of $\PP^1$ in characteristic 2 is hyperelliptic.
For general $p$, Farnell and Pries  \cite{fp13} study branched $\ZZ/p\ZZ$-covers $\pi: Y\rightarrow \PP^1$,
and prove that there is an explicit formula for $a_Y$ in terms of the ramification of $\pi$
whenever the unique break in the ramification filtration at every ramified point is a divisor of $p-1$. 
Unfortunately, there can be no such ``$a$-number formula"
in the spirit of (\ref{DS}) in general: simple examples with $p>2$ show that there are $\ZZ/p\ZZ$-covers 
even of $X=\PP^1$ branched only at $\infty$ which have {\em identical} ramification filtrations,
but different $a$-numbers; {\em cf}. Example~\ref{ex:selectpolys}.

Nonetheless, we will prove that the possibilities for the $a$-number of
$Y$ are tightly constrained by the $a$-number of $X$ and the ramification of $\pi$:

\begin{thm}\label{thm:main}
       Let $\pi:Y\rightarrow X$ be a finite morphism of smooth, projective and geometrically connected curves
       over a perfect field $k$ of characteristic $p>0$ that is 
       generically Galois with group $\ZZ/p\ZZ$.
       Let $S\subseteq X(\overline{k})$ be the finite set of geometric closed points over which $\pi$
       ramifies, and for $Q\in S$ let $d_Q$ be the unique break in the lower-numbering ramification filtration
    at the unique point of $Y$ over $Q$.  Then for any $1\le j \le p-1$,
    \[
       \sum_{Q\in S} \sum_{i=j}^{p-1} \left( \left\lfloor\frac{id_Q}{p}\right\rfloor - \left\lfloor\frac{id_Q}{p} - \left(1-\frac{1}{p}\right)\frac{jd_Q}{p}\right\rfloor\right)
       \le a_Y \le  pa_X + \sum_{Q\in S} \sum_{i=1}^{p-1} \left( \left\lfloor\frac{id_Q}{p}\right\rfloor - (p-i) \left\lfloor\frac{id_Q}{p^2}\right\rfloor  \right).
    \]
\end{thm}

In fact, our main result (Theorem \ref{thm:cleanbounds}) features a slightly sharper---if somewhat messier and less explicit in general --- upper bound,
and Theorem \ref{thm:main} is an immediate consequence of this.  

\begin{remark}
The lower bound is largest for $j\approx p/2$, and in applications we will take $j=\lceil p/2\rceil$.
Fixing $X$ and $S\subseteq X(\overline{k})$ and writing $T:=(p-1)\sum_{Q\in S} d_Q$,
elementary estimates show that our lower (respectively upper) bound is asymptotic to 
$(1-\frac{1}{p^2})\frac{T}{4}$ (respectively $(1-\frac{1}{2p})\frac{T}{3}$) as $T\rightarrow 
\infty$; see Corollary \ref{cor:explicit} for more precise estimates.
Equivalently, since $g_Y\sim \frac{T}{2}$ as $T\rightarrow \infty$ by Riemann--Hurwitz,
our lower and upper bounds are asymptotic to $(1-\frac{1}{p^2}) \frac{1}{2}g_Y$ and $(1-\frac{1}{2p}) \frac{2}{3}g_Y$,
respectively, as $g_Y\rightarrow \infty$ with $X$ and $S$ fixed.  In contrast, $f_Y/ g_Y$ approaches $0$ as $T \to \infty$.
\end{remark}

\begin{remark} \label{rmk:trivialbounds}
Using only information about $X$, elementary arguments give ``trivial'' bounds
\[
\dim_k \ker \left ( V:  H^0(X,\Omega^1_X(E_0)) \to  H^0(X,\Omega^1_X(E_0)) \right)   \leq a_Y \leq p \cdot g_X - p \cdot f_X + \sum_{Q \in S} \frac{1}{2}(p-1)(d_Q-1)
\]
where $E_0 =  \sum_{Q \in S} (d_Q - \lfloor d_Q/p \rfloor)[Q]) $.  
The trivial lower bound comes from the fact that there is an inclusion $\Omega^1_X(E_0) \into \pi_*\Omega^1_Y$ compatible with the Cartier operator: see Lemma~\ref{lem:cartierformula}.  
The trivial upper bound comes from the fact that $a_Y + f_Y \leq g_Y$ and from applying the Riemann-Hurwitz formula for the genus and Deuring-Shafarevich formula for the $p$-rank to $\pi : Y \to X$. 

The lower bound is explicit.  
When $X =\PP^1$, for a divisor $D = \sum n_i [P_i]$ with $n_i \geq 0$ we know that
\[
\dim_k \ker \left( V : H^0(X,\Omega^1_{X}(D)) \to H^0(X,\Omega^1_{X}(D))\right)  = \sum_i \left(n_i - \left\lceil \frac{n_i}{p} \right\rceil\right)
\]
so the lower bound is explicit.  For general $X$, a theorem of Tango (Fact~\ref{fact:tango}, the main theorem of \cite{tango72}) generalizes this to provide information about the kernel of the Cartier operator when the degree of $D$ is sufficiently large.

Theorem~\ref{thm:main} is substantially better than the trivial bounds: see Example~\ref{ex:numerics} for an illustration.
\end{remark}

Note that when $p=2$ and $a_X=0$, the upper and lower bounds of Theorem \ref{thm:cleanbounds} (with $j=1$) {\em coincide}, and we recover Voloch's formula \cite[Theorem 2]{VolochChar2};
see Remark~\ref{rmk:char2}.
Similarly, when $p$ is odd, all $d_Q$ divide $p-1$, and $a_X=0$, we prove in Corollary~\ref{cor:anumformulaordinary} that 
our (sharpest) upper bound and our lower bound with $j=\lceil p/2\rceil = (p+1)/2$ also coincide, thereby establishing the following ``$a$-number formula:"

\begin{cor}\label{cor:anumformula}
    With hypotheses and notation as in Theorem \ref{thm:main} and $p$ odd, assume that $d_Q | (p-1)$ for all $Q\in S$
    and that $X$ is {\em ordinary} $(${\em i.e.} $a_X=0)$.  Then
    \[
           a_Y = \sum_{Q\in S} a_Q\qquad\text{where}\qquad a_Q:=\frac{(p-1)}{2}(d_Q-1) - \frac{p-1}{d_Q}\left\lfloor \frac{(d_Q-1)^2}{4}\right\rfloor .
    \]
\end{cor}
Specializing Corollary \ref{cor:anumformula} to the case of $X=\PP^1$ recovers the main result of \cite{fp13}.

To get a sense of the bounds in Theorem \ref{thm:main}, in \S\ref{sec:examples} we work out a number of
examples.  We show in particular that for $X=\PP^1$ our upper bound of Theorem \ref{thm:cleanbounds} is {\em sharp},
with the family of covers $y^p - y = t^{-d}$ (for $t$ a choice of coordinate on $\PP^1$)
achieving the upper bound for all $p>2$ and all $d$ with $p \nmid d$; see Example \ref{ex:monomial}.
We similarly find in our specific examples that the lower bound is sharp and that ``most'' covers have $a$-number equal to the lower bound
\footnote{When $p=3$ or $p=5$ and $X$ is ordinary,  for any branch locus $S$ and choice of $d_Q$ for $Q \in S$ with $p \nmid d_Q$, subsequent work constructs covers of $X$ whose $a$-number is the lower bound \cite{anumber}. }.
For various $p$, we also provide examples using covers of the elliptic curve with affine equation $y^2=x^3-x$, which has $a_X=1$ when $p\equiv 3\bmod 4$ and $a_X=0$ when $p\equiv 1\bmod 4$.

\subsection{Outline of the Proof}

Without loss of generality, we may assume that $k$ is algebraically closed.  
A key idea in the proof is that the Cartier operator is not defined only on global differentials, but actually is a map of sheaves.  
Let $X$ be a smooth projective curve over $k$. Functorially associated to the finite flat absolute Frobenius map $F:X\rightarrow X$
by Grothendieck's theory of the trace \cite[2.7.36]{Conrad} 
is an $\O_X$-linear map of sheaves 
\[
 V_X:  F_*\Omega^1_X \to \Omega^1_X ;
\]
the Cartier operator considered previously is obtained by taking global sections.
(For the remainder of the paper, we include subscripts to clarify which curve/ring we are working with.)    The advantage of this perspective is that the Cartier operator admits a simple description on stalks, allowing local arguments.  
In particular, the Cartier operator on completed stalks at any $k$-point
is given by 
\begin{equation} \label{eq:localcartier}
 V \left (\sum_i a_i t^i \frac{dt}{t} \right) = \sum_j a_{pj}^{1/p} t^j \frac{dt}{t};
\end{equation}
see, for example \cite[Proposition 2.1]{CaisHida1}.
To relate the kernels of $V_X$ and $V_Y$ on global differentials, we will combine an analysis over the generic point with an analysis at stalks at the points where the cover $\pi : Y \to X$ is ramified.  This strategy allows the use of geometric methods, and allows us to work with general Artin-Schreier covers instead of only covers of $\PP^1$: previous work has focused on curves defined by explicit equations or covers $Y$ of $\PP^1$ where it is possible to find a nice and explicit basis of $H^0(Y,\Omega^1_{Y})$.

For now, we will ignore a few technical issues and sketch the argument.  None of these technical issues arise for $X =\PP^1$, which is a helpful simplification on a first reading.  
Writing $\eta$ for the generic point of $X$, one has an isomorphism
\begin{equation}
\pi_* \Omega^1_{Y,\eta} \simeq \bigoplus_{i=0}^{p-1} \Omega^1_{X,\eta}.\label{eq:genericisom}
\end{equation}
This follows from the fact that the function field $K' = k(Y)$ of $Y$ is an Artin-Schreier extension of the function field $K = k(X)$ given by
$y^p - y = f$ for some $f \in K$. 
This induces an isomorphism
\begin{equation} \label{eq:genericiso}
 (\pi_* \ker V_Y)_\eta \simeq \bigoplus_{i=0}^{p-1} (\ker V_X)_\eta.
\end{equation}
This is Proposition~\ref{prop:isogeneric}; 
using (\ref{eq:genericisom}) to write $\omega\in \pi_* \Omega^1_{Y,\eta}$ as
$\displaystyle \omega = \sum_{i=0}^{p-1} \omega_i y^i$ with $\omega_i \in \Omega^1_{X,\eta}$, the key observation is that
if $V_Y(\omega)=0$ then for all $0 \leq j \leq p-1$, $V_X(\omega_j)$ is determined by $V_X(\omega_i)$ for $j < i \leq p-1$.

Unfortunately, (\ref{eq:genericisom})--(\ref{eq:genericiso}) do not generalize to  isomorphisms of sheaves.  Instead,
there are explicit divisors
$E_i$ (Definition~\ref{defn:collecteddefs}) depending on the ramification of $\pi$ and an isomorphism of $\O_X$-modules
\begin{equation} 
\pi_* \Omega^1_Y \simeq \bigoplus_{i=0}^{p-1} \Omega^1_X(E_i), 
\label{eq:diffsplit}
\end{equation}
as well as an injection (Definition~\ref{defn:varphi})
\begin{equation} 
 \varphi: \pi_* \ker V_Y \into \bigoplus_{i=0}^{p-1} \ker V_X(F_* E_i)
 \label{eq:kersplit}
\end{equation}
inducing \eqref{eq:genericiso} at the generic point.  Here $\ker V_X(F_* E_i):= (\ker V_X )_\eta \cap F_* (\Omega^1_X(E_i))$.  But $\varphi$ is not surjective as a map of sheaves.  The problem is that while (in the generic fiber) $V_X(\omega_j)$ is determined by $\omega_i$ for $i>j$, it is not automatic that the resulting form $\omega_j$ satisfies $\ord_Q(\omega_j) \geq -\ord_Q(E_j)$ when $\ord_Q(\omega_i) \geq -\ord_Q(E_i)$ for $i>j$ and $Q \in S$.  Example~\ref{ex:keyexample} is a key illustration of this problem.

To deal with this issue, we wish to find relations that describe the image of $\varphi$.   Fortunately, this is a purely local problem at the points $S$ where $\pi$ is ramified.  We will identify $\pi_* \ker V_Y$ with a certain sub-sheaf of the target of
(\ref{eq:kersplit})
cut out by linear relations on the coefficients of the power series expansions of elements at points $Q \in S$.  These relations force the corresponding differentials $\omega_i$ on $X$ to be regular.  These relations are expressed in \S\ref{ses:kernel} as maps to skyscraper sheaves supported on $S$, and the kernels of these maps describe the image of $\varphi$; see Theorem~\ref{thm:technicalses}. 

The final step is to extract useful information about the $a$-number of $Y$  (the dimension of the space of global sections of $\pi_* \ker V_Y$) from the short exact sequences resulting from our description of the image of $\varphi$.  The key to doing so is knowledge about the dimension of the kernel of the Cartier operator on $H^0(X,\Omega^1_X(D))$ for various divisors $D$, and related questions about the existence of elements of that space with specified behavior at points in $S$.  In \S\ref{sec:bounds}, we analyze these questions using a theorem of Tango (Fact~\ref{fact:tango}, the main theorem of \cite{tango72}) and obtain bounds on the $a$-number of $Y$ by taking global sections of the exact sequences from Theorem~\ref{thm:technicalses}.


\begin{remark}
    Tango's theorem yields precise results only when the degree of the divisor is sufficiently large.  By exploiting the flexibility of our local analysis---which in particular allows us to arbitrarily increase the degree of certain auxiliary divisors---we may always work in this case; see the proof of Theorem \ref{thm:cleanbounds}.  Given additional information about the size of the dimension of the kernel of the Cartier operator on $H^0(X,\Omega^1_X(D))$ for divisors $D$ of small degree, modest improvements are possible.  
    See \S\ref{sec:unramified} for an example with unramified covers.
\end{remark}

\begin{remark}
    The trivial lower bound of Remark \ref{rmk:trivialbounds}
    follows from 
    the inclusion $\Omega^1_{X}(E_0)\hookrightarrow \pi_*\Omega^1_{Y}$
    that is compatible with the Cartier operator.  This inclusion is a consequence of (\ref{eq:diffsplit}), which is an isomorphism of $\O_X$-modules and is not in general compatible 
    with the Cartier operator.  
    The significant improvements in the bounds come from incorporating information about the Cartier operator to obtain a more refined inclusion (\ref{eq:kersplit}), analyzing the image,  and using Tango's theorem.
\end{remark}

\begin{remark}
As mentioned previously, there are some technical complications to the strategy outlined above.  In the end, these have no effect on the final result.  The issues are:
\begin{itemize}
    \item the short exact sequence 
    \[
 0 \to \ker V_X \to F_* \Omega^1_X \to \Im V_X \to 0
    \]
    is not always split, although it does split when $X=\PP^1$.   In \S\ref{sec:splittings}, we show that we may produce maps which split the sequence over the generic point and introduce poles in a controlled manner.  This is used to define the map $\varphi$ of (\ref{defn:varphi}).
    
    \item  The Artin-Schreier extension of function fields cannot always be described as $y^p-y = f$ where $f$ is regular away from $S$, and $\ord_Q(f) = -d_Q$ for $Q \in S$.  This is possible when $X = \PP^1$ by using the theory of partial fractions.  In \S\ref{sec:asc} we allow $f$ to have a pole at one additional (non-branched) point $Q'$ to ensure the desired property holds for $Q \in S$, and then keep track of this complication throughout the remainder of the argument.  
\end{itemize}
\end{remark}

\begin{remark}
The same arguments, with minor modifications, should yield bounds on the dimension of the kernel of powers of the Cartier operator. We leave that for future work.
\end{remark}

\subsection{Acknowledgments}
We thank Jack Hall, Daniel Litt, Dulip Piyaratne, and Rachel Pries for helpful conversations.  We thank the referee for reading carefully and making many helpful comments.  The second author was partially supported by NSF grant number DMS-1902005.

\section{Producing Splittings} \label{sec:splittings}

Let $k$ be an algebraically closed field of characteristic $p$, and $X$ a smooth projective and connected curve over $k$.  Writing $V_X: F_*\Omega^1_{X/k}\rightarrow \Omega^1_{X/k}$ for the Cartier operator, we are interested in splitting the tautological 
short exact sequence of $\O_X$-modules
\begin{align} \label{eq:cartiersequence}
 0 \to \ker V_X \to F_* \Omega^1_X \to \Im V_X \to 0.
\end{align}

\begin{lem} \label{lem:p1}
 When $X = \PP^1_k$, \eqref{eq:cartiersequence}
is a split exact sequence of sheaves.
\end{lem}

\begin{proof}
Identify the generic fiber of $X$  with $\Spec( k(t))$.  From \eqref{eq:localcartier}, we know that $V_X(\frac{dt}{t}) = \frac{dt}{t}$ and $V_X(t^i \frac{ dt}{t}) = 0$ if $p \nmid i$.  Thus we see that
\[
 (\ker V_X)_\eta = \left\{ \sum_{i=1}^{p-1} h_i(t) t^i \frac{dt}{t} : h_i \in k(t^p) \right\} \quad \text{and} \quad (\Im V_X)_\eta = \Omega^1_{X,\eta}.
\]
An explicit splitting $s:  (\Im V_X)_\eta \to F_*\Omega^1_{X,\eta}$ of \eqref{eq:cartiersequence} over the generic fiber  is given by
\begin{equation} \label{eq:explicitsplitting}
 s\left( \sum_i a_i t^i \frac{dt}{t} \right) = \sum_i a_i^p t^{pi} \frac{dt}{t}.
\end{equation}
For $Q \in \PP^1_k$, a direct calculation shows that for a section $\omega$ to $\Omega^1_X$, if $\ord_Q(\omega) \geq 0$ then $\ord_Q(s (\omega))\geq 0$.  Thus this defines a map of sheaves.
\end{proof}

\begin{remark} \label{rmk:p1}
The corresponding projector $r : F_* \Omega^1_X \to \ker V_X$ is given by
\begin{equation} \label{eq:explicitprojector}
 r\left( \sum_i a_i t^i \frac{dt}{t} \right) =  \sum_{p \nmid i} a_i t^i \frac{dt}{t}.
\end{equation}
\end{remark}

In general, \eqref{eq:cartiersequence} splits over the generic point for any smooth curve $X$, although it is not clear the sequence itself splits.  However, it {\em does} split if we allow the splitting to introduce controlled poles.  

Let $D$ be a divisor on $X$.  
Note that $(F_* \Omega^1_X) (D) = F_* (\Omega^1_X(pD))$, which complicates the relationship between twists and order of vanishing.  To more closely connect twists with order of vanishing, we make the following definition.

\begin{defn} \label{defn:ftwists}
For a subsheaf $\Fscr \subset F_* \Omega^1_X$ and divisor $D$ on $X$ we define a sheaf $\Fscr(F_*D)$ via
\[
 \Fscr(F_* D)(U) := F_* (\Omega^1_X(D))(U) \cap \Fscr_\eta
\]
for open $U \subset X$.
In particular,
\begin{equation}
\ker V_X (F_*D)(U) = \{ \omega \in \Omega^1_{X}(D)(U) : V_X(\omega) = 0 \}.
\end{equation}
\end{defn}

\begin{example} \label{ex:twisting}
It is clear upon taking $\Fscr = F_* \Omega^1_X(D)$ that $(F_* \Omega^1_X)(F_*D)= F_* (\Omega^1_X(D)) $.
Note that $\ker V_X (F_* D)$
consists of differentials $\omega$ that lie in the kernel of $V_X$ and satisfy $\ord_Q(\omega) \geq - \ord_Q(D)$ for all $Q$.  On the other hand, $\ker V_X (D) = \ker V_X \otimes \O_X(D)= \ker V_X (F_* p D) $, which consists of differentials that lie in the kernel of $V_X$ and satisfy $\ord_Q(\omega) \geq -p \ord_Q(D)$.
\end{example}

For any divisor $E = \sum_i n_i P_i \geq 0$, define $\overline{E} := \sum_i \lceil n_i/p \rceil P_i$.

\begin{lem} \label{lem:twisted}
 There is an exact sequence of $\O_X$-modules
 \begin{equation}
  0 \to \ker V_X (F_* E) \to F_*( \Omega^1_X(E)) \overset{V_X} \to \Im V_X (\overline{E}) \to 0.
 \end{equation}
 Furthermore, each term is locally free.
\end{lem}

\begin{proof}
For a closed point $Q$ of $X$ and a section $\omega$ of $F_* (\Omega^1_X(E))$ defined at $Q$, a local calculation shows that $\ord_Q(V_X(\omega)) \geq \lceil \ord_Q(\omega) /p \rceil$, so the right map is well-defined.  The kernel is $\ker V_X (F_* E)$ by definition. As the completion $\O^\wedge_{X,Q}$ is faithfully flat over $\O_{X,Q}$, 
we may check surjectivity on completed stalks. Since $-n_i \leq p \left \lceil \frac{-n_i}{p} \right \rceil$, for $t_Q$ a local uniformizer at $Q$ we have
\[
 V_X\left( \sum_i a_i^p t_Q^{ip} \frac{dt_Q}{t_Q} \right) = \sum a_i t_Q^i \frac{dt_Q}{t_Q}
\]
thanks to (\ref{eq:localcartier}).  Over a smooth curve, to check local freeness it suffices to check the sheaves are torsion-free, which is clear as the sheaves are subsheaves of $F_* \Omega^1_{X,\eta}$.
\end{proof}

We will prove the following:

\begin{prop} \label{prop:abstractsplittings}
Let $S$ be a finite set of points on a smooth projective curve $X$ over $k$ and
\[
0 \to \Fscr_1 \overset{\imath} \to \Fscr_2 \to \Fscr_3 \to 0
\]
an exact sequence of locally free sheaves.  There exists a divisor $D = \sum_i [P_i]$ with the $P_i$ distinct points of $X$ not in $S$ and a morphism $r: \Fscr_2\to \Fscr_1(D)$
such that $r \circ \imath$ is the natural inclusion $\Fscr_1 \to \Fscr_1(D)$.
\end{prop}

The following corollary will be useful in \S\ref{sec:stalks}, especially in Definition~\ref{defn:varphi}.

\begin{cor} \label{cor:splittings}
Let $S$ be a finite set of points of $X$.  Fix an effective divisor $E$ supported on $S$.  There is a divisor $D = \sum_i [P_i]$ with the $P_i$ distinct points of $X$ not in $S$ and a morphism
\[
 r : F_*(\Omega^1_X(E)) \to \ker V_X(F_*(E + pD)) 
\]
such that $r \circ \imath$ is the natural inclusion $\ker V_X(F_* E) \to \ker V_X(F_*(E+pD))$, where $\imath$ is the inclusion $\ker V_X(F_* E) \to F_*( \Omega^1 _X(E))$.
\end{cor}

\begin{proof}
Apply Proposition~\ref{prop:abstractsplittings} to the exact sequence of Lemma~\ref{lem:twisted}.  It is elementary to verify using Definition~\ref{defn:ftwists} that $(\ker V_X(F_* E))(D) = \ker V_X(F_*(E + p D))$.
\end{proof}

The remainder of this section is devoted to proving Proposition~\ref{prop:abstractsplittings}.  The key input is:

\begin{lem} \label{lem:distinctdivisor}
 Let $S$ be a finite set of points of $X$, and $\Fscr$ be a locally free sheaf on $X$.  Then for any divisor $D = \sum_i [P_i]$ with the $P_i$ distinct points of $X$ not in $S$ and $\deg D \gg 0$, we have 
 \[
  H^1(X, \Fscr \otimes \O_X(D)) = 0.
 \]
\end{lem}

\begin{proof}
Pick an ample line bundle $\Lscr = \O_X(D')$.
  By Serre's cohomological criterion for ampleness, we know that there is an $N$ such that
\begin{equation}
 H^1(X, \Fscr \otimes \Lscr^n)=0 \label{eq:ample}
 \end{equation}
for $n \geq N$.
Writing $g$ for the genus of $X$, the Riemann-Roch theorem gives
\[
h^0(X,\O_X(D- N D')) - h^1(X,\O(D-N D')) = \deg(\O(D-ND')) - g + 1,
\]
and when $\deg (\O(D-ND')) > 2g-2$, we have $h^1(X,\O(D-N D')) = 0$ for degree reasons.  Thus when $\deg D \gg 0$ we conclude that
\[
 h^0(X,\O_X(D- N D')) > 0.
\]
Using a global section of $\O_X(D-N D')$, we obtain a short exact sequence
\[
 0 \to \O_X \to \O_X(D'- N D) \to \Gscr \to 0
\]
where $\Gscr$ is a skyscraper sheaf supported on $D'  - ND$.  Tensoring with $\O_X(ND')$ is exact, as is tensoring with the locally free $\Fscr$, so we obtain a short exact sequence
\[
 0 \to \Fscr \otimes \O_X(ND') \to \Fscr \otimes \O_X(D) \to \Gscr' \to 0
\]
where $\Gscr'$ is still a skyscraper sheaf supported on $D - N D'$.  Part of the long exact sequence of cohomology is
$$  H^1(X, \Fscr \otimes \O_X(ND')) \to H^1(X,\Fscr \otimes \O_X(D)) \to H^1 (X, \Gscr')$$
The left term vanishes by \eqref{eq:ample}, and the right vanishes as $\Gscr'$ is a skyscraper sheaf.  Thus $H^1(X,\Fscr \otimes \O_X(D))=0$.
\end{proof}

We now prove  Proposition~\ref{prop:abstractsplittings}.  By looking at stalks we see that the Hom-sheaf
$
 \HOM_{\O_X}(\Fscr_3, \Fscr_1) 
$
is locally free.  Also, notice that for a divisor $D$
\begin{equation}
 \HOM_{\O_X} (\Fscr_3 , \Fscr_1(D))  \simeq \Fscr_3^\vee \otimes \Fscr_1 \otimes \O_X(D) 
 \simeq \HOM_{\O_X}(\Fscr_3 , \Fscr_1)  \otimes \O_X(D).\label{eq:F3iden}
\end{equation}
Applying the functor $\HOM_{\O_X}(\cdot, \Fscr_1(D))$  to the exact sequence in Proposition~\ref{prop:abstractsplittings} and using the assumption that
$\Fscr_3$ is locally free, we obtain an exact sequence
\begin{align*}
 0 \to \HOM_{\O_X}( \Fscr_3, \Fscr_1(D))  \to \HOM_{ \O_X} (\Fscr_2 , \Fscr_1(D)) 
 \to  \HOM_{ \O_X}(\Fscr_1 , \Fscr_1(D) ) \to 0.
\end{align*}
Passing to global sections,
part of the long exact sequence of cohomology is
\begin{align*}
 H^0(X, \HOM_{ \O_X} (\Fscr_2 , \Fscr_1(D)))  \overset{f} \to H^0(X, \HOM_{ \O_X}(\Fscr_1 , \Fscr_1(D) )) \to H^1(X, \HOM_{\O_X}( \Fscr_3, \Fscr_1(D))).
\end{align*}
Applying Lemma~\ref{lem:distinctdivisor} with $\Fscr = \HOM_{\O_X}(\Fscr_3, \Fscr_1)$
and appealing to (\ref{eq:F3iden}), we choose $D = \sum_i [P_i]$ where the distinct $P_i$ avoid $S$ such that $H^1(X, \HOM_{\O_X}(\Fscr_3,\Fscr_1(D))) =0$.  Thus $f$ is surjective, and the desired morphism $r$ is the pre-image of the natural inclusion in $H^0(X, \HOM_{ \O_X}(\Fscr_1 , \Fscr_1(D) ))$.
\qed

\begin{cor} \label{cor:secondsplitting} 
A choice of projector $r : F_* (\Omega^1_X(E)) \to \ker V_X(F_* (E + p D))$ as in Corollary~\ref{cor:splittings} is equivalent to a choice of splitting of \eqref{eq:cartiersequence}
$$s : \Im V_X(\Ebar) \to F_*( \Omega^1_X(E + p D))$$ such that $V_X \circ s$ is the natural inclusion $\Im V_X(\Ebar) \to \Im V_X(\Ebar+D)$.  Furthermore, we may choose $s$ so that for any point $Q \in \sup(E)$, if $\ord_Q(\omega) \geq d$ then 
\begin{equation}
 \ord_Q(s(\omega)) \geq  (d+1)p-1.\label{eq:sectionpolarprop}
\end{equation}
\end{cor}

\begin{proof}
Given $r$, a section of
\[
 0 \to (\ker V_X)_\eta \overset{\imath}\to (F_* \Omega^1_{X,\eta}) \to (\Im V_X)_\eta \to 0
\]
is given by $s(m) = \widetilde{m} - \imath r (\widetilde{m})$, where $\widetilde{m}$ is any preimage of $m \in (\Im V_X)_\eta$.  It is independent of the choice of lift, which allows us to make local calculations using \eqref{eq:localcartier} to check that this recipe defines a map $s: \Im V_X(\overline{E}) \to F_*(\Omega^1_X(E + p D))$.  Conversely, given $s$ we define $r(x) = x - s V_X(x)$ for any section $x$ of $F_* (\Omega^1_X(E))$.  It is straightforward to check that $r \circ \imath$ is the natural inclusion and that the recipes for $r$ given $s$ and $s$ given $r$ are inverse to each other, so that this process really does give an equivalence between a splitting $s$ and a projector $r$ as claimed.

To get the last claim, set
\begin{equation} \label{eq:defne'}
E' := \sum_{Q \in \sup(E)} p (\left \lceil \frac{\ord_Q(E)}{p} \right \rceil-1) +1
\end{equation}
and note that $\overline{E}' = \overline{E}$ and $E' \le E$.
Thanks to Corollary~\ref{cor:splittings} and its proof, we obtain
a morphism $r': F_*(\Omega^1_X(E'))\rightarrow \ker V_X(F_*(E'+pD))$
with the {\em same} $D$ as above.  
Via the correspondence already established,
we obtain a section
 $$s' : \Im V_X(\overline{E}) \to F_*( \Omega^1_X(E' + pD)).$$
Composing with the natural inclusion $F_*( \Omega^1_X(E' + pD)) \into F_* (\Omega^1_X(E+pD))$, we obtain the desired map $s$: the condition on orders of vanishing follows by twisting.  
\end{proof}


\begin{remark}
Let $r : F_*(\Omega^1_X(E)) \to \ker V_X(F_*(E+pD))$ be the projector corresponding to a splitting $s : \Im V_X(\Ebar) \to F_*( \Omega^1_X(E + p D)) $.  No matter the choice of $E$ and $D$, the generic fiber of \eqref{eq:cartiersequence} is split by $r$ and $s$.
\end{remark}

\begin{remark}
 When $X = \PP^1$, the explicit splitting of Lemma~\ref{lem:p1} shows we may take $D=0$.  This is a good simplification for subsequent arguments on a first reading.
\end{remark}

\section{Artin-Schreier Covers and Differential Forms} \label{sec:asc}

Let $\pi: Y \to X$ be a branched Galois cover of
smooth
projective and connected curves over an algebraically closed field $k$ of characteristic $p$, with Galois group $G\simeq \ZZ/p\ZZ$
and branch locus $S\subset X$.
We are mainly interested in the case that $S$ is non-empty, though we do not assume this at the outset.
By Artin-Schreier theory, $K':=k(Y)$ is a degree-$p$ Artin--Schreier extension of $K:=k(X)$;
that is, there exists $f \in K$ such that $K'= K(y)$, where
\begin{equation}
 y^p -  y = f.\label{eq:ASexplicit}   
\end{equation}
We may and do assume a fixed choice of generator $\tau$ of $G$ sends $y$ to $y+1$, and  
we write $g_X$ for the genus of $X$.  We henceforth {\em fix} a closed point $Q' \in X$ with $\ord_{Q'}(f)=0$.

The choice of $f$ is far from unique.  In particular, by replacing
 $y$ with $y + h$ for $h\in K$ we may replace $f$ with $f + h^p - h$ but obtain the same extension.  Nonetheless, it is possible to normalize the choice of $f$ as follows.
 
\begin{defn}\label{def:minimal}
We say $\psi \in K$ is {\em minimal} for $\pi : Y \to X$ (or for $K'/K$) if
  \begin{itemize}
 \item $\ord_Q(\psi) \geq 0$ or $p \nmid \ord_Q(\psi)$ for all $Q \in X$ with $Q\neq Q'$.

 \item  $p | \ord_{Q'}(\psi)$ and $-\ord_{Q'}(\psi) \leq p (2g_X-2)$.
\end{itemize}
  \end{defn}
  
\begin{lem} \label{lem:minimality}
The Artin-Schreier extension $K'/K$ may be described as in $(\ref{eq:ASexplicit})$
with $f$ minimal.        
 \end{lem}
 
\begin{proof}
      Let $f\in K$ be as in (\ref{eq:ASexplicit}),
and suppose $f$ has a pole of order $p \cdot d$ at $Q \neq Q'$.  
By Riemann--Roch, 
 there exists a function with a pole of order $m$ at $Q$ and a pole of order $n$ at $Q'$ provided $m+n > 2g_X-2$.
We may therefore
choose $h\in K$ to have a pole of order $d$ at $Q$ and no poles except possibly at $Q'$,
and satisfy $\ord_Q(f+h^p-h) > \ord_Q(f)$ and 
$\ord_{Q'}(h^p-h) \geq - p(2 g_X-2)$.  Replacing $f$ with $f+h^p-h$ and repeating this procedure, we may thereby arrange for $f$ to be minimal.
\end{proof}

We now {\em fix} a choice of minimal $f\in K$ giving the Artin--Schreier extension
$K'/K$ as in (\ref{eq:ASexplicit}).

 \begin{remark} \label{remark:p1}
 When $X = \PP^1$, it is easy to arrange for $f$ to have poles only at the branch locus since there are functions on $\PP^1$ with a single simple pole.  In this case, our analysis below at $Q'$ is unnecessary: this is a helpful simplification on a first reading.  This reduction is not possible in general: see \cite[\S7]{shabat}, especially the second example after Proposition 49.
\end{remark}

We now fix some notation relating to the cover $\pi : Y \to X$
and our fixed choice of $Q'$ and minimal $f$ giving the corresponding extension
of function fields as in (\ref{eq:ASexplicit}).
 
\begin{defn} \label{defn:collecteddefs}
With notation as above:
\begin{itemize}
    \item Let $S':=S\cup \{Q'\}$.
    \item For $Q\in X$ define $d_Q:=\max\{0,-\ord_Q(f)\}$.
    \item  For $0\leq i \leq p-1$ and $Q\in X$ define
    \[
    n_{Q,i} := \begin{cases} \left \lceil \displaystyle\frac{(p-1-i) d_Q}{p} \right \rceil & \text{if}\ Q\neq Q' \\
    (p-1-i) d_{Q'} & \text{if}\ Q=Q'
    \end{cases}
    \]
    and set
    \[
      E_i := \sum_{Q \in S'}  n_{Q,i} [Q]  \quad \text{and} \quad \overline{E}_i := \sum_{Q \in S'} \lceil n_{Q,i} /p \rceil [Q].
    \]
\end{itemize}
\end{defn}

\begin{lem} \label{lem:ramification}
The map $\pi:Y\rightarrow X$ is ramified over $Q\in X$
if and only if $d_Q>0$ and $p\nmid d_Q$,
in which case $d_Q$
 is the unique break in the lower-numbering ramification filtration of $G$ above $Q$. 
\end{lem}

\begin{proof}
This is well known; see for example \cite[Proposition 3.7.8]{stichtenoth}.
\end{proof}

\begin{remark}\label{rem:nQzero}
It follows from Lemma \ref{lem:ramification} and the fact that the fixed $f$ is minimal that if $Q\not\in S'$ then $d_{Q}=0$ and hence also $n_{Q,i}=0$ for all $i$.
\end{remark}

We now investigate when a meromorphic differential on $Y$ is regular.

\begin{lem} \label{lem:regular}
Let $\omega\in (\pi_* \Omega^1_Y)_\eta$ be a meromorphic differential
on $Y$, and write $\omega=\sum_i \omega_i y^i$
with $\omega_i\in \Omega^1_{X,\eta}$ using the identification $(\ref{eq:genericisom})$.
Let $P\in Y$ with $\pi(P)=Q \neq Q'$.
Then $\omega$ is regular at $P$ provided $\ord_Q(\omega_i) \ge -n_{Q,i}$ 
for all $0 \leq i \leq p-1$.
\end{lem}

\begin{proof}
This is also probably known, but we were unable to locate a reference in the degree of generality we need: for example, Boseck only treats the case that $X =\PP^1$ \cite[Satz 15]{boseck}.  We provide a proof for the convenience of the reader.

Let $P\in Y$ and set $Q=\pi(P)$.
Suppose first that $Q\in S$.  Let $t_Q$ be a uniformizer of $\O_{X,Q}$, and note that the fraction field of $\O_{Y,P}$ can be obtained from the fraction field of $\O_{X,Q}$ by adjoining a root of the polynomial $y^p-y = f$ where $f$ has a  pole of order $d_Q$ at $Q$
by Lemma \ref{lem:ramification}; in particular, $\ord_P(y) = - d_Q$. 

As $p\nmid d_Q$, we may choose positive integers $a$ and $b$ with $1 = ap - b d_Q$, so that $u = t_Q^a y^b$ is a uniformizer of $\O_{Y,P}$.  A direct computation shows that $du = u^{-(p-1) (d_Q+1)} \beta dt_Q$ where $\beta \in \O_{Y,P}^{\times}$. Hence $\ord_P(dt_Q) = (p-1) (d_Q + 1)$, and $\ord_P(\omega_i) = (p-1) (d_Q + 1) + p \ord_Q(\omega_i)$.  
We conclude that 
\begin{equation}
 \ord_P( \omega_i y^i) = p \ord_Q(\omega_i) + (p-1) (d_Q + 1) - i d_Q.\label{eq:valomega}
\end{equation}
This is non-negative precisely when $\ord_Q(\omega_i) \geq - n_{Q,i}$.  
As $p\nmid d_Q$ and $0 \le i \le p-1$, the $p$ integers in (\ref{eq:valomega})
are all distinct modulo $p$, and hence distinct, so 
\begin{equation}
    \ord_P(\omega) = \min_{0\le i\le p-1} \ord_P(\omega_iy^i)
\end{equation}
and $\omega$ is regular at $P$ precisely when $\ord_Q(\omega_i) \geq - n_{Q,i}$ for all $i$.

Now suppose $Q\not\in S'$, so that $\pi$ is \'etale over $Q$.
There are then $p$ points $P=P_0, \ldots , P_{p-1}$ over $Q$, and $\O_{Y,P_j} \simeq \O_{X,Q}$ (and likewise for differentials) for all $j$.
  Under these identifications, the function $y$ corresponds to regular functions $h, h+1,\ldots ,h+(p-1)$ in each of the $\O_{Y,P_j} \simeq \O_{X,Q}$.  Thus $\omega = \sum_i \omega_i y^i$ is regular at one (and hence every) point of $Y$ over $Q$ provided
\[
 \ord_Q\left( \sum_i \omega_i (h+j)^i\right) \geq 0
\]
for $0 \leq j \leq p-1$.  Define $\lambda_j = \sum_i \omega_i (h+j)^i \in \Omega^1_{X,Q}$, depending on $\omega$.  The $\lambda_j$ and $\omega_i$ are related by the Vandermonde matrix $M = ((h+j)^i)_{0 \leq i, j \leq p-1}$:
\[
\begin{pmatrix}
 (h+0)^0 & (h+0)^1 & \ldots & (h+0)^{p-1} \\ 
 (h+1)^0 & (h+1)^1 & \ldots & (h+1)^{p-1} \\
 \vdots & \vdots & \vdots & \vdots \\
 (h+(p-1))^0 & (h+(p-1))^1 & \ldots & (h+(p-1))^{p-1}
 \end{pmatrix} \cdot
 \begin{pmatrix} \omega_1 \\
   \omega_2 \\
   \vdots \\
   \omega_{p-1}
  \end{pmatrix}
= \begin{pmatrix}
   \lambda_1 \\
   \lambda_2 \\
   \vdots \\
   \lambda_{p-1}
  \end{pmatrix}.
\]
The determinant of $M$ is well-known to be
\[
 \prod_{\substack{1 \leq i,j \leq p \\ i \neq j }} ( (h+i) - (h+j)) = \prod_{i \neq j} (i-j) \in \FF_p^\times,
\]
and it follows that all $\omega_i$ are regular at $Q$ if and only if
all $\lambda_j$ are. This completes the proof.
\end{proof}

\begin{remark}
If $\pi(P)=Q'$, then the analogous condition that a meromorphic differential
on $Y$ be regular at $P$
is considerably more complicated as $\ord_{P}(y) < 0$, so the function $h$ in the proof of Lemma \ref{lem:regular} has a pole at $Q$ and the entries of $M$ are no longer in $\O_{X,Q}$.  One simple case will be analyzed in the proof of Proposition~\ref{prop:ws}, while a more complete analysis is deferred to Lemma~\ref{lem:fixQ}.
\end{remark}

We next study a filtration on $\pi_*\Omega^1_Y$ arising from the $G\simeq \ZZ/p \ZZ$ action on $Y$.  
We fix a generator $\tau \in G$ and obtain a map of sheaves $\Omega^1_Y \to \tau_* \Omega^1_Y$.  After pushing forward using the equivariant $\pi$, we obtain a map $\tau: \pi_* \Omega^1_Y \to \pi_* \Omega^1_Y$.  This induces a map at the stalk at the generic point.  

For technical reasons, instead of $\pi_*\Omega^1_Y$,
we will work with a slightly larger sheaf of meromorphic differentials
defined as follows.
Recall that $d_{Q'}$ is the order of the pole of $f$ above $Q'$ and that $p | d_{Q'}$.  Furthermore, $y$ has a pole of order $d_{Q'}$ at any point above $Q'$.  Recall we defined $n_{Q',i} = (p-1-i) d_{Q'}$.

\begin{defn} \label{defn:F0}
Define the sheaf $\Fscr_0 \subset \pi_* \Omega^1_{Y,\eta}$ on $X$ which on an open set $U$ has sections
\[
 \left \{ \omega = \sum_i \omega_i y^i :  \omega \text{ is regular above all } Q \in U\setminus\{Q'\}\ \text{and } \ord_{Q'} \omega_i \geq - n_{Q',i}\ \text{for all}\ i\ \text{if}\ Q' \in U \right\}.
\]
\end{defn}

\begin{lem}\label{lem:F0props}
The natural $G$-action on $\pi_*\Omega^1_{Y,\eta}$ preserves
 $\Fscr_0$, and $\pi_* \Omega^1_Y$ is a subsheaf of $\Fscr_0$.
\end{lem}

\begin{proof}
As the first assertion is clear, 
it suffices to show that any meromorphic differential $\omega = \sum \omega_i y^i$ 
on $Y$ that is regular above $Q'$ has $\ord_{Q'} (\omega_i) \geq - (p-1-i) d_{Q'}$.  We will do so by descending induction on $i$.  Observe that
\[
 (\tau-1)^{p-1} \omega = (p-1)! \omega_{p-1}
\]
which must be regular above $Q'$: this happens only if $\ord_{Q'} (\omega_{p-1}) \geq 0$.  To deal with $i=n$ for the inductive step, note that
\begin{equation}
 (\tau-1)^{n} \omega = n! \omega_{n} + f_{n+1}(y) \omega_{n+1}  + \ldots + f_{p-1}(y) \omega_{p-1}\label{eq:regaftertau} 
\end{equation}
where $f_j(y)$ is a polynomial of degree at most $j-n$ of $y$.  
As in the proof of Lemma \ref{lem:regular},
let $h, h+1,h+2,\ldots , h+(p-1) \in \O_{X,Q'}$
correspond to the image of $y$ in the local rings $\O_{Y,P'}$
under the identifications $\O_{Y,P'}\simeq \O_{X,Q'}$ for all $P'$
lying over $Q'$.
Note that $\ord_{Q'}(h) = \frac{-d_{Q'}}{p} \geq -d_{Q'}$.  
Together with (\ref{eq:regaftertau}), the assumption that
$\omega$ is regular above $Q'$ gives
\[
 \ord_{Q'}\left( n! \omega_n + f_{n+1}(h) \omega_{n+1} + \ldots f_{p-1}(h) \omega_{p-1} \right) \geq 0.
\]
Furthermore, $\ord_{Q'}(f_{j}(h) \omega_{j}) \geq (j-n) (-d_Q) - (p-1-j) d_{Q'} = - (p-1-n) d_{Q'}$ by our inductive hypothesis, and it follows that  $\ord_{Q'} (\omega_n )\geq - (p-1-n) d_{Q'}$.
\end{proof}

Thanks to Lemma \ref{lem:F0props}, our fixed generator
$\tau\in G$ induces a map $\tau:\Fscr_0\rightarrow \Fscr_0$
that is compatible with the canonical $G$-action on $\pi_*\Omega^1_Y$.
We use this to define a filtration on $\Fscr_0$:

\begin{defn} \label{defn:wi}
For $-1\le i\le p-1$ let $W_i :=  \ker( (\tau-1)^{i+1} : \Fscr_0 \to \Fscr_0) \subset  \Fscr_0$.
\end{defn}

Note that $W_{-1} =  0$ and $W_{p-1} = \Fscr_0$. 

\begin{prop} \label{prop:ws}
For $0 \leq i \leq p-1$, we have split exact sequences of $\O_X$-modules
\begin{equation}
 0 \to W_{i-1} \to W_i \to  \Omega^1_X(E_i) \to 0\label{eq:Wfilexact}    
\end{equation}
with $E_i$ as in Definition~$\ref{defn:collecteddefs}$. A splitting is given by sending a section $\omega$ of  $\Omega^1_X(E_i)$ to $\omega y^i$.
\end{prop}

\begin{proof}
We compute that $(\tau-1)$ reduces the degree in $y$:
\[
 (\tau-1)\omega_j y^j = \omega_j ((y+1)^j - y^j) = \omega_j ( j y^{j-1} + \ldots ) .
\]
It follows that $\displaystyle (W_i)_\eta = \bigoplus_{j=0}^i (\Omega_X^1)_\eta y^j$. 
We define a map $\psi_i : W_i \to \Omega^1_X(E_i)$ by the formula
\[
\psi_i(\sum_{j=0}^{i} \omega_j y^j) := \omega_i.
\]
To check $\psi_i$ is well-defined, 
given any element $\displaystyle \omega = \sum_{j=0}^{i} \omega_j y^j$ in the stalk of $W_i$ at $Q\in X$, it suffices to check that $\ord_Q(\omega_i) \geq - \ord_Q(E_i)$.  Recall the definition of $E_i$ from Definition~\ref{defn:collecteddefs} and that $\ord_Q(E_i) = n_{Q,i}$.
 For $Q \neq Q'$, the claim follows immediately from Lemma~\ref{lem:regular}, while for $Q= Q'$ it follows from the definitions; for the $i=0$ case use Definition \ref{defn:F0}. One checks easily that
 the kernel of $\psi_i$ is $W_{i-1}$, and that the map sending a section $\omega$ of $\Omega^1_X(E_i)$ to the section $\omega y^i$ of $W_i$ indeed provides a splitting.
\end{proof}

\begin{remark}
The filtration $W_i$ on $\Fscr_0$ is our replacement for the (perhaps more natural)
filtration $V_i:=\ker((\tau-1)^i : \pi_*\Omega^1_Y\rightarrow \pi_*\Omega^1_Y)$
on $\pi_*\Omega^1_Y$.  Unfortunately, the corresponding exact sequence
\[
0 \to V_{i-1} \to V_i \to \Omega^1_X(\widetilde{E}_i) \to 0
\]
with  $\widetilde{E}_i := \sum_{Q \in S}  n_{Q,i} [Q] $ is {\em not} split:  
the natural splitting at the generic point given by $\omega \mapsto \omega y^i$ does not extend to a map of sheaves on $X$ as $\omega y^i$ is not regular above $Q'$ because $y$ has a pole above $Q'$.  It is precisely for this reason that we work instead
with the sheaf $\Fscr_0$ which has modified behavior above $Q'$.
\end{remark}

 Let $\gr^\bullet \Fscr_0$ denote the associated graded sheaf for the filtration $\{W_i\}$.  

\begin{cor} \label{cor:omega}
There is an isomorphism of $\O_X$-modules  
\begin{equation}
\displaystyle \Fscr_0 \simeq \gr^\bullet \Fscr_0 \simeq \bigoplus_{i=0}^{p-1} \Omega^1_X(E_i).\label{eq:gradedisom}
\end{equation}
\end{cor}

\begin{proof}
  The isomorphism is provided by the 
  splittings of the sequences (\ref{eq:Wfilexact}) of Proposition \ref{prop:ws}.
\end{proof}

\begin{remark}
The isomorphism of Corollary \ref{cor:omega} is only as $\O_X$-modules: it is not compatible with the Cartier operator.
\end{remark}

\section{The Cartier Operator on Stalks} \label{sec:stalks}

We keep the notation and assumptions of \S\ref{sec:asc}.
In particular, $\pi:Y\rightarrow X$ is a branched Galois cover of
smooth, projective, and connected curves over $k=\overline{k}$
with group $G\simeq \ZZ/p\ZZ$.  Recall that 
we have fixed a point $Q'$ of $X$ over which $\pi$ is unramified, 
and that the degree-$p$ Artin--Schreier extension 
$K':=k(Y)$ of $K=k(X)$ is the splitting field of $y^p-y=f$
with $f\in K$ {\em minimal} in the sense of Definition \ref{def:minimal}.
Writing simply $F$ for the absolute Frobenius morphism, we denote by
$V_Y : F_* \Omega^1_Y \to \Omega^1_Y$ and $V_X : F_* \Omega^1_X \to \Omega^1_X$ the Cartier operators on $Y$ and $X$ respectively.  In this section, we will analyze these  
maps at the generic points $\eta'$ and $\eta$ of $Y$ and $X$, respectively,
and will construct an isomorphism
\[
(\ker V_Y)_{\eta'} \simeq (\pi_* \ker V_Y)_\eta \simeq \bigoplus_{i=0}^{p-1} (\ker V_X)_\eta
\]
which we will show gives rise to 
an inclusion of sheaves
$\varphi:\pi_*\ker V_Y\hookrightarrow \bigoplus_{i=0}^{p-1} \ker V_X(F_*(E_i+pD_i))$
for certain auxiliary divisors $D_i$ on $X$ that we are able to control.
In the next section, we will analyze the image of this map.

\begin{lem} \label{lem:cartierformula}
For $\omega \in \Omega^1_{Y,\eta'}$, write $\displaystyle \omega = \sum_{i=0}^{p-1} \omega_i y^i$ with $\omega_i \in \Omega^1_{X,\eta}$.  Then
\[
 V_Y(\omega) = \sum_{j=0}^{p-1}\left( \sum_{i=j}^{p-1} V_X\left(\binom{i}{j} \omega_i (-f)^{i-j} \right)\right) y^j.
\] 
\end{lem}

\begin{proof}
Using the relation $y = y^p - f$ and the fact that the Cartier operator is $1/p$-linear, we compute
\begin{align*}
 V_Y( \omega_i y^i) &= V_Y\left( \omega_i \left( y^p - f \right)^i  \right) \\
  &= V_Y\left( \sum_{j=0}^i \omega_i \binom{i}{j} y^{p j} (-f)^{i-j}   \right) \\
 & = \sum_{j=0}^i V_X \left( \omega_i\binom{i}{j} (-f)^{i-j} \right) y^{j},
\end{align*}
and the result follows by collecting $y^j$ terms.
\end{proof}

\begin{cor} \label{cor:determined}
Let $\omega_0,\ldots, \omega_{p-1} \in \Omega^1_{X,\eta}$ and set $\omega:=\sum_{i=0}^{p-1} \omega_i y^i$.  Then $V_Y(\omega)=0$ if and only if
\begin{equation}
  V_X( \omega_j ) = -\sum_{i=j+1}^{p-1} V_X\left(\binom{i}{j} \omega_i (-f)^{i-j} \right) \qedhere \label{eq:VXrelations}
  \end{equation}
for $0\le j\le p-1$.  In particular, if $V_Y(\omega)=0$ then 
$V_X(\omega_{p-1})=0$ and $V_X(\omega_j)$ is determined via \eqref{eq:VXrelations}
by $\omega_i$ for $i > j$.
\end{cor}

\begin{proof}
 Clear from Lemma~\ref{lem:cartierformula}
\end{proof}

Recall the notation of Definition~\ref{defn:collecteddefs}.  
Using Corollaries~\ref{cor:splittings} and \ref{cor:secondsplitting} with $S$ enlarged to include $Q'$ and the zeroes of $f$, for each $0 \leq i \leq p-1$ pick a divisor $D_i = \sum_j [P_{i,j}]$ consisting of distinct points $P_{i,j}$ of $X$ where $f$ has neither pole nor zero (so $\pi$ is unramified over $P_{i,j}$ by Lemma~\ref{lem:ramification}) and maps 
\begin{equation}
    r_i : F_* (\Omega^1_X(E_i)) \to \ker V_X(F_*(E_i + p D_i)), \quad s_i : \Im V_X(\Ebar_i) \to F_* (\Omega^1_X(E_i+p D_i)).\label{eq:projectormaps}
\end{equation}
These induce corresponding maps on stalks at the generic point of $X$,
which we again denote simply by $r_i$ and $s_i$, respectively.
Note that, by the very construction of these maps, any $\omega_i\in \Omega^1_{X,\eta}$ is
determined by $V_X(\omega_i)$ and $\nu_i:=r_i(\omega_i)$ via
\begin{equation} \label{eq:omegaj}
\omega _i = \nu_i + s_i(V_X(\omega_i)).    
\end{equation}


\begin{defn} \label{defn:varphi}
Given the fixed choices of $r_i$ and $D_i$ above, we define a map of $\O_X$-modules
\begin{equation} \label{eq:varphi}
 \varphi: \pi_* \ker V_Y \into F_* \Fscr_0  \simeq F_* \gr^\bullet \Fscr_0 \simeq \bigoplus_{i=0}^{p-1} F_* (\Omega^1_X(E_i)) \overset{\oplus r_i} \to \bigoplus_{i=0}^{p-1} \ker V_X(F_*(E_i +p D_i))
\end{equation}
with the first inclusion coming from Lemma~\ref{lem:F0props} and the middle isomorphisms coming from Corollary~\ref{cor:omega}. 

Let $\varphi_\eta$ denote the induced map on stalks at the generic point of $X$. 
\end{defn}

\begin{prop} \label{prop:isogeneric}
The map $\varphi_\eta$ is an isomorphism of $k(X)$-vector spaces. 
 For $\omega =\displaystyle \sum_{i=0}^{p-1} \omega_i y^i\in (\pi_* \ker V_Y)_\eta$ we have
\[
\varphi_\eta(\omega) = ( r_0(\omega_0), r_1(\omega_1), \ldots , r_{p-1}(\omega_{p-1})).
\]
\end{prop}

\begin{proof}
    Let $\nu:=(\nu_0,\ldots,\nu_{p-1}) \in \bigoplus_{i=0}^{p-1} (\ker V_X)_{\eta}$ be arbitrary,
    and define $\omega_{0},\ldots,\omega_{p-1}\in \Omega^1_{X,\eta}$
    as follows.  Set $\omega_{p-1}:=\nu_{p-1}$, and if $\omega_i$ has been defined
    for $i>j$, then define 
    \begin{equation} \label{eq:reconstructioneq}
    \omega_j := \nu_j + s_j \left(-\sum_{i=j+1}^{p-1} V_X\left(\binom{i}{j} \omega_i (-f)^{i-j} \right)\right).
\end{equation}
  By construction, since $V_X \circ s_j$ is the natural inclusion the differentials $\omega_j$ satisfy the relation \eqref{eq:VXrelations} for all $j$, so that
$\omega:=\sum_{i=0}^{p-1} \omega_i y^i \in (\pi_*\Omega^1_Y)_{\eta}$ lies in 
$(\pi_*\ker V_Y)_{\eta}$ thanks to Corollary \ref{cor:determined}.
One checks easily that the resulting map $\nu\mapsto \omega$
is inverse to $\varphi_{\eta}$.
\end{proof}

\begin{remark} \label{rmk:omegav}
Notice that the proof of Proposition \ref{prop:isogeneric}
shows that for fixed $j$, $\omega_j$ depends only on $\nu_i$ for $j \geq i$,
and in particular that $\omega_i=0$ for $i\ge j$ if and only if $\nu_i=0$ for $i\ge j$.
\end{remark}

Unfortunately, (\ref{eq:varphi}) is not itself an isomorphism:

\begin{example} \label{ex:keyexample}
Let $p=5$, $X = \PP^1$, and consider the Artin-Schreier covers
$Y_i\rightarrow X$ for $i=1,2$ given by $y^p-y = f_i$ where
$f_1(t) = t^{-3}$, and $f_2(t) = t^{-3} + t^{-2}$; each of these covers
is ramified only over $Q:=0$ and has  $d_Q = 3$.
  The $a$-number of $Y_1$ is $4$, while the $a$-number of $Y_2$
  is $3$.
Using the simple projectors of Remark~\ref{rmk:p1}, which allow us to take $D_i =0$ for all $0 \leq i \leq p-1$, we obtain for $Y=Y_1,Y_2$ a map
\[
\varphi : \pi_* \ker V_Y \to \bigoplus_{i=0}^{4} \ker V_{\PP^1}(F_*( n_{Q,i} [Q])).
\]
The induced map on global sections is injective (as $\varphi_\eta$ is an isomorphism), but need not be an isomorphism in general.  
Indeed, given an element $(\nu_0,\ldots,\nu_4) \in \bigoplus_i H^0(\PP^1,\ker V_{\PP^1}(F_* (n_{Q,i} [Q])))$, let us  investigate when it lies in the image of $\varphi$.

 Recalling Definition~\ref{defn:collecteddefs}, we calculate
\[
 n_{Q,i} = \begin{cases}0 & i= 4\\
 1 & i=3\\
 2 & i = 2\\
 2 & i =1\\
 3 & i =0 \end{cases} \quad \quad \text{and} \quad \quad
 \dim_{\FF_p} H^0(\PP^1,\ker V_{\PP^1}( F_*(n_{Q,i}[Q]))) =
 \begin{cases}
  0 & i = 4\\
  0 & i=3\\
  1 & i =2 \\
  1 & i =1 \\
  2 & i =0
 \end{cases}
\]
so the elements of  $\bigoplus_i H^0(\PP^1,\ker V_{\PP^1}(F_*( n_{Q,i} [Q])))$ are easy to describe.  The space is four dimensional, with basis 
\[
 \nu_{0,3} = (t^{-3} dt,0,0,0,0),\   \nu_{0,2} = (t^{-2} dt,0,0,0,0), \  \nu_{1,2} = (0,t^{-2} dt,0,0,0),\ \nu_{2,2} = (0,0,t^{-2} dt,0,0).
\]
Using \eqref{eq:reconstructioneq}, 
we may compute the (unique) preimage
under $\varphi_{\eta}$ of these differentials
for each of the covers under consideration.
For $Y_1$, we find
\[
 \varphi^{-1}_\eta(\nu_{0,3}) = t^{-3} dt, \quad  \varphi^{-1}_\eta(\nu_{0,2})  = t^{-2}dt, \quad  \varphi^{-1}_\eta(\nu_{1,2}) = t^{-2} dt \cdot y, \quad \varphi^{-1}_\eta(\nu_{2,2}) = t^{-2} dt \cdot y^2.
\]
On the other hand, for $Y_2$ we have
\[
\varphi^{-1}_\eta(\nu_{0,3}) = t^{-3} dt, \quad  \varphi^{-1}_\eta(\nu_{0,2})  = t^{-2}dt, \quad  \varphi^{-1}_\eta(\nu_{1,2}) = t^{-2} dt \cdot y, \quad \varphi^{-1}_\eta(\nu_{2,2}) = -t^{-6} dt + t^{-2} dt \cdot y^2.
\]

The computation of $\varphi^{-1}_\eta( \nu_{2,2})$
is special, as it depends on the exact choice of Artin-Schreier equation $y^p- y =f$.  With $f=f_1 = t^{-3}$, we have
\[
\varphi^{-1}_\eta(\nu_{2,2}) = t^{-2} dt y^2 
\]
because there are no terms of the form $t^i dt$ with $i \equiv -1 \pmod{5}$ appearing in $V_{\PP^1}(\omega_0)$ when using Corollary~\ref{cor:determined}.  On the other hand, for $f=f_2 = t^{-3}  + t^{-2}$ we have
\[
 \varphi^{-1}_\eta(\nu_{2,2}) = -t^{-6} dt + t^{-2} dt y^2 
\]
since $V_{\PP^1}(\omega_0) = -  V_{\PP^1}(t^{-2} ( t^{-6} + 2 t^{-5} + t^{-4})dt)$.
This is not an element of $H^0(Y,\Omega^1_Y)$ because of the $t^{-6} dt$ term.

All of the other differentials showing up are regular. Thus, 
$Y_1$ has $a$-number $4$, while $Y_2$ has $a$-number $3$.  
This behavior illustrates why the $a$-number of the cover cannot depend only on the $d_Q$ and must incorporate finer information (in this case, expressed as whether certain coefficients of powers of $f$ are non-zero).

This example shows that to ensure the regularity of $\varphi_\eta^{-1} (\nu_i)$, the coefficients of the $\nu_i$ need to satisfy certain relations (in this case, the coefficient of $t^{-2} dt$ in $\nu_2$ must be zero).  These relations describe the image of $\varphi$, and are the subject of the next section.  This example will be reinterpreted in that context in Example~\ref{ex:followup}.
\end{example}

\section{Short Exact Sequences and the Kernel}\label{ses:kernel}
 
 Continuing with the conventions and notation of \S\ref{sec:asc}--\ref{sec:stalks},
 and motivated by Example \ref{ex:keyexample}, the goal of this section is to 
 describe the image of the map (\ref{eq:varphi}) by means of a collection 
 of linear relations on the coefficients of local expansions
 of meromorphic differentials at a fixed set of closed points containing the branch
 locus $S$.
 These linear relations will be encoded 
 via an ascending filtration by subsheaves
\begin{equation}
    0 \subset \Gscr_{-1} \subset \Gscr_0 \subset \ldots \subset \Gscr_{p-1} \subset \Gscr_p = \bigoplus_{i=0}^{p-1} \ker V_X(F_*(E_i +p D_i)),\label{eq:Gjfilt}
\end{equation}
with $\Gscr_{-1}=\Im(\varphi)$ and whose successive gradeds are identified with explicit skyscraper sheaves;
see Theorem \ref{thm:technicalses} for a precise statement.

For most of the argument, we will work with the sheaf $\Fscr_0$ of Definition \ref{defn:F0} and variants thereof  having modified behavior at the auxiliary point $Q'$.  Only at the very end (see Lemma~\ref{lem:fixQ}) will we properly account for the modified behavior at $Q'$ in order to recover  
the desired $\Im(\varphi)\simeq \pi_* \ker V_Y$, rather than the (generally larger) $\pi_*\ker(V_Y|_{\Fscr_0})$.
At a first reading, it could be useful to assume that $X = \PP^1$ to remove the need for $Q'$ and the divisors $D_i$ coming from Corollary~\ref{cor:splittings}.

\subsection{Filtration and Skyscraper Sheaves}
The first step is to define the filtration (\ref{eq:Gjfilt}).
Recall we have fixed projectors $r_i$ as in (\ref{eq:projectormaps}) and used them to define $\varphi$ in Definition~\ref{defn:varphi}.  
Recall the notation of Definition \ref{defn:collecteddefs}, noting Remark \ref{rem:nQzero}, and
for $Q\in X$ choose a local uniformizer $t_Q$ at $Q$.

If $A$ is an abelian group, then by a slight abuse of notation we will again write
$A$ to denote the constant sheaf associated to $A$.
For any $k$-valued point $P: \Spec k \rightarrow X$,
we write $P_*A$ for the pushforward of (the constant sheaf) $A$
along $P$; it is the skyscraper sheaf on $X$
supported at $P$ with stalk $(P_*A)_P=A$.  Likewise, if $h:A\rightarrow B$
is any homomorphism of abelian groups, we write $P_*h : P_* A\rightarrow P_*B$
for the induced morphism of sheaves on $X$.

\begin{defn}\label{defn:gjdef}
Set $S_j' := S' \cup \sup(D_j)$, 
and define $g_{j+1}:=\rho_j\circ \iota$ where
 $\iota$ is the inclusion
 \begin{equation*}
     \xymatrix@C=17pt{
         \iota:\displaystyle\bigoplus_{i=0}^{p-1} \ker V_X(F_*(E_i +p D_i)) \ar@{^{(}->}[r] &  
         \displaystyle\bigoplus_{i=0}^{p-1} (\ker V_X)_\eta \ar[r]_-{\simeq}^-{\varphi_{\eta}^{-1}} & (\pi_* \ker V_Y)_\eta 
         \ar@{^{(}->}[r] & 
     F_*(\pi_*\Omega^1_{Y})_{\eta} \ar[r]_-{\simeq}^-{(\ref{eq:genericisom})} & \displaystyle\bigoplus_{i=0}^{p-1} F_* \Omega^1_{X,\eta}
        }
    \end{equation*}    
and $\rho_j$ is the composite
\begin{equation*} 
\xymatrix{
         \rho_j: \displaystyle\bigoplus_{i=0}^{p-1} F_* \Omega^1_{X,\eta} \ar[r]^-{\pi_j} &
         F_* \Omega^1_{X,\eta} \ar[r]^-{\sigma_j} & 
         F_* \displaystyle\bigoplus_{Q \in S_j'} Q_*\left( \Omega^1_{X,Q}\left[\frac{1}{t_Q}\right] / t_Q^{-n_{Q,j}} \Omega^1_{X,Q} \right)
     }.
 \end{equation*}
 Here, the first and third maps in the definition of $\iota$ are the natural inclusions,
 the map $\pi_j$ is projection onto the $j$-th summand of the direct sum, and 
 $\sigma_j$ is induced by 
 the identifications $\Omega^1_{X,\eta} = \Omega^1_{X,Q}[1/t_Q]$ and the canonical quotient maps.
\end{defn}

If $\nu:=(\nu_0,\ldots, \nu_{p-1})$ is a section of $\displaystyle\bigoplus_{i=0}^{p-1} \ker V_X(F_*(E_i+pD_i))$ with 
$\iota(\nu)=(\omega_0,\ldots, \omega_{p-1})$
then $r_i(\omega_i)=\nu_i$ and writing 
$\displaystyle \omega := \sum_{i=0}^{p-1} \omega_i y^i \in (\pi_* \ker V_Y)_{\eta}$ we have $\varphi_{\eta}(\omega)=\nu$
and $g_{j+1}(\nu)= \sigma_j(\omega_j)$.  

\begin{defn} \label{defn:gs}
For $0 \leq j \leq p$, let $\Gscr_j$ be the $\O_X$-submodule
$$\displaystyle \Gscr_j \subset \bigoplus_{i=0}^{p-1} \ker V_X(F_*(E_i +pD_i))$$ 
whose sections are precisely
those sections $\nu=(\nu_0,\ldots,\nu_{p-1})$ 
with 
$\iota(\nu)=(\omega_0,\ldots,\omega_{p-1})$ such that
$\omega_i$ is a section of $F_* (\Omega^1_X(E_i))$ whenever $j\le i \le p-1$.
Define $\Gscr_{-1} = \Im(\varphi)$. 
\end{defn}

\begin{lem} \label{lem:filtered}
Suppose that $0\le j \le p-1$.  Then
\begin{enumerate}
\item  $\Gscr_j = \ker(g_{j+1} |_{\Gscr_{j+1}})$. \label{it:kernel}

 \item  For $Q\in X - S_j'$, the inclusion $\Gscr_j\hookrightarrow \Gscr_{j+1}$
 induces an isomorphism on stalks at $Q$. 
 \label{it:outside}
 
\item  the action of $G = \ZZ/p \ZZ$ on $(F_* \pi_* \Omega^1_Y)_\eta \overset{\varphi_\eta}\simeq \bigoplus_i (\ker V_X)_\eta $ preserves $\Gscr_j$.\label{it:Gstable}

 \item  $\Gscr_0 \simeq \pi_* \ker (V_Y|_{\Fscr_0})$.\label{it:G0}
 \item  $\Gscr_{-1} \simeq \pi_* \ker V_Y$.\label{it:G-1}
 \item $\displaystyle \Gscr_p = \bigoplus_{i=0}^{p-1} \ker V_X(F_*(E_i +pD_i))$.\label{it:Gp}

\end{enumerate}

\end{lem}

\begin{proof}
Let $\nu=(\nu_0,\ldots,\nu_{p-1})$ be a section of $\bigoplus_{i=0}^{p-1} \ker V_X(F_*(E_i+pD_i))$
over an open set $U\subseteq X$ with $\iota(\nu)=(\omega_0,\ldots, \omega_{p-1})$.
By definition, if $\nu\in \Gscr_{j+1}(U)$ then $\omega_i\in F_*(\Omega^1_X(E_i))(U)$ for $i > j$.
In addition, $\omega_j$ lies in $F_*( \Omega^1_X(E_j))(U)$ if and only if
\begin{equation}
 \ord_Q(\omega_j) \geq - \ord_Q(E_j )=-n_{Q,j}\label{eq:valcond}
\end{equation}
for every point $Q$ of $U$.  Condition~\eqref{eq:valcond} holds automatically for $Q \not\in S_j'$.
On the other hand, $\nu$ lies in the kernel of $g_{j+1}$ if and only if
 $\omega_j \in t^{-n_{Q,j}}_Q \Omega^1_{X,Q}$ for all $Q \in S_j'\cap U$, which
 is equivalent to the condition (\ref{eq:valcond}), and this gives the
 claimed identification $\Gscr_j=\ker(g_{j+1}|_{\Gscr_{j+1}})$ in \ref{it:kernel}.  This also shows \ref{it:outside}.

Since $E_i  \geq E_j $ if $i <j$ and the induced action of a generator $\tau$ of $G$ on $\Gscr_p$ from 
Proposition~\ref{prop:isogeneric}
 is given by 
$\tau(y)=y+1$, it is clear that this action preserves $\Gscr_j$ for each $j$. 
 
It follows from Lemma~\ref{lem:regular} and Definition \ref{defn:gs} that $\nu\in \Gscr_0(U)$
if and only if $\omega:=\sum_{i=0}^{p-1} \omega_i y^i \in (\pi_*\Omega^1_{Y})_{\eta}$
is regular over every point $Q$ of $U$ except possibly $Q=Q'$. It is then clear from
the definition of the map $\iota$, Definition~\ref{defn:F0}, and Definition \ref{defn:varphi} that $\omega \in \Fscr_0(U)$
and $\omega$ is killed by $V_Y$; that is, $\omega$ is a section of $\pi_*\ker(V_Y|_{\Fscr_0})$,
which gives the identification \ref{it:G0}.
Since $\varphi$ is injective with source $\pi_*\ker V_Y$, it induces
an isomorphism $\pi_*\ker V_Y \simeq \Im(\varphi)=:\Gscr_{-1}$ as in \ref{it:G-1},
and the description \ref{it:Gp} of $\Gscr_p$ is clear from definitions.
\end{proof}

The main step is to analyze the quotient $\Gscr_{j+1}/\Gscr_{j}$: it will be a skyscraper sheaf supported on $S_j'$.  We will do so by studying the image of $g_{j+1}$.  In particular, for $Q \in S_j'$ we will define $\O_{X,Q}$-modules $M_{j,Q}$ and a map of skyscraper sheaves
\begin{equation}
 c_{j} : \Im(g_{j+1}|_{\Gscr_{j+1}}) \to \bigoplus_{Q \in S_j'} Q_*(M_{j,Q}).  \label{eq:cjmap}
\end{equation}
Composing with the restriction of $g_{j+1}$ to $\Gscr_{j+1}$, we obtain maps of sheaves on $X$
\[
 g'_{j+1} = c_j \circ \left(g_{j+1}|_{\Gscr_{j+1}}\right): \Gscr_{j+1} \to \bigoplus_{Q \in S_j'} Q_*(M_{j,Q})
\]
which we will show via stalk-wise calculations induce isomorphisms
$$g'_{j+1}: \Gscr_{j+1} / \Gscr_{j} \xrightarrow{\sim} \bigoplus_{Q \in S_j'} Q_*(M_{j,Q}),$$ 
for $0\le j\le p-1$, thereby providing an explicit description of these quotients.

In order to motivate the definition of $M_{j,Q}$, we first 
record a result about orders of vanishing that will be useful in 
what follows.


\begin{lem} \label{lem:consolidatedvanishing}
Let $\nu =(\nu_0,\ldots,\nu_{p-1}) \in H^0(X,\Gscr_p)$ and 
set $(\omega_0,\ldots,\omega_{p-1}):=\imath(\nu)$.  
For $0 \leq j \leq p-1$:
 \begin{enumerate}
 \item  For fixed $j$ and $Q\in S$, suppose that $\nu_i=0$ for $i>j$ and $\ord_Q(\nu_j) \ge -n$ for some nonnegative integer $n$.  Let $\mu_{Q,i}$ be the largest multiple of $p$ such that $\mu_{Q,i}+1 \leq n + d_Q(j-i)$.
 Then $\omega_i=0$ for $i>j$, $\ord_Q(\omega_j) \geq -n$, and
  for $i < j$
 \begin{equation}
    \ord_Q(\omega_i) \geq   \min( -\mu_{Q,i} - 1 , \ord_Q(\nu_{i}) )   \geq  -n - d_Q( j-i).\label{eq:orderQinSdesired}
\end{equation}    
    \label{ordvanishQinS}
 \item  If $Q\in D_j$ then $\ord_Q(\omega_j) \geq - p$. \label{ordvanishQinD}
 
 \item  $\ord_{Q'}(\omega_j ) \geq - d_{Q'} (p-1-j)$.\label{ordvanishQ'}
 
 \end{enumerate}
\end{lem}

\begin{proof}
Suppose first that $Q \in S$,
fix $j$ with $0\le j \le p-1$, and assume that $\nu_i=0$ for $i>j$.
Then $\omega_i=0$ for $i>j$ thanks to Remark \ref{rmk:omegav},
while if $i=j$ we have $\omega_j = \nu_j$ thanks to
\eqref{eq:reconstructioneq},
so our assumption that $\ord_Q(\nu_j)\ge -n$ gives $\ord_Q(\omega_j)\ge -n$
as well.
Thus to establish \ref{ordvanishQinS}, it remains to prove
\eqref{eq:orderQinSdesired} for each $i < j$.
We will establish this by descending induction on $i$.  
So suppose that $\ell < j$ and that \eqref{eq:orderQinSdesired}
holds for all $i$ with $\ell < i < j$.  
Since $\omega_i=0$ for $i>j$, Corollary \ref{cor:determined} 
gives $V_X(\omega_{\ell})= V_X(\xi)$ for
\begin{equation}
    \xi := -  \sum_{i=\ell+1}^{j} \binom{i}{\ell} \omega_i (-f)^{i-\ell}.
\end{equation}
Since $\ord_Q(f)=-d_Q$, our inductive hypothesis
and the already established $\ord_Q(\omega_j)\ge -n$
immediately imply that 
\begin{equation*}
    \ord_Q(\xi) \ge \min_{\ell < i \le j} \left\{-n-d_Q(j-i) - d_Q(i-\ell)\right\} = -n -d_Q(j-\ell).
\end{equation*}
Using \eqref{eq:localcartier}, we see that
$\ord_Q(V_X(\xi)) \ge -\mu_{Q,\ell}/p-1$.
It then follows from
Corollary \ref{cor:secondsplitting} that 
$\ord_Q(s_{\ell}(V_X(\xi)))\ge -\mu_{Q,\ell} -1 \geq -n - d_Q(j-\ell)$, so using 
\eqref{eq:omegaj} and remembering that 
$\ord_Q(\nu_{\ell})\ge -n_{Q,\ell} = -\lceil (p-1-\ell) d_Q/p\rceil \ge -d_Q$
yields \eqref{eq:orderQinSdesired} for $i=\ell$, 
completing the inductive step.

The proof of \ref{ordvanishQinD} proceeds by a similar---but simpler---argument,
using descending induction on $j$ and the fact that
for $Q\in D_j$ one has
$\ord_Q (f) =0$ and $\ord_Q(\nu_j) \geq -p$ by definition; we leave the details to the
reader.
Case \ref{ordvanishQ'} likewise follows from a similar argument, using $\ord_{Q'}(f) = -d _{Q'}$ and $\ord_{Q'}(\nu_j) \geq -(p-1-j) d_{Q'}$ by definition; see the proof of  
Lemma~\ref{lem:surj3} for a more detailed version of the analysis in this case.
\end{proof}

\begin{cor} \label{cor:simplevanishing}
Let $\nu =(\nu_0,\ldots,\nu_{p-1}) \in H^0(X,\Gscr_p)$ and 
set $(\omega_0,\ldots,\omega_{p-1}):=\imath(\nu)$.  For $Q \in S$, we have that
\[
\ord_Q(\omega_i) \geq - d_Q (p-1 -i).
\]
\end{cor}

\begin{proof}
Take $j = p-1$ and $n=0$ in Lemma~\ref{lem:consolidatedvanishing}\ref{ordvanishQinS}.
\end{proof}

For a $k[\![t_Q]\!]$-module $M$, let $F_* M$ denote the $k[\![t_Q^p]\!]$-module with underlying additive group $M$ and the action of $t_Q^p$ on $F_* M$ given by multiplication by $t_Q$ on $M$.  Recall the definition of $n_{Q,i}$ from Definition~\ref{defn:collecteddefs}. 

\begin{defn} \label{defn:notation}
For $0\le j \le p-1$ and $Q\in S_j'$ let $m_{Q,j}:=p(n_{Q,j}-1)$ and $\beta_{Q,j}:=\lfloor (p-1)n_{Q,j}/p\rfloor$,
and set
\begin{equation*}
    M_{j,Q}:=\begin{cases}
         k[\![t_Q^p]\!] / (t_Q^{p \beta_{,j}}) = F_* \left(k[\![t_Q]\!]/(t_Q^{\beta_{Q,j}})\right) & \text{if}\ Q\in S\\
         k[\![t_Q^p]\!] / (t_Q^{p(p-1)}) = F_* \left(k[\![t_Q]\!]/(t_Q^{p-1})\right) & \text{if}\ Q\in \sup(D_j) \\
          0  & \text{if}\ Q=Q'
    \end{cases}
\end{equation*}
considered as an $\O_{X,Q}$-module via $\O_{X,Q}\hookrightarrow \O_{X,Q}^{\wedge}\simeq k[\![t_Q]\!]$.
Putting these together, we define a skyscraper sheaf on $X$
\[
M_j := \bigoplus_{Q \in S_j'} Q_*\left(M_{j,Q} \right).
\]
\end{defn}

Note that by construction, if $Q\in S_j'$ then the stalk of $M_j$ at $Q$ is precisely 
$M_{j,Q}$, which justifies the notation.

\begin{remark}\label{rem:rmmeaning}
    For $Q\in S$, one checks easily that $m_{Q,j}$ is the largest integral multiple of $p$ with $-m_{Q,j} \geq -(p-1-j) d_Q +1$, and that $\beta_{Q,j}$ is the number of integer multiples of $p$ between $-m_{Q,j}$ and $-n_{Q,j}$ inclusive. In view of Corollary~\ref{cor:simplevanishing}, 
    the skyscraper sheaves $M_j$ will record all of the possible ways in which $\omega_j$ could fail to be a section of $\Omega^1_X(E_j)$ when $(\omega_0,\ldots,\omega_{p-1}) = \iota(\nu)$ for a section $\nu$ of $\Gscr_{j+1}$.
\end{remark}

In the next sub-sections, we will define the maps (\ref{eq:cjmap}) stalk-by-stalk and check that they are surjective.

\subsection{Local Calculations above Ramified Points} \label{sec:lcrp} 

Fix $j$, let $Q \in S$ and as before let $t_Q$ be a uniformizer at $Q$.  
For $\nu:=(\nu_0,\ldots, \nu_{p-1}) \in \Gscr_{p,Q}$, put 
$(\omega_0,\ldots,\omega_{p-1}):=\iota(\nu)$ and define 
$\omega = \sum_{i=0}^{p-1} \omega_i y^i$, so that $\varphi_{\eta}(\omega)=\nu$ and $V_X(\omega)=0$.
 By \eqref{eq:reconstructioneq}, we have
\begin{equation}
 \omega_j =  \nu_j + s_j ( V_X (\xi))\quad\text{where}\quad \xi:=-\sum_{i=j+1}^{p-1} \binom{i}{j} \omega_i (-f)^{i-j}, \label{eq:omegajdecomp}
\end{equation}
and $\ord_Q(\xi) \geq -(p-1-j) d_Q$ thanks to Corollary~\ref{cor:simplevanishing}. 
Working in the completion $\O^\wedge_{X,Q}\simeq k[\![t_Q]\!]$, we may therefore expand
\begin{equation} \label{eq:powerseries}
 \xi = \sum_{\ell = -N+1}^\infty a_\ell t^\ell_Q \frac{dt_Q}{t_Q}
\end{equation}
where  $N := -(p-1-j) d_Q$. 
Using \eqref{eq:localcartier}, we then compute
\begin{equation} \label{eq:sj}
 s_j \left( V_X \left(\xi  \right) \right) = \sum_{\ell=-m_{Q,j}/p}^\infty a_{p \ell} s_j\left(t_Q^{\ell} \frac{dt_Q}{t_Q}\right),
\end{equation}
where $m_{Q,j}$ is as in Definition \ref{defn:notation}.  
If $p\ell-1 \geq -n_{Q,j}$ then $\ord_Q(s_j( t_Q^\ell \frac{dt_Q}{t_Q} )) \geq -n_{Q,j}$ by \eqref{eq:sectionpolarprop}.
Since we know that $\ord_Q(\nu_j) \geq -n_{Q,j}$, it follows from \eqref{eq:omegajdecomp} that 
\[
 g_{j+1}(\nu_0,\ldots,\nu_{p-1}) \in \Im(g_{j+1})_Q \subset F_* \left( \Omega^1_{X,Q}\left[\frac{1}{t_Q}\right] / t_Q^{-n_{Q,j}} \Omega^1_{X,Q} \right).
\]
Furthermore, $ g_{j+1}(\nu_0,\ldots,\nu_{p-1})$ is determined by the $a_{p\ell}$ for $-m_{Q,j} \leq p \ell \leq -n_{Q,j}$.  There are $\beta_{Q,j}$ such integers by definition (see Remark \ref{rem:rmmeaning}).  
We thus obtain an injection
$c_{j,Q}: \Im(g_{j+1})_Q \into  M_{j,Q} = k[\![t_Q^p]\!] / (t_Q^{p \beta_{Q,j}})$ given by
\begin{equation} 
  c_{j,Q}:\sum_{\ell=-m_{Q,j}/p}^\infty a_{p \ell} s_j\left(t_Q^{\ell} \frac{dt_Q}{t_Q}\right) \bmod t^{-n_{Q,j}}k[\![t_Q]\!] dt_Q\mapsto \sum_{\ell=0}^{\beta_{Q,j}-1} a_{-m_{Q,j}+p\ell}t_Q^{p\ell}
  \bmod t_Q^{p\beta_{Q,j}}k[\![t]\!]
  .\label{eq:cjdef}
\end{equation}
One checks that this is a well-defined
and $\O_X$-linear map, where the action of $\O_X$ on the first factor is the natural one coming from localizing the action on $F_* \Omega^1_X$ and the action of the second factor comes from the 
action of $\O_{X,Q}$ on $F_* k[\![t_Q]\!]$ via the identification $\O_{X,Q}^{\wedge}\simeq k[\![t_Q]\!]$.

\begin{defn}\label{def:g'overS}
  For $Q\in S$ and $0\le j \le p-1$ we define $g'_{j+1,Q}$ as the composite map of $\O_{X,Q}$-modules
  \begin{equation*}
    \xymatrix@C=45pt{
     g'_{j+1,Q}: {\Gscr_{j+1,Q}} \ar[r]^-{(g_{j+1}|_{\Gscr_{j+1}})_Q} & {\Im(g_{j+1}|_{\Gscr_{j+1}})_Q} \ar[r]^-{c_{j,Q}} & M_{j,Q}   
    }.
  \end{equation*}
\end{defn}


\begin{prop} \label{prop:surj1}
For $Q \in S$ and $0\le j \le p-1$, the map $g'_{j+1,Q}$ is surjective.  
\end{prop}

To prove this, we will make use of the following technical result.  As above, 
we identify $\O^{\wedge}_{X,Q}$ with $k[\![t_Q]\!]$ and
note that the completion of $\Omega_{X,Q}^1$ is isomorphic to $\Omega_{X,Q}^1 \tensor{\O_{X,Q}} \O^{\wedge}_{X,Q} \simeq k[\![t_Q]\!] dt_Q$.

\begin{lem} \label{lem:choosingomegas}
 Given $\omega_0, \ldots, \omega_{p-2} \in k(\!(t_Q)\!) dt_Q$ satisfying $\ord_Q(\omega_i) \geq -(p-1-i) d_Q$ for $0\le i\le p-2$, there exists $\omega_{p-1} \in k[\![t_Q]\!] dt_Q$ such that $\omega := \sum_{i=0}^{p-1} \omega_i y^i$ is an element of $(\pi _* \ker V_Y)_Q \tensor{\O_{X,Q}} k(\!(t_Q)\!)$.
\end{lem}

\begin{proof}[Proof of Proposition \ref{prop:surj1}]
Assuming Lemma~\ref{lem:choosingomegas}, we can prove Proposition~\ref{prop:surj1} easily. As in the proof of Lemma \ref{lem:twisted}, 
it suffices to check surjectivity after tensoring with $\O_{X,Q}^{\wedge}$
for each $Q\in S$.  Note that the completion $\Gscr_{j+1,Q}^\wedge$ is isomorphic to $\Gscr_{j+1,Q} \tensor{\O_{X,Q}} \O_{X,Q}^\wedge$.

Pick arbitrary  $\omega_0,\ldots \omega_{j-1} , \omega_{j+1}, \ldots , \omega_{p-2} \in k(\!(t_Q)\!) dt_Q$ with $\ord_Q(\omega_i)\ge -n_{Q,i}$ 
for $i \neq j, p-1$, and set
\[
 \omega_j := s_j\left(t_Q^{-m_{Q,j}/p} \frac{dt_Q}{t_Q} \right).
\]
By Corollary~\ref{cor:secondsplitting} and Definition \ref{defn:notation}, we have
$\ord_Q (\omega_j) \geq -m_{Q,j} -1 \geq -(p-1-j)d_Q$.  As we visibly also have $-n_{Q,i} \geq -(p-1-i) d_Q$
for $i\neq j,p-1$, we may apply Lemma~\ref{lem:choosingomegas} to find $\omega_{p-1}$ such that $ \omega := \sum_{i=0}^{p-1} \omega_i y^i$ lies in  $(\pi_* \ker V_Y)_Q \tensor{\O_{X,Q}} k(\!(t_Q)\!)$ and $\ord_Q(\omega_{p-1}) \geq 0$.  Put $(\nu_0,\ldots , \nu_{p-1}) := \varphi_{\eta}(\omega)$, and note that as $\nu_i = r_i(\omega_i)$ for all $i$, when $i \neq j$ we have $\ord_Q(\nu_i) \geq -n_{Q,i}$.  On the other hand, $\nu_j = 0$ since $r_j \circ s_j = 0$ (as $r(m) = m - s(V_X(m))$ in Corollary~\ref{cor:secondsplitting}).  Thus $(\nu_0,\ldots, \nu_{p-1}) \in \Gscr_{j+1,Q}^{\wedge}$.  
By the definition of $c_{j,Q}$ in (\ref{eq:cjdef}) and the choice of $\omega_j$, we have arranged that 
$$g'_{j+1,Q} (\nu_0,\ldots \nu_{p-1}) = 1 \in k[\![t_Q^p]\!]/(t^{p \beta_{Q,j}})=M_{j,Q} =  M_{j,Q}\otimes_{\O_{X,Q}} \O^\wedge_{X,Q}.$$  This suffices to show that $g'_{j+1,Q}$ is surjective, as it is a map of $\O_{X,Q}$-modules.  
\end{proof}


\begin{proof}[Proof of Lemma \ref{lem:choosingomegas}]
For $h\in k(\!(t_Q)\!)$, let us write $\coef_m ( h)$ for the coefficient of $t_Q^m$ in $h$, and 
for $\xi \in k(\!(t_Q)\!) dt_Q$, write
$\coef_m(\xi)$ for the coefficient of $t_Q^m \frac{dt_Q}{t_Q}$ in $\xi$.
Given $\omega_0,\ldots,\omega_{p-2}$ as in the statement of the lemma,
it follows from the formula
\[
 0 = V_Y(\omega) = \sum_{j=0}^{p-1}\left( \sum_{i=j}^{p-1} V_X\left(\binom{i}{j} \omega_i (-f)^{i-j} \right)\right) y^j
\]
of Lemma~\ref{lem:cartierformula} that our goal is to 
construct $\omega_{p-1}\in k[\![t_Q]\!]dt_Q$ satisfying
\begin{equation}
 V_X\left(\binom{p-1}{j} \omega_{p-1} (-f)^{p-1-j} \right) = - \sum_{i=j}^{p-2} V_X\left(\binom{i}{j} \omega_i (-f)^{i-j} \right)\label{eq:VXidentity}
\end{equation}
for $0 \leq j < p-1$. 
  By \eqref{eq:localcartier} it suffices to check that for each $j$ and every $m \equiv 0 \pmod{p}$, the $m$-th coefficients of both sides of (\ref{eq:VXidentity}) agree.  So we need to show that
\begin{equation} \label{eq:reconstruct}
  \sum_{i} \binom{p-1}{j} \coef_i(\omega_{p-1}) \coef_{m-i} \left( (-f)^{p-1-j} \right) = -\coef_m \left( \sum_{\ell=j}^{p-2}  \binom{\ell}{j} \omega_\ell (-f)^{\ell-j} \right).
\end{equation}
Note the right side is determined by the choice of $\omega_0 , \ldots, \omega_{p-2}$.  We observe:
\begin{enumerate}
 \item $\coef_i \left( (-f)^{p-1-j} \right) =0$ for $i < -(p-1-j) d_Q$; \label{fcoef}
 \item $\coef_{-(p-1-j) d_Q} \left( (-f)^{p-1-j} \right) \neq 0$; \label{fleading}
 \item  $\displaystyle \coef_m\left( \sum_{\ell=j}^{p-2}  \binom{\ell}{j} \omega_\ell (-f)^{\ell-j} \right) = 0$ for $m \leq -(p-1-j) d_Q$.  \label{detcoef}
\end{enumerate}
where \ref{fcoef}, \ref{fleading} are immediate consequences
of the fact that $\ord_Q(f)=-d_Q$, while \ref{detcoef}
follows from the fact that by hypothesis
\[
 \ord_Q(\omega_\ell (-f)^{\ell-j} ) \geq -(p-1-\ell)d_Q + d_Q(j-\ell) = -(p-1-j)d_Q. 
\]
We construct $\omega_{p-1}=\sum_{i} b_i t_Q^i \frac{dt_Q}{t_Q}$ by specifying $b_i=\coef_i(\omega_{p-1})$ inductively as follows.
For $i\le 0$, set $b_i =0$; this choice implies that for each $j$ and any $m\leq -(p-1-j) d_Q$ that is a multiple of $p$, 
the left side of  (\ref{eq:reconstruct}) vanishes,
and likewise for the right side by \ref{detcoef}.  
To specify $b_1$, 
choose $0 \leq j < p$ so $ d_Q(p-1-j)  \equiv 1 \pmod{p}$.  
In light of \ref{fcoef} and the fact that $b_i=0$ for $i \leq 0$, \eqref{eq:reconstruct} 
with $m=-d_Q(p-1-j)+1$
specifies that
\[
\ldots + 0 + \binom{p-1}{j} b_1 \coef_{-d_Q(p-1-j)} \left( (-f)^{p-1-j} \right) + 0 + \ldots = -\coef_{-d_Q(p-1-j)+1} \left( \sum_{\ell=j}^{p-2}  \binom{\ell}{j} \omega_\ell (-f)^{\ell-j} \right)
\]
where the right side has already been specified.  By \ref{fleading}, there is a unique solution $b_1$.

In general, if $N$ is any positive integer and 
 $b_{i}$ has been chosen for all $i < N$, first choose $0 \leq j < p$ so that 
 $d_Q(p-1-j) \equiv N\bmod p$.  The right side of \eqref{eq:reconstruct} with $m = - d_Q(p-1-j)+N$ is already specified. The finitely many non-zero terms of the left side with $i < N$ are determined by our previous choices, while the terms with $i> N$ are zero by \ref{fcoef}.  By \ref{fleading}, we may uniquely solve for $b_N$.
Since $p \nmid d_Q$, when considering $b_i$ with $i$ between $b$ and $b + (p-1)$, each $j$ satisfying $0\leq j < p$ occurs once. Thus the inductive choice of the $b_i$'s makes 
\eqref{eq:reconstruct} hold for every $j$ and every $m$ that is a multiple of $p$.  This completes the proof.
\end{proof}

\subsection{Local Calculations at Poles of Sections} \label{sec:lcps}
We now repeat the analysis of \S \ref{sec:lcrp} for points $Q\in \sup(D_j)$.
Fix $j$ and $Q\in \sup(D_j)$, let $t_Q$ be a uniformizer of $\O_{X,Q}$,
and let $\nu:=(\nu_0,\ldots,\nu_{p-1})\in \Gscr_{p,Q}$.
As before, we put $\omega:=\sum_{i=0}^{p-1} \omega_i y^i$
where $(\omega_0,\ldots,\omega_{p-1}) =\iota(\nu)$, and note that $V_Y(\omega)=0$.

We have $\ord_Q(\omega_j) \geq -p$ thanks to Lemma~\ref{lem:consolidatedvanishing}\ref{ordvanishQinD},
whence a local expansion 
$\omega_j =  \sum_{i=-p}^{\infty} b_i t_Q^i dt_Q $ in the completed stalk at $Q$.
By definition, we then have 
$g_{j+1,Q}(\nu) \equiv \sum_{i=-p}^{-1} b_i t_Q^i dt_Q \bmod k[\![t_Q]\!]dt_Q$,
and we define a map $c_{j,Q} : \Im(g_{j+1})_Q \rightarrow M_{j,Q}=k[\![t_Q]\!]/(t_Q^{p-1})$ by
\begin{equation} \label{eq:cpole}
 c_{j,Q}:  \sum_{i=-p}^{-1} b_i t_Q^i dt_Q \bmod k[\![t_Q]\!]dt_Q  \mapsto  \sum_{i=0}^{p-2} b_{i-p} t_Q^i \bmod t_Q^{p-1} k[\![t]\!].
\end{equation}
This is a well-defined map of $\O_{X,Q}$-modules, using the natural maps $\O_{X,Q} \into \O_{X,Q}^\wedge \simeq k[\![t_Q]\!]$.  

\begin{defn}
    For $0\le j\le p-1$ and $Q\in \sup(D_j)$ we define $g'_{j+1,Q}$ as the $\O_{X,Q}$-linear composite
    \begin{equation*}
    \xymatrix@C=45pt{
     g'_{j+1,Q}: {\Gscr_{j+1,Q}} \ar[r]^-{(g_{j+1}|_{\Gscr_{j+1}})_Q} & {\Im(g_{j+1}|_{\Gscr_{j+1}})_Q} \ar[r]^-{c_{j,Q}} & M_{j,Q}   
    }.
  \end{equation*}
\end{defn}

In other words, $g'_{j+1,Q}$ extracts the coefficients of the monomials $t_Q^{-i} dt_Q$ for $2\le i\le p$ in the local expansion of $\omega_j$ at $Q$.  This suffices to check regularity at $Q$ in the following sense: 

\begin{lem}\label{lem:checkregatQ}
If $\nu \in \Gscr_{j+1,Q}$ and $g'_{j+1,Q}(\nu) =0$, then $\ord_Q(\omega_j) \ge -n_{Q,j}=0$, {\em i.e.}
$\nu \in \Gscr_{j,Q}$.
\end{lem}

\begin{proof}
 Recall \eqref{eq:omegaj} that $\omega_j =  \nu_j + s_j( V_X(\omega_j))$.
 By definition of $\Gscr_{j+1}$ (Definition \ref{defn:gs}), we have
 $\ord_Q(\omega_i) \geq 0$ for $i > j$, 
and $\ord_Q(f)=0$ by the very choice of the divisors $D_j$ in \S\ref{sec:stalks}.
It follows that the differential 
$$\xi:= -\sum_{i=j+1}^{p-1} \binom{i}{j} \omega_i (-f)^{i-j} $$
is regular at $Q$.  By Corollary~\ref{cor:determined} and the fact that $V(\omega)=0$, we have
$V_X(\omega_j)=V_X(\xi)$, which must also then be regular at $Q$, so 
working in the completed stalk at $Q$ we may expand
it locally as
$V_X(\omega_j)=(a_{1}+ a_2t_Q + \ldots) dt_Q$.
But then
\begin{equation}
    \omega_j = \nu_j + s_j( V_X(\omega_j)) = \nu_j + a_1^p s_j(dt_Q) + a_2^p t_Q^p s_j(dt_Q)+ \ldots,
    \label{eq:omegaandv}
\end{equation}
Now $V_X(\nu_j)=0$, so the local expansion of $\nu_j$ has no $t_Q^{-1}dt_Q$-term, 
and since $V_X \circ s_j$ is the identity, $s_j(dt_Q)$ has no $t_Q^{-1} dt_Q$-term either.
From \eqref{eq:omegaandv} 
and the fact that $\ord_Q(s_j(dt_Q))\ge -p$
(by the very construction of $s_j$), we conclude
that the local expansion
of $\omega_j$ has no $t_Q^{-1}dt_Q$-term.  If in addition $g'_{j+1,Q}(\nu)=0$, then
the local expansion of $\omega_j$
has
no $t_{Q}^{-i} dt_Q$-terms for $2 \leq i \leq p$ either, and $\omega_j$ is regular at $Q$.
\end{proof}

\begin{lem} \label{lem:surj2}
For $0\le j\le p-1$ and $Q \in \sup(D_j)$, the map $g'_{j+1,Q}$ is surjective.
\end{lem}

\begin{proof}
    Such surjectivity may be checked after passing to completions, where it follows immediately
     from the proof of Lemma \ref{lem:checkregatQ}
      as $(\ker V_X)_Q^{\wedge}[\frac{1}{t_Q}]$ is (topologically) generated as
     a $k$-vector space by $\{t_Q^{i-1} dt_Q\ :\  p \nmid i\}$, and the only restriction on $\nu_j$ is that $\ord_Q(\nu_j) \geq -p$.
\end{proof}

\subsection{Local Calculations at $Q'$} \label{sec:Q'}
Finally, we analyze the behavior at $Q'$. Because we chose to work with the sheaf $\Fscr_0$ instead of $\pi_* \Omega^1_Y$, the relationship between $\Gscr_j$ and $\Gscr_{j+1}$ at $Q'$ is particularly simple when $0 \leq j \le p-1$:

\begin{lem} \label{lem:surj3}
For $0 \leq j \le p-1$, the natural inclusion $\Gscr_{j,Q'} \to \Gscr_{j+1,Q'}$ is an isomorphism.
\end{lem}

\begin{proof}
We check surjectivity.  Consider $\nu:=(\nu_0,\ldots,\nu_{p-1}) \in \Gscr_{p,Q'}$ 
with $\iota(\nu)=(\omega_0,\ldots,\omega_{p-1})$.
By Lemma~\ref{lem:consolidatedvanishing} \ref{ordvanishQ'}, we in fact have
$\ord_{Q'}(\omega_i) \geq - (p-1-i) d_{Q'}$ for {\em all} $i$, 
which implies in particular that $\nu\in \Gscr_{0,Q'}$, whence the composite
map $\Gscr_{0,Q'}\into \Gscr_{j,Q'}\into \Gscr_{j+1,Q'}\into \Gscr_{p,Q'}$ is an isomorphism,
which gives the claim.  
As the details of the proof of Lemma ~\ref{lem:consolidatedvanishing} \ref{ordvanishQ'}
were abridged, we will spell them out here (sans the descending induction,
which is unnecessary if we assume that $\nu\in \Gscr_{j+1,Q'}$).
So assume that $\nu\in \Gscr_{j+1,Q'}$, or equivalently that
$\ord_{Q'} (\omega_i) \geq - (p-1-i) d_{Q'}$ for $i\geq j+1$. 
Since $\ord_{Q'}(f) = -d _{Q'}$, we compute that
$$\ord_{Q'}(\omega_i(-f)^{i-j}) \ge -(p-1-i)d_{Q'}+(j-i)d_{Q'}=-(p-1-j)d_{Q'}$$
for $i\ge j+1$.
On the other hand, for {\em any} meromorphic
differential $\xi$ on $X$, Corollary~\ref{cor:secondsplitting} shows that $\ord_{Q'}(s_j (V_X(\xi)))\ge \ord_{Q'}(\xi)$.  Then \eqref{eq:reconstructioneq} shows that $\ord_{Q'}(\omega_j) \geq - (p-1-j) d_{Q'}$,
and hence that  $\nu  \in \Gscr_{j,Q'}$ as desired.
\end{proof}

The downside of working with the sheaf $\Fscr_0$ from Definition~\ref{defn:F0} is that we require a separate analysis to relate $\Gscr_0$ with the sheaf we care about, $\Im (\varphi) \simeq \pi_* \ker V_Y$.  Recall
(Definition~\ref{defn:gs}) that we defined 
 $\Gscr_{-1}:=\Im(\varphi) \subseteq \Gscr_0$. 

\begin{defn} \label{defn:-1notation}
 Define $\displaystyle M_{-1,Q'} := \bigoplus_{i=0}^{p-1} \left( k[\![t_{Q'}^p]\!] / (t_{Q'}^{(p-1-i) d_{Q'}}) \right) ^{\oplus (p-1)}$ and write $M_{-1}:=Q'_*(M_{-1,Q'})$ for the skyscraper sheaf on $X$ supported at $Q'$ with stalk $M_{-1,Q'}$.
 Put $S_{-1}':=\{Q'\}$.
\end{defn}

\begin{lem} \label{lem:fixQ}
The cokernel of the natural inclusion map $\Gscr_{-1} \into \Gscr_0$ is isomorphic to $M_{-1}$.  
\end{lem}

We interpret Lemma \ref{lem:fixQ} as giving a short exact sequence
 \begin{equation} \label{eq:g'_0}
 \xymatrix{
  0 \ar[r] & {\Gscr_{-1}} \ar[r] &  {\Gscr_0} \ar[r]^-{g'_{0}} \ar[r] & { M_{-1}} \ar[r] & 0
  }
 \end{equation}
 where $g'_0$ is the composition of the natural map $\Gscr_0 \to \on{coker}( \Gscr_{-1} \to \Gscr_0)$ with the isomorphism of Lemma~\ref{lem:fixQ}. 

\begin{proof}
 We already know that the map $\Gscr_{-1}\rightarrow \Gscr_0$ is an inclusion, 
 and is an isomorphism away from $Q'$, so this is a local question at $Q'$.  
 We have $\Gscr_{-1} \simeq \pi_* \ker V_Y$ and $\Gscr_0 \simeq \pi_*\ker (V_Y|_{\Fscr_0})$ thanks to Lemma \ref{lem:filtered}, and $(\pi_* \ker V_Y)_{Q'}$ consists of meromorphic forms on $Y$ that are regular above $Q'$ and lie in the kernel of $V_Y$, while $\pi_*\ker(V_Y|_{\Fscr_0})_{Q'}$ consists of meromorphic $\omega = \sum_i \omega_i y^i$ in the kernel of $V_Y$ that satisfy $\ord_{Q'}(\omega_i) \geq -(p-1-i) d_{Q'}$ for all $i$ (Definition \ref{defn:F0}). 
 Despite the fact that $\pi$ is \'{e}tale over $Q'$, the decomposition $\omega = \sum_i \omega_i y^i$ is tricky to analyze since the element $f$ defining the Artin-Schreier extension of function fields has a pole at $Q'$ (as does $y$).

 By the very choice of $f$ in \S\ref{sec:asc}, we may find a meromorphic function $g\in K=k(X)$
 such that $f':=f+g^p-g$ has $\ord_{Q'}(f')=0$; necessarily
 $\ord_{Q'}(g) = - d_{Q'}/p$. 
 Let $y' = y + g$, so that 
 \[
  (y')^p- y' = f'
 \]
and $K'=k(Y)$ is the Artin-Schreier extension of $K$
given by adjoining $y'$.
Observe that $y'$ is regular above $Q'$.  Any meromorphic differential $\omega$ on $Y$ may be 
written as
\begin{equation}
    \omega = \sum_{i=0}^{p-1} \omega_i y^i = \sum_{i=0}^{p-1} \omega_i' (y')^i.\label{eq:diffprime}
\end{equation}
While the condition of $\omega$ being regular above $Q'$ is tricky to describe in terms of the $\omega_i$, it is simple to describe in terms of the $\omega'_i$.  Indeed,
the proof of Lemma~\ref{lem:regular} shows that
$\omega$ is regular above $Q'$ if and only if $\ord_{Q'}(\omega'_i) \geq 0$ for $0 \leq i \leq p-1$.
We deduce that the stalk of $\Gscr_{-1}$ at $Q'$ is isomorphic to 
$$N_{-1} := \left \{ \omega = \sum_{i=0}^{p-1} \omega'_i (y')^i : \omega \in (\pi_*\ker V_Y)_{Q'} \text{ and } \omega_i' \in \Omega^1_{X,{Q'}} \right \}.$$ 
On the other hand, substituting $y'=y+g$ in (\ref{eq:diffprime}) and collecting $y^j$-terms gives
 \[
  \omega_j = \sum_{i = j}^{p-1} \binom{i}{j} \omega'_i g^{i-j}.
 \]
A descending induction on $i$ using $\ord_{Q'}(g)=-d_{Q'}/p$ then shows that
$\ord_{Q'}(\omega_i) \geq - (p-1-i) d_{Q'}$ if and only if $\ord_{Q'}(\omega_i') \geq - (p-1-i) d_{Q'}$,
and we conclude that the stalk of $\Gscr_{0}$ at $Q'$ is isomorphic to
$$N_0 := \left \{ \omega = \sum_{i=0}^{p-1} \omega'_i (y')^i : \omega \in (\pi_*\ker V_Y)_{Q'} \text{ and }  \omega'_i \in t_{Q'}^{-(p-1-i) d_{Q'}} \Omega^1_{X,{Q'}} \right\}.$$
We view both $N_{-1}$ and $N_0$ as $k[t_{Q'}^p]$-modules.  
To complete the proof, it suffices to show the cokernel of the natural inclusion $N_{-1} \into N_0$ is isomorphic to $M_{-1,Q'}$.  It suffices to do so after completing.

To analyze the cokernel, we first observe that \eqref{eq:explicitsplitting}--\eqref{eq:explicitprojector} give a section
$s : (\Im V_X)_{Q'}^{\wedge} \to F_* \Omega_{X,Q'}^{1,\wedge}$ and projector $r: F_* \Omega_{X,{Q'}}^{1,\wedge} \to (\ker V_X)_{Q'}^{\wedge}$ to the 
completion of the exact sequence \eqref{eq:cartiersequence} at $Q'$.  By Corollary~\ref{cor:determined}, we have the relations
\begin{equation}
    V_X(\omega_j') = -\sum_{i=j+1}^{p-1} V_X\left(\binom{i}{j} \omega_i' (-f')^{i-j} \right)
    \label{eq:diffprimedetermine}
\end{equation}
for $0\le j\le p-1$, and in particular 
$V_X(\omega'_j)$ is determined by $\omega'_i$ for $i>j$. It follows that   
the map
\begin{equation*}
    \xymatrix{
     {\varphi'_{Q'}: (\pi_*\ker V_Y)_{Q'}^{\wedge}} \ar[r] & {\displaystyle\bigoplus_{i=0}^{p-1} (\ker V_X)_{Q'}^{\wedge}}   
    },
    \quad
    \sum_{i=0}^{p-1} \omega_i' (y')^i \mapsto (r(\omega'_0),\ldots, r(\omega'_{p-1}))
\end{equation*}
is an isomorphism after inverting $t_{Q'}^p$.
Now for any differential $\xi$ on $X$, we have
$\ord_{Q'}(s(V_X(\xi))) \ge p\lfloor \ord_{Q'}(\xi)/p \rfloor$.
As $\ord_{Q'}(f')=0$, it follows from this and \eqref{eq:diffprimedetermine} that:
\begin{itemize}
    \item $s (V_X(\omega'_j)) \in \Omega^{1,\wedge}_{X,{Q'}}$ if $\omega'_i \in \Omega^{1,\wedge}_{X,{Q'}}$ for $i>j$, 
    \item $s( V_X(\omega'_j)) \in t_{Q'}^{-(p-1-j) d_{Q'}} \Omega^{1,\wedge}_{X,{Q'}}$ if  $\omega'_i \in t_{Q'}^{-(p-1-i)d_{Q'}} \Omega^{1,\wedge}_{X,Q'}$ for $i>j$. 
\end{itemize}
As $\omega_j' = r(\omega'_j) + s(V_X(\omega'_j))$ for all $j$,
together with \eqref{eq:diffprimedetermine}  this shows
that $\varphi'_{Q'}$ induces identifications
\[
 {N_{-1}^\wedge}\simeq \bigoplus_{i=0}^{p-1} (\ker V_X)_{Q'}^{\wedge} \quad \text{and} \quad {N_0^\wedge} \simeq \bigoplus_{i=0}^{p-1} \left(t_{Q'}^p \right)^ {-(p-1-i) (d_{Q'}/p)} (\ker V_X)_{Q'}^{\wedge}
\]
as submodules of $\bigoplus_i (\ker V_X)_{Q'}^{\wedge} \left[\frac{1}{t_{Q'}^p} \right]$, 
with the natural inclusion 
$N_{-1}^{\wedge} \into N_0^{\wedge}$ corresponding to the canonical inclusion of direct summands.  As $(\ker V_X)_{Q'}^{\wedge}$ is a free $k[\![t_{Q'}^p]\!]$-module generated by $dt_{Q'}, \ldots , t_{Q'}^{p-1} dt_{Q'}$, this completes the proof.
\end{proof}

\begin{remark} \label{remark:sharper-1}
 Tracing through the proof of Lemma \ref{lem:fixQ}, we see that the image under $H^0(g_{0}')$ 
  of $H^0\left(X,\ker \left( (\tau-1)^{j+1} : \Gscr_0 \to \Gscr_0 \right) \right)$ has dimension at most 
 $$\sum_{i=0}^{j} \frac{(p-1-i) d_{Q'}}{p} (p-1)$$
 as $\omega_i = \omega'_i =0$ for $i>j$ whenever $\omega=\sum_i \omega_i y^i=\sum_i \omega'_i (y')^i$
 is killed by $(\tau-1)^{j+1}$.
\end{remark}

\subsection{Short Exact Sequences} \label{sec:ses}
Combining the local analyses of \S\ref{sec:lcrp}--\ref{sec:Q'},
we can finally construct the desired exact sequences relating the $\Gscr_j$.  Recall the definitions of the skyscraper sheaf $M_j$ in Definition~\ref{defn:notation} (for $0 \leq j \leq p-1$) and Definition~\ref{defn:-1notation} (for $j=-1$).  From \S\ref{sec:lcrp}, \ref{sec:lcps}, and \ref{sec:Q'} we have maps from the stalks of $\Gscr_j$ to pieces of these skyscraper sheaves.  We now put them together.

\begin{defn}\label{def:g'final}
For $0 \leq j \le p-1$, put $S_j:=S\cup\sup(D_j)$ and define 
\begin{equation*}
\xymatrix@C=55pt{
 {c_{j}  : \Im(g_{j+1})} \ar[r] &
 {\displaystyle\bigoplus_{Q \in S_j} Q_*(\Im(g_{j+1})_Q)}
 \ar[r]^-{\bigoplus_{Q \in S_j} Q_* (c_{j,Q})} & 
 {\displaystyle\bigoplus_{Q \in S_j} Q_*(M_{j,Q}) =: M_j}.
}
\end{equation*}
where the first map is the canonical one.   Let $g'_{j+1} := c_j \circ (g_{j+1}|_{\Gscr_{j+1}})$.
\end{defn}

Recall that we also defined $g'_{0}$ in \eqref{eq:g'_0}.

\begin{remark} \label{rem:defg'}
    Notice that the map induced by $g'_{j+1}$ on stalks at $Q$ coincides with 
the previously defined map $g'_{j+1,Q}$, justifying our notation.
\end{remark}

\begin{thm} \label{thm:technicalses}
For $-1 \leq j \le p-1$, there are short exact sequences of sheaves on $X$
\begin{equation*}
\xymatrix{
 0 \ar[r] & {\Gscr_{j}} \ar[r] &  {\Gscr_{j+1}} \ar[r]^-{g'_{j+1}} & {M_j} \ar[r] & 0.
 }
\end{equation*}
We have  $\Gscr_{-1} = \Im (\varphi) \simeq \pi_* \ker V_Y$ and $\Gscr_{p} = \bigoplus_{i=0}^{p-1} \ker V_X(F_*(E_i +p D_i))$.  Furthermore, $M_j$ is a skyscraper sheaf supported on $S_j := S  \cup \sup(D_j)$ for $j \geq 0$, and supported at $Q'$ for $j = -1$.  
\end{thm}

\begin{proof}
 We first establish the exact sequence.  The case $j=-1$ is just Lemma~\ref{lem:fixQ},  
 so suppose $0 \leq j \leq p-1$.  Left exactness is obvious from the definition of $\Gscr_j$ and $\Gscr_{j+1}$.  We can check the rest locally.  For points $Q$ not in $S'_j$, exactness in the middle and on the right is \ref{it:outside} of Lemma \ref{lem:filtered} plus the fact that $M_{j,Q}=0$.  At $Q'$ such exactness is the content of Lemma \ref{lem:surj3} plus the fact that $M_{j,Q'}=0$.  For $Q \in S$, exactness in the middle is simply the first statement of Lemma~\ref{lem:filtered} plus the observation that $c_{j,Q}$ is injective when $Q\in S$, while exactness on the right is Proposition \ref{prop:surj1}.  For $Q\in \sup(D_j)$, exactness in the middle and on the right 
 follow from
 Lemmas \ref{lem:checkregatQ} and \ref{lem:surj2}, respectively.
 Finally, the statements about the support of $M_j$ are clear from the definitions.
\end{proof}

\begin{example} \label{ex:followup}
We continue the notation of Example~\ref{ex:keyexample}, working with the two covers $Y_1, Y_2$ of $\PP^1$ given by $y^5-y = f_1$ and $y^5 -y =f_2$, respectively.  The fact that $\omega = \sum \omega_i y^i = \varphi_\eta^{-1}(\nu_{2,2} )$ is regular for one cover but not the is captured by the machinery of this section as follows.

Note that the element $\nu_{2,2}$ naturally lies in $H^0(\PP^1,\Gscr_2)$.  For both covers, it is moreover in the kernel of $H^0(g'_2)$: looking at the definition of $g'_2$ in \S\ref{sec:lcrp}, the computation that $H^0(g'_2)(\nu_{2,2})=0$ is equivalent to the computation that $\omega_1 = 0$.  Thus, in both cases we also have $\nu_{2,2} \in H^0(\PP^1,\Gscr_1)$.

In Example~\ref{ex:keyexample}, we computed $\omega_0$ using the relations
$$sV_{\PP^1}(\omega_0) = - sV_{\PP^1}(t^{-2} ( t^{-6} dt))=0 \quad \text{and} \quad sV_{\PP^1}(\omega_0) = - s V_{\PP^1}(t^{-2} ( t^{-6} + 2 t^{-5} + t^{-4})dt) = -t^{-6} dt$$
for the covers $Y_1$ and $Y_2$, respectively.  The definition of  $g'_1$ in \S\ref{sec:lcrp} is exactly recording the coefficient of $t^{-6} dt$.  Thus $H^0(g'_1)(\nu_{2,2})$ is $0$ in the case of $Y_1$, but non-zero for the cover $Y_2$, reflecting the fact that $\varphi^{-1}_\eta(\nu_{2,2})$ is regular for $Y_1$ but not $Y_2$. We therefore see that in the case of
$Y_1$ we have
$a_{Y_1}=\dim H^0(\PP^1,\Gscr_0)=  4$ while $a_{Y_2}=\dim H^0(\PP^1, \Gscr_0)  =3$ in the case of $Y_2$.
\end{example}

\section{Bounds} \label{sec:bounds}

We continue with the notation and conventions of the previous sections,
so $\pi : Y \to X$ is a degree-$p$ Artin--Schreier cover of smooth projective curves over $k$
with branch locus $S\subseteq X$.
In this section we will use the short exact sequences of Theorem~\ref{thm:technicalses}
to extract bounds on the $a$-number of $Y$ in terms of the 
$a$-number of $X$ and the breaks in the ramification filtrations at points above $Q\in S$.
Our starting point is the equality
\[
 a_Y=\dim_k \ker\left( V_Y : H^0(Y,\Omega^1_Y) \to H^0(Y,\Omega^1_Y) \right) = \dim_k H^0(Y, \ker V_Y) = \dim_k H^0(X,\Gscr_{-1}),
\]
which follows immediately from Lemma~\ref{lem:filtered}.

\subsection{Abstract Bounds} 

We will first obtain a relatively abstract upper bound on the $a$-number using the short exact sequences of Theorem~\ref{thm:technicalses}.  To streamline the analysis,
we first encode the information contained in these exact sequences
in the study of a single linear transformation.

\begin{defn} \label{defn:wideg} 
For $0\le j \le p-1 $ define 
\begin{equation}
    \xymatrix@C=45pt{
        {\widetilde{g}_{j+1}: H^0(X,\Gscr_p)} \ar[r]^-{H^0(c_j\circ g_{j+1})} & H^0(X,M_j) = \displaystyle\bigoplus_{Q\in S_j'} M_{j,Q} 
    }
\end{equation}
to be the map on global sections induced by $c_j\circ g_{j+1}: \Gscr_p\rightarrow M_j$,
where $c_j$ and $g_{j+1}$ are as in Definitions \ref{def:g'final} and \ref{defn:gjdef}, respectively.
Define 
\begin{equation}
    \xymatrix{
        {\widetilde{g}: H^0(X,\Gscr_p)} \ar[r] & {\displaystyle\bigoplus_{j=0}^{p-1} H^0(X,M_j)} 
    }
    \quad\text{by}\quad \widetilde{g}(\nu):=(\widetilde{g}_1(\nu), \ldots, \widetilde{g}_{p}(\nu))
\end{equation}
Finally define 
\begin{equation}
 \xymatrix{
    {\widetilde{g}_0  : H^0(X,\Gscr_0)} \ar[r]^-{H^0(g'_0)} &  {H^0(X,M_{-1}) = M_{-1,Q'}}
    }
\end{equation}
where $g'_0$ is as in \eqref{eq:g'_0} and $M_{-1,Q'}$ is as in Definition \ref{defn:-1notation}.
\end{defn}

\begin{remark}\label{rem:gmapcompat}
By construction, the restriction of $\widetilde{g}_{j+1}$ to $H^0(X,\Gscr_{j+1})\subseteq H^0(X,\Gscr_p)$ coincides with the map $H^0(g_{j+1}'): H^0(X,\Gscr_{j+1})\rightarrow H^0(X,M_j)$,
where $g_{j+1}'$ is as in Definition \ref{def:g'final} for $j\ge 0$
and as in \eqref{eq:g'_0} for $j=-1$.
While the use of the maps $g'_{j+1}$ is essential for Theorem ~\ref{thm:technicalses},
our analysis below is much simpler when phrased in terms of $\widetilde{g}$
and $\widetilde{g}_0$
as the following key Lemma indicates.
\end{remark}

\begin{lem} \label{lem:kerwidetildeg}
The kernel of $\widetilde{g}$ is $H^0(X,\Gscr_0)$.   
\end{lem}

\begin{proof}
Let $\nu\in H^0(X,\Gscr_p)$ with $\widetilde{g}(\nu)=0$, 
{\em i.e.}~$\widetilde{g}_{j+1}(\nu)=0$
for $0\le j\le p-1$.  Supposing that $\nu\in H^0(X,\Gscr_{j+1})$
for {\em some} $j \le {p-1}$, we have $H^0(g_{j+1}')(\nu)=\widetilde{g}_{j+1}(\nu)=0$
by assumption and the compatibility between $g_{j+1}'$ and $\widetilde{g}_{j+1}$
noted in Remark \ref{rem:gmapcompat}.  Thus 
$\nu\in H^0(X,\Gscr_j)$ thanks to
Theorem~\ref{thm:technicalses} and left exactness of $H^0$.
We conclude by descending induction on $j$ that $\nu\in H^0(X,\Gscr_0)$,
whence $\ker(\widetilde{g})\subseteq H^0(X,\Gscr_0)$, and
the reverse containment is clear.
\end{proof}

\begin{defn} \label{defn:M}
Set $N_1(X,\pi):=\dim_k \Im(\widetilde{g})$ and $N_2(X,\pi):=\dim_k \Im(\widetilde{g}_0)$, and let
$N(X,\pi) = N_1(X,\pi) + N_2(X,\pi)$.  Define $$U(X,\pi) := \sum_{i=0}^{p-1} \dim_k H^0(X,\ker V_X( F_* (E_i + p D_i))) - N(X,\pi).$$
\end{defn}

\begin{lem} \label{lem:upperbound}
We have that $a_Y = U(X,\pi)$.
\end{lem}

\begin{proof}
Lemma~\ref{lem:kerwidetildeg} and the rank-nullity theorem give
\[
\dim_k H^0(X,\Gscr_0) = \dim_k H^0(X,\Gscr_p) - N_1(X,\pi).
\]
Likewise, the $j=-1$ case of Theorem~\ref{thm:technicalses}, left exactness of $H^0$
and the rank-nullity theorem yield 
\[
a_Y= \dim_k H^0(Y,\ker V_Y)  = \dim_k H^0(X,\Gscr_{-1}) = \dim_k H^0(X,\Gscr_0) - N_2(X,\pi).
\]
Combining these equations with the computation
\[ \dim_k H^0(X,\Gscr_p) = \sum_{i=0}^{p-1} \dim_k H^0(X,\ker V_X( F_* (E_i + p D_i))) , 
\]
which follows at once from Lemma \ref{lem:filtered} \ref{it:Gp},
gives the claimed equality.
\end{proof}

The lower bound is a more elaborate application of linear algebra.  The idea is to find a family of subspaces $U_j \subset H^0(X,\Gscr_{p})$ and choose $j$ so that $\dim_k U_j$ is ``large'' while it is easy to see that $\dim_k \widetilde{g} (U_j)$ is ``small''.  Then the rank-nullity theorem gives a lower bound on the kernel of $\widetilde{g}$.
Recall that $G = \ZZ/p \ZZ$ acts on $\pi_* \Omega^1_Y$ by having a generator $\tau\in G$ send $y \mapsto y+1$.  
By Lemma~\ref{lem:filtered}, this action preserves $\Gscr_j$.

\begin{defn} Let $i$ and $j$ be integers with  $ 0 \le i \leq j \le p-1$.  \label{defn:bounds} 
\begin{itemize}
    \item  Define $U_j := \ker \left( (\tau-1)^{j+1} : H^0(X,\Gscr_{j+1}) \to H^0(X,\Gscr_{j+1}) \right)$.

    \item  For $Q \in S$, define $c(i,j,Q)$ to be the number of integers $n$ congruent to $-1$ modulo $p$ such that $-n_{Q,j} - d_Q (j-i) \leq n < -n_{Q,i}$.  
    
    \item  For $Q \in \on{sup}(D_i)$, define $c(i,j,Q) = p-1$. 
    
    \item  Define $c(i,j,Q') = \frac{d_{Q'} (p-1) (p-1-i)}{p}$.
    
    \item   Define $L(X,\pi)$ to be
 \[
L(X,\pi) := \max_{0\le j \le p-1} \left( \dim_k  U_j  - \sum _{i=0} ^j \sum_{Q \in S_i'} c(i,j,Q) \right).
 \]
\end{itemize}
\end{defn}

\begin{lem} \label{lem:vk}
There is an isomorphism \[
 U_j \simeq \bigoplus_{i=0}^{j} H^0(X,\ker V_X(F_*( E_i + pD_i)))  .
\]
\end{lem}

\begin{proof}
Both sides are isomorphic to
 $$\left \{ \omega = \sum_{i=0}^{j} \omega_i y^i \in (\pi_* \ker V_Y)_\eta: \varphi_\eta(\omega) \in H^0(X,\Gscr_{p})  \right\}.$$
 This follows from the formulas in Lemma~\ref{lem:cartierformula}, the fact that $\tau-1$ reduces the maximum power of $y$ appearing in a differential by one, and the definition of $\varphi$.
\end{proof}  

\begin{lem} \label{lem:lowerbound}
We have that $a_Y \geq L(X,\pi)$.
\end{lem}

\begin{proof}
Fix $j$.  The key idea is to use descending induction on $i$ and the rank-nullity theorem to give a lower bound on $\dim_k (U_j \cap H^0(X,\Gscr_i))$ for $-1 \leq i \leq j+1$.  Taking $i=-1$ gives a lower bound (depending on $j$) on $\dim_k ( U_j \cap H^0(Y, \ker V_Y)) \leq a_Y$, which will establish the result.  
 
 For $i =j+1$, we see that 
 \begin{equation}
 \dim_k (U_j \cap H^0(X,\Gscr_{j+1})) = \dim_k U_j.\label{eq:indbasecase}
 \end{equation}
 To ease the notational burden, for
  $-1\le i\le j$ 
 let us write $\psi_{i+1}^j$ for the map
 \begin{equation}
    \xymatrix{
    {\psi_{i+1}^j: U_j\cap H^0(X,\Gscr_{i+1})} \ar[r] & H^0(M_i)
    }\quad\text{given by}\quad \psi_{i+1}^j:= H^0(g'_{i+1})\big|_{U_j\cap H^0(X,\Gscr_{i+1})}\label{eq:psimapdefn}
\end{equation}
Theorem~\ref{thm:technicalses} then gives exact sequences
 \begin{equation}
    \xymatrix{
    0 \ar[r] &{U_j \cap H^0(X,\Gscr_i)} \ar[r] & {U_j \cap H^0(X,\Gscr_{i+1})}\ar[r]^-{\psi_{i+1}^j}&
    {H^0(X,M_i) = \displaystyle \bigoplus_{Q\in S_i'} M_{i,Q}}
 }\label{eq:LBexseq}
 \end{equation}
 for $-1 \leq i \leq  j$, 
which need not be right exact.  Nonetheless, we may work stalk-by-stalk to obtain 
information about the image of $\psi_{i+1}^j$ as follows. 
For $Q\in S_i'$, let $\psi_{i+1,Q}^j$ denote the composition of $\psi_{i+1}^j$
with projection onto the factor of the direct sum indexed by $Q$;
by construction $\psi_{i+1,Q}^j$ coincides with the composition
\begin{equation}
    \xymatrix{
     {\psi_{i+1,Q}^j: U_j\cap H^0(X,\Gscr_{i+1})} \ar@{^{(}->}[r] & {H^0(X,\Gscr_{i+1})}
     \ar[r] &
     {\Gscr_{i+1,Q}} \ar[r]^-{g'_{i+1,Q}} & M_{i,Q}
     }\label{eq:psidefrecap}
\end{equation}
wherein the first two maps are the canonical ones.  We have the obvious inclusion 
\begin{equation}
    \Im(\psi_{i+1}^j) \subseteq \bigoplus_{Q\in S_i'} \Im(\psi_{i+1,Q}^j)\label{eq:dsumincl}
\end{equation}
so to bound $\dim_k \Im(\psi_{i+1}^j)$ from above it will suffice
to do so for each $\dim_k \Im(\psi_{i+1,Q}^j)$.
First suppose that $i \geq 0$.  Lemma~\ref{lem:imagegj} below will give an upper bound of $c(i,j,Q)$ on $\dim_k \Im(\psi_{i+1,Q}^j)$ for $Q \in S$.  For $Q \in \on{sup}(D_i)$, the dimension of  $\Im (\psi_{i+1,Q}^j)$ is at most $p-1$ by the very definition of $g'_{i+1,Q}$ in \S\ref{sec:lcps}, and we defined $c(i,j,Q) := p-1$ in this case.  
Applying the rank-nullity theorem to the exact sequence 
\eqref{eq:LBexseq} and using \eqref{eq:dsumincl}
then gives
 \[
 \dim_k( U_j \cap H^0(X,\Gscr_{i})) \geq \dim_k (U_j \cap H^0(X,\Gscr_{i+1})) - \sum_{Q \in S_i'} c(i,j,Q).
 \]
 By descending induction on $i$ and \eqref{eq:indbasecase}, we deduce
 \[
 \dim_k ( U_j \cap H^0(X,\Gscr_0)) \geq \dim_k U_j - \sum_{i=0}^j \sum_{Q \in S_i'} c(i,j,Q).
 \]
Finally, using Remark~\ref{remark:sharper-1} to analyze \eqref{eq:LBexseq} when $i=-1$, we conclude
\[
 \dim_k ( U_j \cap H^0(X,\Gscr_{-1})) \geq \dim_k U_j - \sum_{i=0}^j \sum_{Q \in S_i'} c(i,j,Q). \qedhere
\]
\end{proof}

\begin{lem} \label{lem:imagegj}
If $Q \in S$ and $0 \leq i \leq j < p$, 
then $\dim_k \Im(\psi_{i+1,Q}^j)\leq c(i,j,Q)$.
\end{lem}

\begin{proof} 
We continue the notation of \S\ref{sec:lcrp}.
Fixing $0\le i\le j < p$ and
recalling Lemma \ref{lem:consolidatedvanishing}, for $\nu  \in U_j$ 
we have 
$$\nu = (\nu_0,\ldots,\nu_j,0,\ldots,0)\ \text{and}\ (\omega_0,\ldots,\omega_{j},0,\ldots,0):=\iota(\nu).$$ 
From \eqref{eq:reconstructioneq}, we have
$\omega_i = \nu_i + s_i(V_X(\xi))$ where
\begin{equation}
    \xi:=-\sum_{\ell=i+1}^{j} \binom{\ell}{j} \omega_\ell (-f)^{\ell-i} 
\end{equation}
since $\nu_\ell$ and $\omega_\ell$ are $0$ for $\ell>j$.
From Lemma~\ref{lem:consolidatedvanishing} \ref{ordvanishQinS}
with $n=n_{Q,j}$ we have
\begin{equation} \label{eq:orderestimate}
 \ord_Q(\xi)   \geq -n_{Q,j} - (j-i) d_Q.
\end{equation}
By the very definitions (see \eqref{eq:psidefrecap}
and Definitions \ref{def:g'overS} and \ref{defn:gjdef}),
to compute $\psi_{i+1,Q}^j(\nu)$, we first expand $\xi$ locally at $Q$ as a power series
$\xi= \sum_n a_n t^n_Q \frac{dt_Q}{t_Q}$
and evaluate
\begin{equation}
s_i ( V_X(\xi))  = \sum_{n=-B}^\infty a_{p n} s_i\left(t_Q^{n} \frac{dt_Q}{t_Q}\right)\label{eq:siVx}
\end{equation}
for some integer $B$.  Then $\psi_{i+1,Q}^j(\nu)$ records only those coefficients on the right side of \eqref{eq:siVx}
where $n \geq -B$ and $p n -1 < -n_{Q,i}$.  In our situation we have {\em fewer} coefficients to record than in the general analysis of \S\ref{sec:lcrp}, precisely because
$\nu_{\ell}=0 = \omega_{\ell}$ for $\ell > j$.
Indeed, based on \eqref{eq:orderestimate} we may take $B$ to be the smallest integer greater than or equal to $\left(-n_{Q,j} - (j-i)d_Q+1\right)/p$.
In particular, the number
of potentially non-zero coefficients $a_{pn}$ is the number of multiples of $p$ between
$-n_{Q,j} - (j-i)d_Q+1$ and $-n_{Q,i}$ inclusive, which by definition is the integer $c(i,j,Q)$.  This integer therefore gives an upper bound on the dimension of
$\Im(\psi_{i+1,Q}^j)$
\end{proof}

\begin{prop} \label{prop:abstractbound}
The $a$-number of $Y$ satisfies
\[
 L(X,\pi) \leq a_Y = U(X,\pi).
\]
\end{prop}

\begin{proof}
Combine Lemma~\ref{lem:upperbound} and Lemma~\ref{lem:lowerbound}.
\end{proof}

\subsection{Tools} \label{sec:tools}
The quantities $L(X,\pi)$ and $U(X,\pi)$ are quite abstract.  We now explain how to bound $L(X,\pi)$ and $U(X,\pi)$ in terms of the ramification of $\pi : Y \to X$ and the genus and $a$-number of $X$.

The key is a theorem of Tango which allows us to compute the dimension of the kernel of the Cartier operator on certain spaces of (global) meromorphic differential forms on $X$
whose poles are ``sufficiently bad''.  
Let $\sigma: k\rightarrow k$ denote the $p$-power Frobenius automorphism of $k$, and let
$g_X$ be the genus of $X$.  Attached to $X$ is its {\em Tango number}:
\begin{equation}
	\n(X) := \max\left\{ \sum_{x\in X(k)} \left\lfloor \frac{\ord_{x}(df)}{p}\right\rfloor\ : 
	f\in k(X) - k(X)^p \right\}.
\end{equation}
In Lemma 10 and Proposition 14 of
\cite{tango72}, one sees that $\n(X)$ is well-defined and is an integer satisfying
$-1\le \n(X)\le \lfloor (2g_X-2)/p\rfloor$, with the lower bound an equality if and only if $g_X=0$.

\begin{fact}[Tango's Theorem]\label{fact:tango}
	Let $\L$ be a line bundle on $X$.  If $\deg \L > \n(X)$ then the natural $\sigma$-linear map
	\begin{equation}
		\xymatrix{
			{F_X^*:H^1(X,\L^{-1})} \ar[r] & {H^1(X,\L^{-p})}
			}\label{Frob}
	\end{equation} 
	induced by pullback by the absolute Frobenius of $X$ is injective,
	and
	the $\sigma^{-1}$-linear Cartier operator
	\begin{equation}
		\xymatrix{
			{V_X :H^0(X,\Omega^1_{X/k}\otimes\L^p)} \ar[r] & {H^0(X,\Omega^1_{X/k}\otimes\L)}
			}\label{Cartier}
	\end{equation} 
	 is surjective.
\end{fact}

\begin{remark}
This is \cite[Theorem 15]{tango72}; strictly speaking, Tango requires $g_X>0$;  however, by tracing through Tango's argument---or by direct calculation---one sees easily that 
	the result holds when $g_X=0$ as well.
\end{remark}

To simplify notation, let $\delta(H^0(X,\Omega^1_X(E)))$ denote the dimension of the kernel of $V_X$ on that space of differentials. 
Tango's theorem tells us the following:

\begin{cor} \label{cor:tango}
Let $D, R$ be divisors on $X$ with  $R =\sum r_i P_i$ where $0 \le r_i < p$.  If $\deg(D) > \max(\n(X),0)$, then 
$$\delta (H^0(X,\Omega^1_X(pD+R))) = (p-1) \deg(D) + \sum_i \left( r_i- \left\lceil \frac{r_i}{p} \right\rceil\right). $$
For arbitrary $D$, we have the weaker statement
$$0 \leq  \delta(H^0(X,\Omega^1_X(pD+R))) -  \left( (p-1) \deg(D) + \sum_i \left( r_i- \left\lceil \frac{r_i}{p} \right\rceil\right) \right) \leq a_X. $$
\end{cor}

\begin{proof}
When $R=0$, the first case follows from the surjectivity of the Cartier operator in Fact~\ref{fact:tango} taking $\L = \O_X(D)$, plus the fact that 
$$\dim_k H^0(X,\Omega^1_X(D)) = g_X-1+ \deg(D) \quad \text{and} \quad \dim_k H^0(X,\Omega^1_X(pD)) = (g_X-1) + p \deg(D)$$
from the Riemann-Roch theorem.  

We can build on this to prove the remaining cases of the first statement and establish the inequalities of the second statement.
We know that for any divisor $E$ with $\deg(E) \geq 0$ and any closed point $P$ of $X$,
\[
\dim_k H^0(X,\Omega^1_X(E + [P])) \le \dim_k H^0(X,\Omega^1_X(E)) + 1.
\]
with equality whenever $\deg(E)>0$.
Thus we know that 
\[
0 \leq \delta(H^0(X,\Omega^1_X(E + [P]))) - \delta( H^0(X,\Omega^1_X(E)) ) \leq 1.
\]
Furthermore, if $p | \ord_P(E)$ then this difference is $0$, as a differential in $H^0(X,\Omega^1_X(E + [P]))$ not in $H^0(X,\Omega^1_X(E))$ must have a non-zero $t_P^{-\ord_P(E)}\frac{dt_P}{t_p}$-term in the completed stalk at $P$, which forces the differential to not lie in the kernel of the Cartier operator.

When $\deg(D) > \max(\n(X),0)$ and $R \neq 0$, we prove the equality by induction on the number of points in the support of $R$.  
Assume the equality holds for some fixed $R$ and all $D$ with $\deg(D) > \max(\n(X),0)$.
Pick another point $P$ (not in the support of $R$) and $0 < r < p$.  Then 
\[
 \delta(H^0(X,\Omega^1_X(p D + R))) \leq \delta( H^0(X,\Omega^1_X(pD +R + [P])))\leq \ldots \leq \delta(H^0(X,\Omega^1_X(p(D+[P]) +R))) 
\]
and by the inductive hypothesis the last is $p-1$ more than the first.  This means that at each step after the first, the dimension of the kernel must increase by one.  
This shows that
$$\delta(H^0(X,\Omega^1_X(p D + R + r [P]))) = \delta(H^0(X,\Omega^1_X(pD +R ))) + (r-1),$$
which completes the induction. 

The second statement follows from the same type of reasoning.  When passing from
$$a_X = \delta( H^0(X,\Omega^1_X)) \quad \text{to} \quad \delta (H^0(X,\Omega^1_X(pD +R))),$$ there are $(p-1) \deg(D) + \sum (r_i-\lceil r_i/p\rceil)$ times the dimension might increase by one.  When passing from $$\delta (H^0(X,\Omega^1_X(pD +R))) \quad \text{to} \quad  \delta(H^0(X,\Omega^1_X( p D'))) = (p-1) \deg(D')$$ with $D'$ chosen so that $pD' \geq p D+R$ and $D' > \n(X)$, there are $(p-1) \deg(D'-D) + \sum (p-1-r_i)$ chances for the dimension to increase by one.  This completes the proof.
\end{proof}

\begin{remark}
Choosing $D' \geq D$ with $\deg(D') > \n(X)$ and $p D' \geq pD +R$, we also obtain the bound $\delta( H^0(X,\Omega^1_X(pD+R)) ) \leq (p-1) \deg D'$  from 
the inclusion $H^0(X,\Omega^1_X(p D + R)) \subset H^0(X,\Omega^1_X(pD'))$ and Tango's theorem.
\end{remark}

Fix a divisor $E = p D + R = \sum_i a_i Q_i$ with $R = \sum r_i Q_i$ and $0 \leq r_i < p$.  For fixed $j$ and $1 \leq n \leq a_j$, 
\[
\delta(H^0(X,\Omega^1_X(\sum_{i <j} a_i Q_i + (n-1)Q_j))) \leq \delta(H^0(X,\Omega^1_X(\sum_{i <j} a_i Q_i + n Q_j))).
\]
If the dimension increases by one, let $\xi_{Q_j,n}$ be a differential in the larger space not in the smaller space.  This differential satisfies the following properties:
\begin{enumerate}
    \item $V_X(\xi_{Q_j,n}) =0$;
    \item $\ord_{Q_i} (\xi_{Q_j,n}) \geq -a_i $ for $i < j$;
    \item  $\ord_{Q_j}( \xi_{Q_j,n}) = -n$;
    \item and $\ord_{Q_i} (\xi_{Q_j,n}) \geq 0$ for $i>j$.
\end{enumerate}
Note that such a differential never exists if $n \equiv 1 \pmod{p}$.

\begin{example}
Suppose $X = \PP^1$.
    For any positive integer $n$ with $n \not \equiv 1 \pmod{p}$ and closed point $Q = [\alpha]$ of $X$, we may take
$\xi_{Q,n}$ to be $(t-\alpha)^{-n} dt$.
\end{example}

\begin{cor} \label{cor:omegaqi}
The differentials $\xi_{Q_j,n}$ are linearly independent.  There are at least  $(p-1) \deg D + \sum_i (r_i - \lceil r_i / p\rceil) - a_X$ of them.
\end{cor}

\begin{proof}
This follows from the proof of the previous Corollary.
\end{proof}

 We can use Corollary~\ref{cor:tango} to relatively easily give a formula for $\dim_k U_j$ using Lemma~\ref{lem:vk} provided that the auxiliary divisors $D_i$ are sufficiently large.  This will allow us to bound $L(X,\pi)$; we do so in the proof of Theorem~\ref{thm:cleanbounds}.  
 We can also use Corollary~\ref{cor:tango} to obtain a bound on 
 the number $N(X,\pi)$ of Definition \ref{defn:M}; this is more involved, and is the subject of the next section.  

\subsection{Estimating $N(X,\pi)$}

We will use Corollary~\ref{cor:omegaqi} to construct differential forms to give a lower bound on $N(X,\pi)$, and hence an upper bound on $U(X,\pi)$.  

\begin{defn} \label{defn:T}
Define $T$ to be the set of triples $(Q,n,j)$ where:
\begin{itemize}
 \item $0 \leq j \leq p-1$ and  $Q \in S'_j$;
 \item $ 0 < n \leq \ord_Q(E_j + p D_j)$;
 \item  $n \not \equiv 1 \pmod{p}$;
 \item if $Q \in S$, there exists an integer $m$ with $0\leq m \leq j$ such that $m \equiv j + (n-1) d_Q^{-1} \pmod{p}$.
\end{itemize}
\end{defn}

We fix an ordering on $ \bigcup_j S'_j$ with $Q'$ the smallest element.
For $(Q,n,j) \in T$, suppose the element $\xi_{Q,j}$ from Corollary~\ref{cor:omegaqi} exists (taking  $E = E_j + p D_j$ and $a_i = n_{Q_i,j}$ in the notation there, and using the ordering on $ \bigcup_j S'_j$). 
We set
\[
 \nu_{Q,n,j} := (0 , \ldots, 0 , \xi_{Q,j}, 0, \ldots, 0) 
 \in H^0(X,\Gscr_{j+1})
\]
where $\xi_{Q,j}$ occurs in the $j$-th component.
In that case, we say that {\em $\nu_{Q,n,j}$ exists}.  As $Q'$ is smallest in the chosen ordering, we know that $\nu_{Q',n,j} \in H^0(X,\Gscr_0)$.
Our aim is to prove:

\begin{prop} \label{prop:t}
Let $N(X,\pi)$ be as in Definition \ref{defn:M}.
Then $N(X,\pi) \geq \#T - B$, where $B$ is the number of triples $(Q,n,j) \in T$ such that $\nu_{Q,n,j}$ does not exist.
\end{prop}

The idea is to show that the images under 
$\widetilde{g}$ of the $\nu_{Q,n,j}$ for $(Q,n,j) \in T$ with $Q \neq Q'$ are linearly independent, and similarly for the images under $\widetilde{g}_0$ of the $\nu_{Q',n,j}$ with $(Q',n,j) \in T$.

\begin{defn}
\label{defn:UQnj}
We define an ordering on $T$ by setting $(P,a,j) < (P',a',j')$ provided 
\begin{enumerate}
    \item $j < j'$;
    \item or $j =j'$ and $P < P'$ ;
    \item or $j=j'$ and $P = P'$ and $a<a'$.
\end{enumerate}
For $(Q,n,j) \in T$, we then define
\[
 U_{Q,n,j} := \spn_k \left \{ \widetilde{g}( \nu_{P,a,i} ) : 
 \nu_{P,a,i}\ \text{exists}\ \text{and}\ 
 (P,a,i) < (Q,n,j)  \right \}.
\]
\end{defn}

\begin{lem}\label{lem:orderprop}
    Let $(P,a,i), (Q,n,j)\in T$ with $(P,a,i)< (Q,n,j)$
    and suppose that the differentials $\nu_{P,a,i}$
    and $\nu_{Q,n,j}$ both exist.  Writing 
    $(\omega_0,\ldots,\omega_{p-1}) = \imath(\nu_{Q,n,j})$
and $(\omega'_0,\ldots, \omega_{p-1}') = \imath(\nu_{P,a,i})$,
we have 
$\ord_{Q}(\omega'_j) > \ord_{Q}(\omega_j) = - n$.
\end{lem}

\begin{proof}
    Suppose first that $i < j$.  Since the $\ell$-th component 
    of $\nu_{P,a,i}$ is zero for $\ell >i$, we have 
    $\omega_{\ell}'=0$ for $\ell>i$ by Lemma \ref{lem:consolidatedvanishing} \ref{ordvanishQinS},
    and in particular $\omega_{j}'=0$ which immediately gives the claim.
    So suppose that $i=j$ and $P < Q$.  In this case, 
    by very construction of the differentials in Corollary \ref{cor:omegaqi}
    and our choice of ordering, we have
    $\ord_Q(\omega'_j)=\ord_Q(\xi_{P,a})\ge 0$ while 
    $\ord_Q(\omega_j) = \ord_Q(\xi_{Q,n}) = -n$.
    Finally, suppose that $i=j$, $P=Q$, and $a< n$.  Then,
    \begin{equation*}
        \ord_Q(\omega'_j) = \ord_Q(\xi_{Q,a}) = -a > -n = \ord_Q(\xi_{P,n}) = \ord_Q(\omega_j)
    \end{equation*}
    as claimed.
\end{proof}

\begin{lem} \label{lem:t1}
Suppose $\nu_{Q,n,j}$ exists 
with $(Q,n,j) \in T$ and $Q\in S$.
Then $\widetilde{g} (\nu_{Q,n,j})\not\in U_{Q,n,j}$.
\end{lem}

\begin{proof}
In this proof, let $\widetilde{g}_{m+1,Q}$ denote the composition of $\widetilde{g}_{m+1}$ with the projection onto $M_{j,Q}$.
Write $(\omega_0,\ldots,\omega_j,0,\ldots,0)=\iota(\nu_{Q,n,j})$.
As $Q \in S$, by Definition~\ref{defn:T} there exists an integer $0 \leq m \leq j$ 
such that  $m \equiv j + (n-1) d_Q^{-1} \pmod{p}$. 
Looking at the definition of $g'_{m+1,Q}$ in \S\ref{sec:lcrp}, to compute $\widetilde{g}_{m+1,Q}(\nu_{Q,n,j})$ 
we work in the completed stalk at $Q$ and write
 \begin{equation}  \label{eq:xim}
    \xi := - \sum_{i=m+1} ^{j} \binom{i}{m} \omega_i (-f)^{i-m}  = \sum_{i = -d_Q(p-1-m)+1}^\infty a_i t^i_Q \frac{dt_Q}{t_Q}.
\end{equation}
 Now $\widetilde{g}_{m+1,Q}$ records the $a_i$ for which $p|i$ and $-d_Q(p-1-m) \leq i < -n_{Q,m}$.

Using Lemma~\ref{lem:consolidatedvanishing} \ref{ordvanishQinS} we find
for $i\le j$
\begin{equation} \label{eq:strongordervanishing}
\ord_Q(\omega_i (-f)^{i-m} ) \geq -n -(j-m) d_Q 
\end{equation}
In fact, we know more: for $i\ne j$, the $i$-th component of
$\nu_{Q,n,j}$ is {\em zero}, whence
Lemma~\ref{lem:consolidatedvanishing} \ref{ordvanishQinS} shows that for $m < i < j$
\begin{equation}
\begin{aligned}
\ord_Q(\omega_i (-f)^{i-m}) &= \ord_Q(\omega_i) - (i-m)d_Q\\
&\ge - p\left\lfloor \frac{n+(j-i)d_Q - 1}{p}\right\rfloor - 1 - (i-m) d_Q\\
&> -(n+(j-i)d_Q-1) - 1 -(i-m) d_Q = -n -(j-m)d_Q
\end{aligned}\label{eq:strongerordervanishing}
\end{equation}
where the strict inequality comes from the
fact that $n + (j-i) d_Q \not \equiv 1 \pmod{p}$ for $m < i \leq j$, 
as $m$ is by definition the {\em unique} solution modulo $p$ to
\[
n + (j-x) d_Q \equiv 1 \pmod{p}.
\]

We now adopt the notation from the proof of Lemma~\ref{lem:choosingomegas} for coefficients of differentials and functions.  
As $\omega_j = \xi_{Q,n}$, we have $\coef_{1-n}(\omega_j)\neq 0$.  As $\ord_Q(f) = -d_Q$, we have $\coef_{-d_Q}( f) \neq 0$.  By \eqref{eq:strongerordervanishing}, we know that $\coef_{1-n-(j-m)d_Q }(\omega_i (-f)^{i-m}) =0$ for $m < i < j$.
Then looking at coefficients in \eqref{eq:xim} we see that 
\begin{align}
\begin{split}
    a_{1-n-(j-m)d_Q} &= \coef_{1-n-(j-m)d_Q}\left(- \sum_{i=m+1} ^{j} \binom{i}{m} \omega_i (-f)^{i-m}\right) \\
    &= 
    \coef_{1-n-(j-m)d_Q}\left(-\binom{j}{m}\omega_j (-f)^{j-m}\right) \\
    &= -\binom{j}{m}\coef_{1-n}(\omega_j)\cdot 
    \coef_{-(j-m)d_Q}((-f)^{j-m})\neq 0.
\end{split}
\end{align}
But $1-n - (j-m) d_Q \equiv 0 \pmod{p}$, so 
 $a_{1-n-(j-m)d_Q}$ is one of the {\em nonzero} terms that $\widetilde{g}_{m+1,Q}$ records.
  This is enough to show that $\widetilde{g}(\nu_{Q,n,j})$ does not lie 
 in $U_{Q,n,j}$.  Indeed, suppose 
$(P,a,i)\in T$ with $(P,a,i) < (Q,n,j)$ and that $\nu_{P,a,i}$ exists.
Set $(\omega'_0, \ldots,\omega_{p-1}') = \imath(\nu_{P,a,i})$.  
By Lemma \ref{lem:orderprop},
we have $\ord_Q(\omega_j') > -n$, 
so thanks to Lemma~\ref{lem:consolidatedvanishing} \ref{ordvanishQinS} we see that
$\ord_Q(\omega'_\ell) > -n - d_Q(j-\ell)$ for all $\ell \le j$.
It follows that
  \[
   \coef_{1-n-(j-m)d_Q} (\omega'_m) = 0.
  \]
As $U_{Q,n,j}$ is spanned by $\widetilde{g}(\nu_{P,a,i})$ 
for such triples $(P,a,i)$, this
shows $\widetilde{g}(\nu_{Q,n,j})$ does not lie in $U_{Q,n,j}$
as claimed.
\end{proof} 
 
\begin{lem} \label{lem:t2}
Suppose $\nu_{Q,n,j}$ exists 
with $(Q,n,j) \in T$ and $Q\in D_j$.
Then $\widetilde{g} (\nu_{Q,n,j})\not\in U_{Q,n,j}$.
\end{lem}

\begin{proof}
In this proof, let $\widetilde{g}_{m+1,Q}$ denote the composition of $\widetilde{g}_{m+1}$ with the projection onto $M_{j,Q}$.
As $\ord_Q(E_j + p D_j) = p$ and we assumed $n \neq 1$, we know $1 < n \leq p$.  Set
$(\omega_0,\ldots,\omega_{p-1})=\iota(\nu_{Q,n,j})$
and notice that $\omega_j  = \xi_{Q,n}$.
Looking at the definition of $g'_{j+1,Q}$ in \S\ref{sec:lcps}, we see that $g'_{j+1,Q}$ extracts the coefficients of $t_Q^{-p} dt_Q$ through $t_Q^{-2} dt_Q$ in a local expansion of $\omega_j$ at $Q$.  In particular, since $\ord_Q(\omega_j) = -n$,
we see that the component of $\widetilde{g}_{j+1,Q}(\nu_{Q,n,j})$ 
corresponding to the coefficient of $t_Q^{-n}dt_Q$ 
in the local expansion of $\omega_j$ at $Q$
is {\em nonzero}.
On the other hand, 
for any $(P,a,i)\in T$ such that $\nu_{P,a,i}$ exists and
$(P,a,i) < (Q,n,j)$, 
consider $(\omega'_0,\ldots, \omega'_{p-1}) = \imath(\nu_{P,a,i})$.  By 
Lemma \ref{lem:orderprop} we have
$\ord_Q(\omega'_j) > - n$ so that the 
component
of $\widetilde{g}_{j+1,Q}(\nu_{P,a,i})$
corresponding to the coefficient of $t_Q^{-n}dt_Q$ 
in the local expansion of $\omega'_j$ at $Q$
is by contrast {\em zero}.
As the $\widetilde{g}(\nu_{P,a,i})$ for $(P,a,i)<(Q,n,j)$ span $U_{Q,n,j}$, this shows 
$\widetilde{g}(\nu_{Q,n,j})\not\in U_{Q,n,j}$ as desired.
\end{proof}

\begin{lem} \label{lem:t3}
The elements
\[
 \{ \widetilde{g}_0( \nu_{Q',n,j}) : (Q',n,j) \in T \text{ and } \nu_{Q',n,j} \text{ exists }\}
\]
are linearly independent.
\end{lem}

\begin{proof}
For notational convenience, let $T'\subseteq T$ be the subset of triples
of the form $(Q',n,j)$ such that $\nu_{Q',n,j}$ exists, and note that
$\nu_{Q',n,j}\in H^0(X,\Gscr_0)$ as $n \leq (p-1-j)d_{Q'} = \ord_{Q'} (E_j)$.
Let $\omega_{Q',n,j} = \varphi^{-1}_\eta( \nu_{Q',n,j})$; we have $(\tau-1)^j \omega_{Q',n,j} = j! \xi_{Q',n}$ and $(\tau-1)^{j+1} \omega_{Q',n,j} = 0$.

Suppose that there exist scalars $c_{Q',n,j} \in k$ for $(Q',n,j)\in T'$ such that
\[
 \sum_{(Q',n,j)\in T'}  c_{Q',n,j} \omega_{Q',n,j} = \omega' \in H^0(X,\Gscr_{-1}) = H^0(Y,\ker V_Y).
\]
We will show that $c_{Q',n,j}=0$ for all $(Q',n,j)\in T'$ by
descending induction on $j$.  
First observe that
there are no triples of the form $(Q',n,p-1)$ in $T$ (so {\em a fortioti} there are none in $T'$) as $\ord_{Q'}(E_{p-1}+pD_{p-1})=0$.  
Now fix $j$ and suppose $c_{Q',n,i}=0$ for all $n$ whenever $i > j$.  We compute 
\begin{equation}
 (\tau-1)^{j} \omega' = j! \sum_{\{n\ :\ (Q',n,j)\in T'\}} c_{Q',n,j} \xi_{Q',n}
 \label{eq:tautrick}
\end{equation}
since $(\tau-1)^j \omega_{Q',n,i} \neq 0$ implies that $i \geq j$.  As  $\ord_{Q'}(\xi_{Q',n})=-n$, each of the non-zero terms on the right of \eqref{eq:tautrick} has a different, negative valuation at $Q'$, while the left side is {\em regular} at $Q'$ as $ (\tau-1)^{j} \omega'$ is regular at the points above $Q'$ (in fact, everywhere) and $\pi:Y\rightarrow X $ is unramified over $Q'$.
This forces $c_{Q',n,j}=0$ for all $n$ in the sum, as desired.

We conclude that the images of the $\omega_{Q',n,j}$ are linearly independent in the cokernel of the inclusion $H^0(Y,\pi_* \ker V_Y) \simeq H^0(X,\Gscr_{-1}) \into H^0(X,\Gscr_0)$ and the result follows.
\end{proof}

\begin{proof}[Proof of Proposition~\ref{prop:t}]
    By Definition \ref{defn:M}, we have $N(X,\pi):=N_1(X,\pi)+N_2(X,\pi)$
    where $N_1(X,\pi):=\dim_k \Im(\widetilde{g})$ and $N_2(X,\pi):=\dim_k \Im(\widetilde{g}_0)$.
    Let $T_1 \subseteq T$ be the set of triples $(Q,n,j)\in T$ with $Q\in S_j$
    for some $0\le j\le p-1$ such that $\nu_{Q,n,j}$ exists
    and $T_2\subseteq T$ the set of triples $(Q',n,j)\in T$
    such that $\nu_{Q',n,j}$ exists.  
    Thanks to Lemmas \ref{lem:t1} and \ref{lem:t2} we have $\# T_1 \le N_1(X,\pi)$,
    while $\# T_{2} \le N_2(X,\pi)$ due to Lemma~\ref{lem:t3}.
    It follows that 
    $$N(X,\pi)=N_1(X,\pi)+ N_2(X,\pi) \ge  \#T_1 + \#T_2 = \#T - B$$
    where $B$ is the number of triples $(Q,n,j)\in T$ for which $\nu_{Q,n,j}$
    does not exist.
\end{proof}

\begin{remark} 
 Using Corollary~\ref{cor:tango} and the observation that at most $a_X$ of the $\xi_{Q,n}$ do not exist, we will use Proposition \ref{prop:t} to bound $U(X,\pi)$,
 thereby obtaining bounds on $a_Y$. A precise statement will be given in
  Theorem~\ref{thm:cleanbounds}.
\end{remark}

\begin{remark}
 Let $Q \in S$.  This argument uses very limited information about the coefficients of
 a local expansion of $f$ at $Q$.  In particular, it only uses that $f$ has a leading term (a non-zero $t_Q^{-d_Q}$ term).  If there were additional non-zero coefficients, the image of $\widetilde{g}$ could very well be larger, and $N(X,\pi)$ could be strictly larger than $\# T - B$.
 However, Example~\ref{ex:monomial} shows it is not possible to do better in general.  This is the main place the argument loses information.
\end{remark}

\subsection{Explicit Bounds}

We continue to use the notation from \S\ref{sec:asc};
in particular, 
$\pi:Y\rightarrow X$ is a fixed degree-$p$ Artin--Schreier cover
of smooth, projective, and connected curves over $k$
with branch locus $S\subseteq X(\overline{k})$, and $d_Q$ is the unique break in the ramification filtration at the unique point of $Y$ above $Q\in S$.
For nonnegative integers $d,i$ with $p\nmid d$ let $\tau_p(d,i)$ be the number of positive integers $n \le \lfloor i d/p\rfloor $ 
with the property that $-n \equiv m d\bmod p$ for some $m$ with $0 < m \le p-1-i$.

\begin{thm} \label{thm:cleanbounds}
    With notation as above, 
    \begin{align}
        a_Y \le pa_X + \sum_{Q\in S} \sum_{i=1}^{p-1} \left( \left\lfloor\frac{id_Q}{p}\right\rfloor - \left\lfloor\frac{id_Q}{p^2}\right\rfloor -\tau_p(d_Q,i)\right) 
         \label{eq:explicitupper}
    \end{align}
    and for any $j$ with $1\le j\le p-1$
    \begin{equation}
        a_Y\ge \sum_{Q\in S} \sum_{i=j}^{p-1}\left( \left\lfloor \frac{id_Q}{p} \right\rfloor - \left\lfloor \frac{id_Q}{p} - \left(1-\frac{1}{p}\right)\frac{jd_Q}{p}\right\rfloor\right).
        \label{eq:explicitlower}
    \end{equation}
    Moreover, for any nonnegative integers  $d,i$ with $p\nmid d$ and $i \le p-1$,
    \begin{equation}
        \tau_p(d,i) \ge (p-1-i)\left\lfloor  \frac{1}{p}\left\lceil\frac{id_Q}{p}\right\rceil\right\rfloor \ge (p-1-i)\left\lfloor \frac{id_Q}{p^2}\right\rfloor.\label{eq:taubound}
    \end{equation}
\end{thm}

\begin{proof}
 Let $\n(X)$ be the Tango number of $X$, and for $0\le i\le p-1$ let $D_i$ be as 
 above Definition \ref{defn:varphi} (see also Corollary \ref{cor:splittings}).
 Adding in more points to each $D_i$ as needed,
we may without loss of generality assume that $\deg(D_i) > \n(X)$ for all $i$.  Using this assumption,
we will analyze the lower and upper bounds from Proposition~\ref{prop:abstractbound} separately.  
Recall that 
we have fixed an Artin--Schreier equation $y^p-y=f$ giving the extension of function fields $k(Y)/k(X)$ with $f$
minimal in the sense of Definition \ref{def:minimal}; in particular $f$
has a pole of order $d_Q$ at each $Q\in S$, as well as a pole of
some order $d_{Q'}$ that is divisible by $p$ at the fixed point $Q'\in X(k)- S$.
As in Definition \ref{defn:collecteddefs}, for 
$Q\in S'=S\cup\{Q'\}$ and
$0\le i \le p-1$ we set $n_{Q,i}:=\lceil (p-1-i) d_Q/p \rceil$ if $Q\neq Q'$
and $n_{Q',i}:=(p-1-i)d_{Q'}$, and we put $E_i:=\sum_{Q\in S'} n_{Q,i}[Q]$.
We now define
\[
 A_i := \sum_{Q \in S'} \left \lfloor \frac{n_{Q,i}}{p} \right \rfloor [Q], \quad \quad R_i= \sum_{ Q \in S} r(Q,i) [Q] :=  E_i - p A_i.  
\]
Note that $R_i$ is supported on $S$, as $n_{Q',i}$ is a multiple of $p$.
Let $s(Q,i) = r(Q,i) -1$ if $r(Q,i) >0$, and $0$ otherwise.

We will first establish \eqref{eq:explicitupper}.  By proposition~\ref{prop:abstractbound} we have
\begin{equation} \label{eq:abstractupper}
 a_Y \leq \sum_{i=0}^{p-1} \dim_k H^0(X,\ker V_X( F_* (E_i + p D_i))) - N(X,\pi).
\end{equation}
We will understand the first sum using Corollary~\ref{cor:tango}, and $N(X,\pi)$ using Proposition~\ref{prop:t}.  

Fix $i$.   As $\deg(D_i)>n(X)$, the first part of Corollary~\ref{cor:tango} implies that
\begin{gather}
 \dim_k H^0(X,\ker V_X(F_*(E_i + p D_i))) = \dim_k H^0(X,\ker V_X(F_*(p (A_i + D_i) +R_i )) \nonumber \\
 = (p-1) \deg(A_i + D_i) + \sum_{Q \in S}  s(Q,i) \nonumber \\
  = (p-1)\left( \sum_{Q \in S} \left \lfloor \frac{n_{Q,i}}{p} \right \rfloor + \frac{d_{Q'}}{p} (p-1-i) + \# \sup(D_i) \right) + \sum_{Q \in S}  s(Q,i). \label{eq:uppersizes}
\end{gather}

We also want to count elements of the form $(Q,n,i)$ in the set $T$ of Proposition~\ref{prop:t} for fixed $i$ and $Q\in S_i'=S'\cup \sup(D_i)$.  If $Q \in \sup(D_i)$, then there are $p-1$ such elements by definition since $\ord_Q(E_i + p D_i) = p$.  If $Q = Q'$, there are $(p-1) (p-1-i) d_{Q'}/p$ such elements since $\ord_{Q'}(E_i) = (p-1-i) d_{Q'}$.  When $Q\in S$, the number of elements of the form $(Q,n,i)$ in $T$ is precisely $\tau_p(d_Q,p-1-i)$
since the set of positive integers $n \le n_{Q,i}$ with the property that $i+(n-1)d_{Q}^{-1}\equiv m\bmod p$ for some $m$ with $0\le m < i$
is in bijection with the set of positive integers $n' \le \lfloor (p-1-i)d_Q/p \rfloor$ with $-n'\equiv m' d_Q \bmod{p}$ for some $m'$ with $0 < m' \le i$
via $n':= n-1$ and $m' = i-m$.
Putting this together, we find that there are exactly
\begin{equation}\label{eq:exactt}
     \sum_{Q\in S} \tau_p(d_Q,p-1-i) + (p-1) \# \sup(D_i) + (p-1) (p-1-i) \frac{d_{Q'}}{p}
\end{equation}
elements of $T$ of the form $(Q,n,i)$.

Putting \eqref{eq:abstractupper} together with \eqref{eq:uppersizes} and the bound on $N(X,\pi)$ coming from \eqref{eq:exactt}, Proposition~\ref{prop:t}, and the observation that at most $p a_X$ of the $\nu_{Q,n,j}$ do not exist (since at most $a_X$ of the $\xi_{Q,n}$ do not exist),  we conclude that
\begin{equation}
 a_Y \leq p \cdot a_X + \sum_{i=0}^{p-2} \sum_{Q \in S}  (p-1) \left \lfloor \frac{n_{Q,i}}{p}   \right \rfloor +s(Q,i) -\tau_p(d_Q,p-1-i) .\label{eq:almostupper}
\end{equation}
Using the very definition of $s(Q,i)$, one finds the formula
\begin{equation}
    s(Q,i) =n_{Q,i} - p\cdot \left \lfloor \frac{n_{Q,i}}{p}   \right \rfloor  -\left(\left \lceil \frac{n_{Q,i}}{p}   \right \rceil - \left \lfloor \frac{n_{Q,i}}{p}   \right \rfloor\right) = n_{Q,i}-\left \lceil \frac{n_{Q,i}}{p}   
    \right \rceil-(p-1)\cdot\left \lfloor \frac{n_{Q,i}}{p}   \right \rfloor \label{eq:sval}
\end{equation}
Substituting \eqref{eq:sval} into \eqref{eq:almostupper}, re-indexing the sum $i\mapsto p-1-i$ and using the equality $\lceil x\rceil = \lfloor x\rfloor +1$ for $x\not\in \ZZ$ 
gives \eqref{eq:explicitupper}.  

To make this bound more explicit, 
we must bound $\tau_p(d,i)$ from below.
To do this, simply note that for any $0 < m \le p-1-i$ and any interval of length $p$, there is a unique $n$ in that interval such that 
\[
 m d \equiv -n \pmod{p}
\]
and that necessarily $n \not \equiv 0 \pmod{p}$ as $p\nmid d$.  Thus, 
\begin{equation}\label{eq:taupbound}
        \tau_p(d_Q,i) \ge (p-1-i) \cdot  \left\lfloor\frac{1}{p}\left(\left\lfloor \frac{id_Q}{p} \right\rfloor + 1\right)\right\rfloor=
        (p-1-i) \cdot  \left\lfloor\frac{1}{p}\left\lceil \frac{id_Q}{p} \right\rceil \right\rfloor
        \ge (p-1-i)\left\lfloor\frac{id_Q}{p^2}\right\rfloor,
\end{equation}
as $0 \le n \leq \lfloor id_Q/p\rfloor = \lceil id_Q/p\rceil - 1$.  

We now turn to the lower bound \eqref{eq:explicitlower}.  Proposition~\ref{prop:abstractbound} and Definition~\ref{defn:bounds}
give that for each $0 \leq j \leq p-1$,
\begin{equation} \label{eq:lowerbound}
 \dim_k U_j - \sum_{i=0}^j \sum_{Q \in S_j'} c(i,j,Q) \leq a_Y.
\end{equation}
Using Lemma~\ref{lem:vk}, Corollary~\ref{cor:tango} and the assumption that $\deg D_i > n(X)$ for all $i$, we calculate
\begin{align*}
 \dim_k U_j &= \sum_{i=0}^j \dim_k H^0( X,\Omega^1_X(F_*(E_i + p D_i))) \\
 & =\sum_{i=0}^j \left((p-1) \deg(D_i) + (p-1)(p-1-i) \frac{d_{Q'}}{p} +  \sum_{Q \in S} (p-1) \left \lfloor \frac{n_{Q,i}}{p} \right \rfloor +s(Q,i) \right) .
\end{align*}
 Now for $Q \in \on{sup}(D_i)$ we have $c(i,j,Q) =p-1$, while $c(i,j,Q') = (p-1) (p-1-i) d_{Q'}/p$.  If $Q \in S$, then $c(i,j,Q)$ is the  number of integers $n\equiv -1\bmod p$ satisfying $-n_{Q,j} - d_Q (j-i) \leq n < -n_{Q,i}$.  That is,
 \begin{align*}
  c(i,j,Q) &=   \left \lceil \frac{n_{Q,j} + d_Q (j-i)}{p} \right \rceil - \left \lceil \frac{n_{Q,i}}{p} \right \rceil. 
 \end{align*}
Thus, the contributions from $D_i$ and $Q'$ in \eqref{eq:lowerbound} cancel, so the formula (\ref{eq:sval}) for $s(Q,i)$ yields
\begin{equation*} 
  \sum_{i=0}^j  \sum_{Q \in S}\left( n_{Q,i}  - \left \lceil \frac{n_{Q,j} + d_Q (j-i)}{p} \right \rceil  \right)  \leq a_Y
\end{equation*}
for all $0\le j< p-1$. Using the equality 
\[
    \left\lceil \frac{n_{Q,j}+(j-i) d_Q}{p}\right\rceil = \left\lceil \frac{(p-1-i) d_Q}{p} - \left(1-\frac{1}{p}\right)\frac{(p-1-j) d_Q}{p} \right \rceil,
\]
changing variables $i\mapsto p-1-i$ and $j\mapsto p-1-j$, and employing again the equality $\lceil x\rceil = \lfloor x \rfloor + 1$ for $x\not\in \ZZ$
then gives \eqref{eq:explicitlower}.
\end{proof}

\begin{cor} \label{cor:explicit}
Suppose $p$ is odd.
With the hypotheses of Theorem \ref{thm:cleanbounds} we have
\[
  a_Y \leq p \cdot a_X +  \left( \frac{(p-1)(p-2)}{2} + \left(1-\frac{1}{p}\right)^2 \right) \cdot \# S + \left(1-\frac{1}{p}\right)\sum_{Q \in S} \frac{(2p-1)}{6} d_Q
\]
and
\[
  a_Y \geq  \left(1-\frac{1}{p}\right)^2 \left(\sum_{Q \in S} \frac{ (p+1)}{4} d_Q   - \#S\cdot \frac{p}{2}\right).
\]
\end{cor}

\begin{proof}
The upper bound follows easily from \eqref{eq:explicitupper} and the first inequality in \eqref{eq:taubound} by basic properties of the floor function and
the well-known equality
\begin{equation}
 \sum_{i=1}^{n-1} \left\lfloor \frac{id}{n}\right\rfloor = \frac{(n-1)(d-1)}{2}\label{eq:floorsum}
\end{equation}
for any positive and co-prime integers $d,n$.  For the lower bound, we have
\begin{align*}
     \left\lfloor \frac{id_Q}{p} \right\rfloor - \left\lfloor \frac{id_Q}{p} - \left(1-\frac{1}{p}\right)\frac{jd_Q}{p}\right\rfloor
     &\ge
     \frac{id_Q-(p-1)}{p} - \left( \frac{id_Q}{p} - \left(1-\frac{1}{p}\right)\frac{jd_Q}{p}\right)\\
     &=\left(1-\frac{1}{p}\right)\left(\frac{jd_Q}{p}-1\right).
\end{align*}
Summing over $i$ with $j\le i \le p-1$ and then $Q$ and using \eqref{eq:explicitlower} gives the lower bound
\begin{equation}\label{eq:optimizeeq}
     \left(1-\frac{1}{p}\right)\sum_{Q\in S} \left(\frac{d_Q}{p}(p-j)j-(p-j)\right)
\end{equation}
which holds for all $1\le j\le p-1$; we then take $j = \frac{p+1}{2}$.  Remark~\ref{remark:choice} motivates this choice.    
\end{proof}

\begin{remark} \label{remark:choice}
 The choice of $j = \frac{p+1}{2}$ in the proof of Corollary~\ref{cor:explicit} is optimized for $\sum d_Q$ large relative to $\#S$.
 Indeed, \eqref{eq:optimizeeq} is a quadratic function in $j$ which attains its maximum when 
 \begin{equation}\label{jchoice}
     j\approx \frac{p}{2}\left(1-\frac{\#S}{\sum_{Q\in S} d_Q}\right).
 \end{equation}
 Any nearby value of $j$ will give a similar bound.  
 Supposing that $\sum d_Q $ is large relative to $\#S$ gives optimal choice $j= \lceil p/2\rceil = (p+1)/2$.  When all $d_Q$ are small and $p$ is large, one can get a better
 explicit lower bound by choosing a value of $j$ in accordance with \eqref{jchoice}; \emph{cf}. Example~\ref{ex:numerics}.
\end{remark}

\begin{remark}
 For fixed $X$, $p$ and $S$ with all $d_Q$ becoming large, the dominant terms of the lower and upper bounds in Corollary \ref{cor:explicit} are respectively
 \[
 \sum_{Q \in S} d_Q \frac{p}{4}  \quad \text{and} \quad  \sum_{Q \in S} d_Q \frac{p}{3}.
 \]
 On the other hand, the dominant term in the Riemann--Hurwitz formula for the genus of $Y$ is $\sum_{Q\in S} d_Qp/2$,
 so that, for large $d_Q$, the $a$-number is approximately between $1/2$ and $2/3$ of the genus of $Y$. 
\end{remark}

\begin{remark} \label{rmk:char2}
When $p=2$, the statement and proof of the Corollary do not work as written.  If we take $j=1$ in Theorem~\ref{thm:cleanbounds} we obtain
\[
\sum_{Q \in S} \left \lfloor \frac{d_Q}{2} \right \rfloor - \left \lfloor \frac{d_Q}{4} \right \rfloor \leq a_Y \leq 2 a_X + \sum_{Q \in S} \left \lfloor \frac{d_Q}{2} \right \rfloor - \left \lfloor \frac{d_Q}{4} \right \rfloor .
\]
In particular, when $X$ is ordinary (i.e. $a_X=0$) we obtain an exact formula for $a_Y$.  This recovers \cite[Theorem 2]{VolochChar2} (note that the formula there is for the rank of the Cartier operator, and that for $Q_i \in S$, our $d_{Q_i}$ is $2 n_{i}-1$).
\end{remark}

Similarly, Corollary~\ref{cor:anumformulaordinary} will give an exact formula for $a_Y$ when $p$ is odd, $X$ is ordinary, and $d_Q | (p-1)$.  To derive it, we will need to investigate situations when   
 it is possible to derive an exact formula for the quantity 
\begin{equation}
    \tau_p(d):=\sum_{i=0}^{p-1} \tau_p(d,i)
\end{equation}
occurring in the upper bound \eqref{eq:explicitupper}:

\begin{prop}\label{prop:taupformula}
    Let $p>2$ and suppose that $d\equiv d'\bmod p^2$.  Then 
    \begin{equation}
        \tau_p(d) = \tau_p(d') + (d-d')\frac{(p-1)(p-2)}{6p}.
    \end{equation}
    Moreover, $\tau_p(1)=0$ and for $1< d < p$ we have
    \begin{equation}
        \tau_p(d) = u_p(d)\cdot\frac{p}{d} + v_p(d)
    \end{equation}
    where $u_p(d)$ and $v_p(d)$ are the integers $($depending only on $p\bmod d$$)$ given by
     \begin{equation}
        u_p(d):= \sum_{j=0}^{b-1} \sum_{r\in T_j} \left((j+1)d-(b+1)r\right)\qquad
        v_p(d):= \sum_{j=0}^{b-1} \sum_{r\in T_j} \left( \frac{r}{d}-\left\{\frac{a(d-r)}{d}\right\}\right)
    \end{equation}
    where $a:=p \bmod d$ and $b:=a^{-1}\bmod d$ with $0< a,b < d $ and
    \[
        T_j:=\{ r \in \ZZ\ :\ jd/b < r < (j+1)d/(b+1)\}.
    \]
    In particular, for $1< d < p$
    the quantity $\tau_p(d)$ depends only on $d$ and $p\bmod d$.
\end{prop}

Before proving Proposition \ref{prop:taupformula}, let us give some indication of its utility.
For example, if $p\equiv 1\bmod d$, we have $a=b=1$, whence 
\begin{equation*}
    u_p(d) = \sum_{r=1}^{\lfloor d/2\rfloor} (d - 2r) = \left\lfloor\frac{d}{2}\right\rfloor \cdot \left\lfloor\frac{d-1}{2}\right\rfloor =\left\lfloor\frac{(d-1)^2}{4}\right\rfloor 
\end{equation*}
and
\begin{equation*}
    v_p(d) = \sum_{r=1}^{\lfloor d/2\rfloor}\frac{r}{d}-\frac{d-r}{d} =  \sum_{r=1}^{\lfloor d/2\rfloor}\frac{2r-d}{d}=-\left\lfloor\frac{(d-1)^2}{4}\right\rfloor. 
\end{equation*}
Thus, 
\begin{equation} \label{eq:tauformula}
    \tau_p(d) = \left\lfloor\frac{(d-1)^2}{4}\right\rfloor \cdot \frac{p-1}{d}
\end{equation}
whenever $d<p$ and $p\equiv 1\bmod d$.  Another simple example is when $p\equiv -1\bmod d$, so that $a=b=d-1$.  In this case,
$T_j:=\{r\in \ZZ\ :\ jd/(d-1) < r < (j+1)\}$ is the empty set for all $j$, whence $\tau_p(d)=0$. 

\begin{proof}
    The first assertion follows easily from the fact that, as observed immediately prior to \eqref{eq:taupbound},
    for any $0 <  m \le p-1-i$ and any interval of length $p$, there is a unique $n$ in that interval such that 
\[
 m d \equiv -n \pmod{p}
\]
and  necessarily $n \not \equiv 0 \pmod{p}$.  So suppose that $d < p$, and for integers $m,i$, define 
\[
\chi(m,i):=\begin{cases}
    1 & \text{if}\ m \le i\\
    0 & \text{otherwise}
\end{cases}.
\]
We will use the convention that ``$x\bmod y$'' denotes the unique integer $x_0$ with $0\le x_0 < y$ and $x\equiv x_0\bmod y$, and we set 
$a:=p\bmod d$ and $b:=a^{-1}\bmod d$.
By definition, we have
\begin{equation}
    \tau_p(d) = \sum_{i=0}^{p-1} \sum_{\substack{0<j\le n_i(d) \\ j\not\equiv 0 \bmod p}} \chi(-jd^{-1}\bmod p, p-1-i),\label{eq:taupdef}
\end{equation}
where $n_i(d):=\lfloor id/p\rfloor$.  For an integer $r$, observe that $r \le  n_i(d) < d$ if and only if $ i> rp/d$,
and that
$-jd^{-1} \equiv j(bp-1)/d\bmod p$ (note that $d|(bp-1)$).
Taking this into account, we swap the order of summation in \eqref{eq:taupdef}, collecting all terms with $j=r$ together to obtain 
\begin{equation}
         \sum_{r=1}^{d-1} \sum_{i=\lceil rp/d\rceil}^{p-1} \chi(r\frac{bp-1}{d}\bmod p, p-1-i).\label{eq:taupreorder}
\end{equation}
Since $r<d < p$ and $rb\equiv rp^{-1}\not\equiv 0\bmod d$, we have 
\begin{equation*}
    r\frac{bp-1}{d} \bmod p = \frac{(rb\bmod d)\cdot p - r}{d}
\end{equation*}
and moreover $rb\bmod d = rb - \ell d$ for $\ell d/b \le r < (\ell+1)d/b$.
  Breaking the sum over $r$ up into these regions converts \eqref{eq:taupreorder} into
\begin{equation}
         \sum_{\ell=0}^{b-1}\sum_{\substack{ \ell d/b < r \\ r < (\ell+1)d/b}} \sum_{i=\lceil rp/d\rceil}^{ p-1} \chi(\frac{(rb-\ell d)p-r}{d}, p-1-i).\label{eq:taupreorder2}
\end{equation}
Indeed, the quantity $\ell d/b$
is either $0$ or 
{\em non-integral} since $\gcd(b,d)=1$
and $0\le \ell < b$; thus $\ell d/b \le r$ is equivalent to $\ell d/b < r$
since $r$ ranges over all positive integers less than $d$. 
The innermost sum in \eqref{eq:taupreorder2} has value the cardinality of the subset of positive integers
\[
    T(r,\ell):=\left\{i \in \NN\ :\ \frac{(rb-\ell d)p-r}{d}\le p-1-i \le p-1-\left\lceil\frac{rp}{d}\right\rceil = \left\lfloor\frac{(d-r)p}{d}\right\rfloor-1\right\}.
\]
Now $T(r,\ell)$ is empty unless 
\begin{equation}
    \frac{(d-r)p}{d} - 1 > \frac{(rb-\ell d)p - r}{d},\label{eq:kineq}
\end{equation}
in which case 
\begin{align}
    \# T(r,\ell)=\left\lfloor\frac{(d-r)p}{d}\right\rfloor - \frac{(rb-\ell d)p-r}{d}& = \frac{(d-r)(p-a)}{d}+ \left\lfloor \frac{a(d-r)}{d}\right\rfloor - \frac{(rb-\ell d)p-r}{d}\nonumber \\
    & = \left((\ell+1)d-(b+1)r\right)\frac{p}{d} + \left( \frac{r}{d}-\left\{\frac{a(d-r)}{d}\right\}\right).\label{eq:taupsummand}
\end{align}
The inequality \eqref{eq:kineq} is equivalent to $((\ell+1)d - (b+1)r)p > (d-r)$ which, as $p> d > d-r$ is equivalent to $(\ell+1)d > (b+1)r$,
or what is the same thing, $r < (\ell+1)d/(b+1)$.
In other words, the contribution from the innermost sum \eqref{eq:taupreorder} for $r$ in the ranges $(\ell+1)d/(b+1) < r < (\ell+1)d/b$ is zero
so that these values of $r$ in the middle sum may be omitted, and the expression \eqref{eq:taupsummand} then substituted for the innermost sum,
which yields the claimed formula for $\tau_p(d)$ when $d<p$.  
\end{proof}

Using the bounds of Theorem \ref{thm:cleanbounds}, we are able to generalize the main result of \cite{fp13}, which provides an $a$-number {\em formula}
for branched $\ZZ/p\ZZ$-covers of $X=\PP^1$ with all ramification breaks $d$ dividing $p-1$, to the case of
{\em arbitrary} ordinary base curves $X$:

\begin{cor}\label{cor:anumformulaordinary}
    Let $\pi:Y\rightarrow X$ be a branched $\ZZ/p\ZZ$-cover with $a_X=0$, and suppose $p$ is odd.  If $d_Q$ divides $p-1$ for every branch point $Q$, then
    \begin{equation}
        a_Y = \sum_{Q} a_Q\qquad\text{where}\qquad a_Q:=\frac{(p-1)}{2}(d_Q-1) - \frac{p-1}{d_Q}\left\lfloor \frac{(d_Q-1)^2}{4}\right\rfloor .\label{eq:aYformula}
    \end{equation}
\end{cor}

\begin{proof}
 We will compute the upper and lower bounds for $a_Y$ given by \eqref{eq:explicitupper} and \eqref{eq:explicitlower}, and show that these bounds coincide
 when $d_Q|(p-1)$ for all $Q$ and $a_X=0$, and agree with the stated formula.  Using the hypothesis $a_X=0$ together with the explicit formula for $\tau_p(d)$ when $d\equiv 1\bmod p$
 provided by \eqref{eq:tauformula}, the upper bound \eqref{eq:explicitupper}
  becomes
 \begin{equation}
     a_Y \leq    \sum_{Q \in S}  \sum_{i=1}^{p-1}\left( \left \lfloor \frac{i d_Q}{p} \right \rfloor - \left \lfloor \frac{i d_Q}{p^2} \right \rfloor \right)  -\sum_{Q\in S} \frac{p-1}{d_Q}\left\lfloor \frac{(d_Q-1)^2}{4}\right\rfloor .
     \label{eq:aYupper}
 \end{equation}
 Now $d_Q < p$ for all $Q$, so $\left \lfloor \frac{i d_Q}{p^2} \right \rfloor = 0$ for $1\le i \le p-1$.  Using \eqref{eq:floorsum}, we conclude that
 \begin{equation}\label{eq:aYfirstsum}
      \sum_{Q \in S}  \sum_{i=1}^{p-1}\left( \left \lfloor \frac{i d_Q}{p} \right \rfloor - \left \lfloor \frac{i d_Q}{p^2} \right \rfloor \right)  = \frac{(p-1)}{2}(d_Q-1)   .
 \end{equation}
Combining \eqref{eq:aYupper} and \eqref{eq:aYfirstsum} gives the formula \eqref{eq:aYformula} as the upper bound of $a_Y$.
 It remains to prove that $a_Y$ is also bounded below by the same quantity.
 
 Set $j:=(p+1)/2$, and write $d_Q \cdot f_Q = p-1$.  For $i < p$, we see that
 \[
 \left \lfloor \frac{i d_Q}{p} \right \rfloor = \begin{cases}
 \lfloor i /f_Q \rfloor  & \text{ when } f_Q \nmid  i \\
 \lfloor i /f_Q \rfloor  - 1 & \text{ when } f_Q | i.
 \end{cases}
 \]
 Likewise we can check that
 \[
 \left\lfloor \frac{id_Q}{p} - \left(1-\frac{1}{p}\right)\frac{jd_Q}{p}\right\rfloor = \left\lfloor \frac{i-j}{f_Q} \right \rfloor.
 \]
 The inner sum from  \eqref{eq:explicitlower} becomes
 \[
 \sum_{i = (p+1)/2}^{p-1} \left \lfloor \frac{i}{f_Q} \right \rfloor  - \sum_{i=0}^{(p-3)/2} \left \lfloor \frac{i}{f_Q} \right \rfloor  - \# \bigg \{ (p+1)/2 \leq i \leq p-1 : f_Q |i \bigg \}.
 \]
The term in the first sum can be rewritten as $\lfloor d_Q/2 + (i - (p-1)/2) \rfloor$.  When $d_Q$ is even, $d_Q/2$ can be removed from the floor function, allowing cancellation with the second sum giving that the inner sum from  \eqref{eq:explicitlower} is
\[
 \frac{(p-1) d_Q}{4} .
\]
This equals $a_Q$ from \eqref{eq:aYformula} when $d_Q$ is even.  When $d_Q$ is odd, a similar argument removing $\lfloor d_Q/2 \rfloor = \frac{d_Q-1}{2}$ from the first sum shows that the lower bound is 
\[
\lfloor d_Q/2 \rfloor \frac{p-1}{2} + \frac{f_Q \cdot \lfloor d_Q/2 \rfloor}{2} - f_Q /2 =
f_Q \cdot \frac{(d_Q+1)(d_Q-1)}{4}.
\]
Again, this matches the formula for $a_Q$.  Summing over $Q$ completes the proof.
\end{proof}

\begin{remark}
    Note that the proof shows that one has the alternative expression
     \[
     a_Q = \begin{cases}
         \frac{(p-1)d_Q}{4} & d_Q \text{ is even} \\
         \frac{(p-1) (d_Q +1)(d_Q-1)}{4 d_Q} & d_Q \text{ is odd.}
       \end{cases},
   \]
   for $a_Q$ as in \eqref{eq:aYformula}; {\em cf.} \cite[Theorem 1.1]{fp13}.
   
\end{remark}

\subsection{Unramified Covers} \label{sec:unramified}

Now suppose that $\pi : Y \to X$ is an unramified degree-$p$ Artin-Schreier cover.  (No such covers exist when $X = \PP^1$.)  Theorem~\ref{thm:cleanbounds} applies with $S = \emptyset$, but the analysis is not optimized for this situation.  In particular, the argument uses the dimension of $H^0(X, \ker V_X)$ (the $a$-number) but no extra information about the dimension of $H^0(X, \ker V_X(F_* D))$ for other divisors $D$.  This is not an issue when $\pi$ has a substantial amount of ramification, but would be essential to obtaining sharper bounds for unramified covers.  We now illustrate this.

In this setting, the bounds of Theorem~\ref{thm:cleanbounds} become
\[
0 \leq a_Y \leq p \cdot a_X.
\]
On the other hand, since $E_0 = 0$ the trivial bounds are
\[
a_X \leq a_Y \leq p \cdot (g_X -  f_X) .
\]
Note that $p \cdot (g_X -f_X ) \geq p \cdot a_X$, so our upper bound is an improvement, but the trivial lower bound is better!

Any of these bounds are enough to re-derive the following fact (which is typically deduced from the Deuring-Shafarevich formula):

\begin{cor}
Let $\pi : Y \to X$ be an unramified degree-$p$ Artin-Schreier cover.  Then $Y$ is ordinary if and only if $X$ is ordinary.
\end{cor}

The trivial lower bound is better because we lost some information in our applications of Tango's theorem.  In particular, in the proof of Theorem~\ref{thm:cleanbounds} the first step is to add more points to the auxiliary divisors $D_i$ (if needed) so that $\deg(D_i) > \n(X)$ for all $i$ and we can apply Corollary~\ref{cor:tango} to exactly calculate 
\[
\dim_k H^0(X,\ker V_X(F_* (E_i + p D_i))) = \dim_k H^0 (X,\ker V_X(F_*(p D_i)) = (p-1) \deg(D_i).
\]
The trivial bounds comes from the inclusion $\ker V_X \into \pi_* \ker V_Y$ given by $\omega \mapsto \omega y^0$, which would correspond to taking $j=0$ in Definition~\ref{defn:bounds} and being allowed to take $D_0 = 0$.  If that were the case, $S$ is empty and the lower bound is 
\[
\dim U_0 - \sum_{Q \in S_0} c(0,0,Q) = \dim H^0(X,\ker V_X( F_* 0 )) = a_X
\]
which matches the trivial lower bound.  
However, our analysis uses Corollary~\ref{cor:tango}, which requires a sufficiently large $D_0$. In that case,
\[
\dim U_0 - \sum_{Q \in S_0} c(0,0,Q) = \dim H^0(X,\ker V_X(F_* D_0)) = (p-1) \deg(D_0) - \sum_{Q \in D_0} (p-1) = 0
\]
as $c(0,0,Q) = (p-1)$ for $Q \in D_0$.  Thus the lower bound of Theorem~\ref{thm:cleanbounds} is zero.  
(In Theorem~\ref{thm:cleanbounds}, we re-indexed so there we are taking $j= p-1$.)

The technical requirements of using sufficiently large divisors $D_i$ in order to apply Tango's theorem are not an issue when $\pi$ is ``highly ramified'' in the sense that we do not need to increase the degree of the $D_i$ in order to apply Corollary~\ref{cor:tango} to compute the dimension of $H^0(X,\ker V_X(F_*(E_i + p D_i)))$.  
Once we have reached that point, increasing $D_i$ further does not change the bound: the dimension of this space increases, but the increase is canceled by additional $c(i,j,Q)$ terms for $Q \in D_i$.  This is why the divisor $D_i$ makes no appearance in Theorem~\ref{thm:cleanbounds}.

\section{Examples}\label{sec:examples}
As always, let $\pi : Y \to X$ be a degree-$p$ Artin-Schreier cover of curves.  We give some examples of the bounds given by Theorem~\ref{thm:cleanbounds} and  the trivial bounds 
\[
 \dim_k H^0(X,\ker V_X(F_* E_0 )) \leq a_Y \leq p \cdot g_X - p \cdot f_X + \sum_{Q \in S} \frac{p-1}{2} (d_Q-1)
\]
discussed in Remark~\ref{rmk:trivialbounds}.
When $p=3$, one checks that the lower bound $L(X,\pi)$ {\em coincides} with the trivial lower bound.  Outside of this special case, the bounds in Theorem~\ref{thm:cleanbounds} are {\em always} better
than the trivial bounds, and often sharp in the sense that there are degree-$p$ Artin--Schreier covers $\pi:Y\rightarrow X$ with $a_Y$ realizing our bounds in many cases.
We will give a number of examples illustrating these features.   Magma programs which do the calculations in the following examples are available on the authors' websites. 

\subsection{The Projective Line}
Suppose $X = \PP^1$. Then $\n(X) = -1$, $g_X = 0$, $a_X=0$, and $f_X=0$.
Remark~\ref{remark:p1} shows that we may choose $f\in k(X)$
so that the extension of function fields $k(Y)/k(X)$
is given by adjoining the roots of $y^p-y=f$ with $f\in k(X)$
regular outside the branch locus $S$ of $\pi$, and we may further assume
that $f$ has a pole of order $d_Q$ at each $Q\in S$,
where $d_Q$ is the unique break in the lower-numbering ramification filtration above $Q$.

For $Q\in S$ put $n_{Q,i} = \left \lceil (p-1-i) d_Q/p \right \rceil$; note that $n_{Q,0} = d_Q - \left \lfloor d_Q/p \right \rfloor$.  The trivial bounds are
\[
 \sum_{Q \in S}\left( n_{Q,0} - \left\lceil\frac{n_{Q,0}}{p}\right\rceil\right)  \leq a_Y \leq  \sum_{Q \in S} \frac{p-1}{2} (d_Q -1).
\]

\begin{example} \label{ex:numerics}
 Let $p=13$, and suppose $f$ has a single pole.  Table~\ref{tab:p13} shows the trivial upper and lower bounds, as well as the bounds from Theorem~\ref{thm:cleanbounds} for various values of $d_Q$.  
 When $d_Q >4$, an optimum value of $j$ to use in the lower bound turns out to be $\frac{p+1}{2} = 7$.
\begin{table}[t]
\begin{tabular}{c|lllllllllllllllll}
$d_Q=$ & $1$ & $2$ & $3$ & $4$ & $5$ & $6$ & $7$ &
$8$ & $9$ & $10$ & $11$ & $12$ & $14$ & $15$
& $32$ & $128$ & $1024$ \\
\hline
Trivial Lower &
$0$ & $1$ & $2$ & $3$ & $4$ & $5$ & $6$ &
$7$ & $8$ & $9$ & $10$ & $11$ & $12$ & $12$
& $27$ & $109$ & $873$  \\
$L(\PP^1,\pi)$ &
$0$ & $6$ & $8$ & $12$ & $15$ & $18$ & $21$
& $24$ & $26$ & $30$ & $33$ & $36$ & $42$ &
$45$ & $96$ & $382$ & $3054$ \\
$U(\PP^1,\pi)$&
$0$ & $6$ & $8$ & $12$ & $16$ & $18$ & $36$
& $30$ & $34$ & $36$ & $38$ & $36$ & $78$ &
$60$ & $120$ & $488$ & $3936$ \\
Trivial Upper &
$0$ & $6$ & $12$ & $18$ & $24$ & $30$ & $36$
& $42$ & $48$ & $54$ & $60$ & $66$ & $78$ &
$84$ & $186$ & $762$ & $6138$ \\
\end{tabular}
\caption{Bounds for a single pole with $p=13$}
\label{tab:p13}
\end{table}
Notice that our bounds are substantially better than the trivial bounds.
\end{example}

\begin{example} \label{ex:selectpolys}
 Using Magma \cite{magma} or a MAPLE program from Shawn Farnell's thesis \cite{farnell10}, we can compute the $a$-number for covers of $\PP^1$.  For example, let $p=13$ and suppose $f$ has a single pole of order $7$.  Our results show that the $a$-number of the cover is between $21$ and $36$.  Table~\ref{tab:polynomials} lists the $a$-numbers for some choices of $f$, and shows that our bounds are sharp in this instance.
  \begin{table}[t] 
 \begin{tabular}{c|c}
  Polynomial & $a_Y$ \\
  \hline
  $t^{-7} + 2 t^{-6} + 7 t^{-5}$ & $21$ \\
  $t^{-7} + t^{-2} + t^{-1}$ & $23$ \\
  $t^{-7} + 8 t^{-2}$ & $24$\\
  $t^{-7}+t^{-1}$ & $27$ \\
  $t^{-7}$ & $36$
 \end{tabular}
\caption{a-Numbers for some select Artin-Schreier curves}  
\label{tab:polynomials}
 \end{table}
\end{example}

\begin{example} 
Let us generalize Example~\ref{ex:keyexample} and Example~\ref{ex:followup}.  
As before, let $p=5$ and $X=\PP^1$,
but now consider covers $Y_{A,B}$ of $X$ given by $y^5 - y =  f_{A,B} = t^{-3} + A t^{-2} + B t^{-1}$.  Then $Y_{A,B}\rightarrow X$
is branched only over $Q:=0$ and has $d_Q=3$, so our bounds
are $3 \leq a_{Y_{A,B}} \leq 4$.  

Recall the maps $g'_i$ depend on the cover via the choice $f_{A,B}$.  It is clear that the elements, $\nu_{0,2}, \nu_{0,3}$, and $\nu_{1,2}$ lie in $H^0(\PP^1, \Gscr_1)$.  In fact, so does $\nu_{2,2}$
as $g'_2(\nu_{2,2})=0$: indeed, recalling 
 \S\ref{sec:lcrp}, we see that $g'_2(\nu_{2,2})$ records the coefficient of $t^{-6} dt$ in
\[
- 2 \cdot (t^{-2} dt) \cdot -( t^{-3} + A t^{-2} + B t^{-1})
\]
which is visibly zero.

To compute the $a$-number (equal to $\dim H^0(\PP^1,\Gscr_0)$ in this case), it therefore suffices to understand $\ker H^0(g'_1) = H^0(\PP^1,\Gscr_0)$.  It is clear that $\nu_{0,2}$ and $\nu_{0,3}$ lie in $H^0(\PP^1,\Gscr_0)$.  An identical calculation to the one above shows that $g'_1(\nu_{1,2})=0$ as well.  Finally, we see that $g'_1(\nu_{2,2})$ records the coefficient of $t^{-6} dt$ in
\[
 - t^{-2} (t^{-3} +A t^{-2} + B t^{-1} )^2 dt = - (t^{-8}+2 A t^{-7} + (A^2 + 2 B) t^{-6} + \ldots ) dt.
\]
Thus $g'_1(\nu_{2,2})=0$ when $A^2 + 2 B =0$, and is non-zero otherwise.  In particular, $a_{Y_{A,B}} =3$ when $A^2 + 2B \neq 0$ and $a_{Y_{A,B}}=4$ when $A^2 + 2B =0$.  
This generalizes Example~\ref{ex:followup}, again showing 
why the $a$-number of the cover cannot depend only on the $d_Q$ and must incorporate finer information.

The set $T$ in Proposition~\ref{prop:t} is an attempt to produce differentials not in the kernel of some $g'_i$.  It only uses the leading terms of powers of $f_{A,B}$ (since those are the only terms guaranteed to be nonzero).  For the cover $Y_{0,0}$, the leading term is the only term, and there are no differentials not in the kernel of some $g'_i$.  This is reflected in the fact that $T$ is empty in this case.  When $A^2 + 2B \neq 0$, the $a$-number was smaller because there was a differential not in $\ker H^0(g'_1)$; this relied on the non-leading terms of $f_{A,B}$.
\end{example}

\begin{example}
Let us now consider an example with multiple poles.  Let $p=5$, and suppose that $f$ has two poles of order $7$ (at $\infty$ and $-1$).  Then our bounds say that $14 \leq a_Y \leq 16$.  Computing the $a$-number for a thousand random choices of $f$ defined over $\FF_5$ subject to the constraint on the poles, $942$ of them had $a$-number $14$, $41$ had $a$-number $15$, and $1$ had $a$-number $16$.  The $f$ giving $a$-number $16$ was 
\[
 {t}^{7}+3{t}^{6}+{t}^{4}+{t}^{3}+2{t}^{2}+t+3 \left( t+1
 \right) ^{-1}+3 \left( t+1 \right) ^{-2}+4 \left( t+1 \right) ^{-
3}+4 \left( t+1 \right) ^{-4}+4 \left( t+1 \right) ^{-6}+3
 \left( t+1 \right) ^{-7}.
\]
It would be interesting to understand the geometry of the relevant moduli space of Artin-Schreier covers with prescribed ramification.  In this situation, these numerical examples seem to suggest that the locus of curves with $a$-numbers greater than or equal to $15$ (respectively $16$) is of codimension two in the locus with $a$-number greater than or equal to $14$ (respectively $15$).
\end{example}

\begin{example} \label{ex:monomial} 
 For general $p$, take $f(t) = t^{-d}$ with $p \nmid d$.  We will show that this family achieves our upper bound.  As discussed in \cite[Remark 2.1]{fp13} (which extracts the results from \cite{pries05}), the resulting $a$-number is
 \[
  \sum_{b=0}^{d-2} \min\left(h_b,\left \lfloor \frac{p d - bp-p-1}{d} \right \rfloor\right)
 \]
where $h_b$ is the unique integer in $[0,p-1]$ such  that $h_b \equiv (-1 - b) d^{-1} \pmod{p}$.  Note that if $b \equiv -1 \pmod{p}$ then $h_b=0$.  This counts the number of elements in the set
\[
 T': = \left \{ (b,j) : 0 \leq b \leq d-2,\,\, 0 \leq j d \leq p( d-b-1) -d-1 ,\, \, j < h_b \right\}.
\]
On the other hand, our upper bound is $\displaystyle \bigoplus_{i=0}^{p-1} \dim_{\FF_p} H^0(\PP^1,\ker V_{\PP^1}(F_* E_i)) - \# T$.  For $0 \leq i \leq p-1$, the differentials $t^{-n} dt$ with  $0 < n \leq n_{Q,i} = \left \lceil (p-1-i) d/p \right \rceil$
and $n \not \equiv 1 \pmod{p}$ form a basis for $H^0(\PP^1,\ker V_{\PP^1}(F_* E_i))$.   The condition $n \leq n_{Q,i}$ can be expressed as $ p n \leq (p-1-i) d+(p-1)$, or equivalently
\[
  i \cdot d \leq (p-1) d + (p-1) - pn = p(d-n + 1) -1  -d.
\]
Assigning to each such basis element the triple $(Q,n,i)$, 
note that $(Q,n,j)$ {\em does not} lie in $T$ if there does not exist an integer $m \in [0,j]$ such that $m \equiv j - (n-1)d_Q^{-1} \pmod{p}$: this can be rephrased as $j <h_{n-2}$.  Thus our upper bound 
is the size of the set
\[
 T'': = \left\{ (n,j) : 2 \leq n \leq d, \, 0\leq j \cdot d \leq  p(d-n + 1) -1  -d ,\,  j< h_{n-2} \right\}.
\]
(The condition $n \not \equiv 1 \pmod{p}$ is implicit, as in that case $h_{n-2} =0$.)  But $T'$ and $T''$ have the same size: there is a bijection given by taking $b = n-2$.  Thus the covers given by $f(t) = t^{-d}$ for $p \nmid d$ realize the upper bound for a cover ramified at a single point.
\end{example}

\subsection{Elliptic Curves}  

We now suppose that $E$ is the elliptic curve over $\FF_p$ with affine equation $y^2 = x^3 - x$
(recall that $p>2$).  Of course, $g_E=1$ and it is not hard to compute that the Tango number of $E$ is $n(E)=0$
and that $E$ is ordinary (so $a_E=0$ and $f_E= 1$) when $p\equiv 1\bmod 4$
and supersingular ($a_E=1$ and $f_E=0$) when $p\equiv 3\bmod 4$.  In this simple case, we can say more than what Tango's theorem tells us.

\begin{lem}\label{lem:ellVexact}
    Let $D=\sum n_Q\cdot [Q] > 0$ be a divisor on $E$ with $n_Q \ge 2$ for some $Q$,
    and set $D':=\sum \lceil n_Q/p \rceil \cdot [Q]$.
    Then 
    \[
        V_E: H^0(E,\Omega^1_E(D)) \rightarrow H^0(E,\Omega^1_E(D'))
    \]
    is surjective.
\end{lem}

\begin{proof}
   The Lemma is an immediate consequence of Corollary \ref{cor:tango} when $n_Q\ge p$ for some $Q$
   or when $a_E=0$,
   so we suppose that $n_Q < p$ for all $Q$  and $p\equiv 3\bmod 4$.  
   A straightforward induction on the size of the support of $D$
   reduces us to the case that $D=n\cdot [Q]$ with $n\ge 2$, and then as $D' = [Q]$ it suffices to treat the case $n=2$.
   Suppose that $Q$ is not the point $P$ at infinity on $E$.
   Consider the commutative diagram with exact rows
   \begin{equation*}
      \xymatrix{
        0 \ar[r] & {H^0(E,\Omega^1_E(2[P]))} \ar[r]\ar[d]^-{V} & {H^0(E,\Omega^1_E(2[P]+2[Q]))} \ar[r]\ar[d]^-{V} & {(k[x_Q]/x_Q^2) x_Q^{-2} dx_Q }\ar[r]\ar[d]^-{V} & 0\\
        0 \ar[r] & {H^0(E,\Omega^1_E([P]))} \ar[r] & {H^0(E,\Omega^1_E([P]+[Q]))} \ar[r] & {(k[x_Q]/x_Q) x_Q^{-1} dx_Q }\ar[r] & 0.
      } 
   \end{equation*}
   The local description of $V$ shows that the right vertical map is surjective, and it therefore 
   suffices to prove the lemma in the special case $D=2[P]$.  Now $\{dx/y, x dx/y\}$ is a basis
   of $H^0(E,\Omega^1_E(2[P]))$ and one calculates 
   \[
    V\left(x \frac{dx}{y}\right) = \frac{1}{y}V(x(x^3-x)^{(p-1)/2} dx) = (-1)^{(p+1)/4}\binom{\frac{p-1}{2}}{\frac{p-3}{4}} \frac{dx}{y} \neq 0.
   \]
   Since $V(dx/y)=0$, the image of $V$ is $1$-dimensional.  But $H^0(E,\Omega^1_E([P])) = H^0(E,\Omega^1_E)$
   because the sum of the residues of a meromorphic differential on a smooth projective curve is zero, and this
   space therefore has dimension one as well; the Lemma follows.
\end{proof}

We will first consider Artin-Schreier covers $Y\rightarrow E$ with defining equation of the form
\begin{equation*}
    z^p - z = f
\end{equation*}
with $f\in H^0(E,\O_E(d_P\cdot [P]))$ having a pole of exact order $d_P$ at $P$, for $P$ the point at infinity on $E$, where $d_P\ge 2$ is prime to $p$.
Such covers are branched only over 
$P$ with unique break in the ramification filtration
$d_P$.  Conversely, every $\ZZ/p\ZZ$-cover of $E$ that is defined
over $\FF_p$ and is
ramified only above $P$ with unique ramification break $d_P$ occurs this way by Lemma~\ref{lem:minimality}.

\begin{example}
For $p=5,7$ we compare the bounds given by Theorem~\ref{thm:cleanbounds} with the trivial upper and lower bounds, which are
    \begin{equation*}
         \left\lceil\frac{p-1}{p}d_P \right\rceil -\left\lceil\frac{p-1}{p^2}d_P \right\rceil \le a_Y \le p \cdot a_E + \frac{p-1}{2} (d_P-1)
    \end{equation*}
    thanks to Lemma \ref{lem:ellVexact} and the fact that $\lceil (p-1)d_P/p \rceil = d_P - \lfloor d_P/p\rfloor \ge 2$.
We summarize these computations in Tables \ref{tab:ellp5} and \ref{tab:ellp7}.
\end{example}

\begin{table}[t]
\begin{tabular}{c|lllllllllllllll}
$d_P=$ & $2$ & $3$ & $4$ & $6$ & $7$ & $8$ & $9$ &
$11$ & $12$ & $13$ & $14$ & $16$ & $32$ & $128$
& $1024$  \\
\hline
Trivial Lower &
$1$ & $2$ & $3$ & $4$ & $4$ & $5$ & $6$ &
$7$ & $8$ & $8$ & $9$ & $10$ & $20$ & $82$
& $656$   \\
Lower &
$2$ & $3$ & $4$ & $6$ & $7$ & $8$ & $9$
& $10$ & $12$ & $12$ & $14$ & $15$ & $31$ &
$123$ & $984$  \\
Upper &
$2$ & $4$ & $4$ & $10$ & $8$ & $10$ & $10$
& $14$ & $16$ & $14$ & $16$ & $20$ & $38$ &
$154$ & $1230$  \\
Trivial Upper &
$2$ & $4$ & $6$ & $10$ & $12$ & $14$ & $16$
& $20$ & $22$ & $24$ & $26$ & $30$ & $62$ &
$254$ & $2046$  \\
\end{tabular}
\caption{Bounds for a single pole at the point at infinity with $p=5$}
\label{tab:ellp5}
\end{table}

\begin{table}[t]
\begin{tabular}{c|llllllllllllllll}
$d_P=$ & $2$ & $3$ & $4$ & $5$ & $6$ & $8$ & $9$ &
$10$ & $11$ & $12$ & $13$ & $15$ & $16$ & $32$
& $128$ & $1024$  \\
\hline
Trivial Lower &
$1$ & $2$ & $3$ & $4$ & $5$ & $6$ & $6$ &
$7$ & $8$ & $9$ & $10$ & $11$ & $12$ & $24$
& $94$ & $752$   \\
Lower &
$3$ & $4$ & $6$ & $7$ & $9$ & $12$ & $13$
& $15$ & $16$ & $18$ & $20$ & $21$ & $24$ &
$47$ & $188$ & $1505$  \\
Upper&
$10$ & $11$ & $16$ & $15$ & $16$ & $28$ & $23$
& $24$ & $27$ & $30$ & $29$ & $37$ & $40$ &
$68$ & $244$ & $1910$ \\
Trivial Upper &
$10$ & $13$ & $16$ & $19$ & $22$ & $28$ & $31$
& $34$ & $37$ & $40$ & $43$ & $49$ & $52$ &
$100$ & $388$ & $3076$  \\
\end{tabular}
\caption{Bounds for a single pole at the point at infinity with $p=7$}
\label{tab:ellp7}
\end{table}

\begin{example}
  Again for $p=5,7$ and select $d_P$ in the tables above, we have computed $a_Y$
  for several thousand randomly selected $f\in H^0(E,\O_E(d_P\cdot [P]))$ (working over $\FF_p$).  In Tables \ref{tab:ellp5actual} and \ref{tab:ellp7actual},
  we record the values of $a_Y$ that we found, as well as a function $f$ that produced it.
 These values of $a_Y$ should be compared to the bounds in Tables \ref{tab:ellp5} and \ref{tab:ellp7}.
\end{example}

\begin{table}[t]
\renewcommand{\arraystretch}{1.2}
\begin{tabular}{||c|c||c|c||c|c||c|c||}
    \hline
     \multicolumn{2}{||c||}{$d_P=3$}  & \multicolumn{2}{c||}{$d_P=6$}  & \multicolumn{2}{c||}{$d_P=8$}  & \multicolumn{2}{c||}{$d_P=11$} \\
     \hline
      $a_Y$ & $f$ & $a_Y$ & $f$ & $a_Y$ & $f$ & $a_Y$ & $f$  \\
      \hline
      $3$ & $-2y-x$ &  $6$  & $x^3-2x^2$ & $8$ &  $x^4$    &    $10$    & $(x^4-x)y-2x^5-x^4$       \\
      $4$ & $2y$   &   $7$ & $x^3+x$    & $9$ & $2x^4-x^2+x$      &     $11$        &  $(-2x^4+2x+1)y+2x^5+x^2$       \\
      \cline{1-2}
\multicolumn{2}{c||}{}& $8$& $2x^3-xy$  & $10$ &    $-x^4-2xy-2x^2$   &  $12$       &    $(-2x^4-x^2-x)y-2x^5-1$   \\ 
\cline{5-6}
\multicolumn{2}{c||}{}& $9$& $-x^3+x$   & \multicolumn{2}{c||}{}       &    $13$         &     $(2x^4-x^2+2x)y-2x^5-x^3-x$      \\
\cline{3-4}\cline{7-8}
\end{tabular}
\caption{$a$-numbers of $\ZZ/p\ZZ$-Covers of $E$ branched only over $P$ with $p=5$}
\label{tab:ellp5actual}
\end{table}

\begin{table}[t]
\renewcommand{\arraystretch}{1.2}
\begin{tabular}{||c|c||c|c||c|c||}
    \hline
     \multicolumn{2}{||c||}{$d_P=6$}  & \multicolumn{2}{c||}{$d_P=8$}   & \multicolumn{2}{c||}{$d_P=10$} \\
     \hline
      $a_Y$ & $f$ & $a_Y$ & $f$ & $a_Y$ & $f$  \\
      \hline
      $10$ & $x^3+xy$ &  $12$  & $3xy+2x^4-2x^3$ &    $15$    & $-3x^5-3x^4-x^2+3x$       \\
      $11$ & $x^3+y$   &   $13$ & $2xy-3x^4-1$ &     $16$        &  $2x^5-3x^3+x^2+3x$       \\
      
      $12$ & $x^3$         & $14$& $2xy-2x^4+x-3$ &  $17$       &    $-x^5+x^3+2x^2-1$   \\ 
      $14$ &  $2x^3+x$       & $15$& $-x^4+2x^3+2x^2-3x-3$  & $18$ &$2xy-x^5+x^3+2x^2-x-1$      \\
      \cline{1-2}\cline{5-6}
      \multicolumn{2}{c||}{}       & $16$& $(x^2+x)y+3x^4-2x^3-3x^2-x+1$   &  \multicolumn{2}{c}{}\\
\cline{3-4}
\end{tabular}
\caption{$a$-numbers of $\ZZ/p\ZZ$-Covers of $E$ branched only over $P$ with $p=7$}
\label{tab:ellp7actual}
\end{table}

\begin{remark}
    Larger values of $a_Y$
    are more rare.  For example, when $p=7$, $S=\{P\}$
    and $d_P=6$, among {\em all} $100842=6\cdot 7^5$ functions  $f\in H^0(E,\O_E(d_P\cdot [P]))$
    with a pole of exact order $6$ at $P$,
    there are 86436 ($=85.71 \%$) with $a_Y=10$,
    11760 ($=11.66\%$) with $a_Y=11$, 2562 ($=2.54\%$) with $a_Y=12$ and 84 ($=0.08\%$) with $a_Y=14$.
    Curiously, none had $a_Y=13$.
    In the spirit of \cite{CEZB}, it would be interesting to investigate the limiting distribution of $a$-numbers in branched $\ZZ/p\ZZ$-covers
    of a fixed base curve with fixed branch locus, as the sum of the ramification breaks tends to infinity.  Recent work of Soumya Sankar investigates the limiting distribution of non-ordinary ($a$-number greater than $0$) covers of the projective line \cite{sankar}.

    Note also that although $d_P=6$ divides $p-1$ when $p=7$, the $a$-number of $Y$ can be $10,11,12$ or $14$;
    in particular, the ordinarity hypothesis in Corollary \ref{cor:anumformulaordinary} is necessary.
\end{remark}

\begin{example}
    We now work out some examples with $\pi: Y\rightarrow E$ branched at exactly two points.  
    As before, let $P$ be the point at infinity on $E$ and $Q$ be the point $(0,0)$.  
    For $p=5$ we considered Artin-Schreier covers $Y$ of $E$ branched only over $P,Q$
    with $d_P=6$ and $d_Q=4$.  Our bounds are $10 \le a_Y \le 14$, and among a sample of
    10001 functions $f\in H^0(E,\O_E(d_P\cdot [P] + d_Q\cdot [Q]))$ with a pole of exact order
    $d_{\star}$ at $\star=P,Q$, we found $8021$ with $a_Y=10$, $1818$ with $a_Y=11$,
    149 with $a_Y=12$, and $13$ with $a_Y=13$.  One of the 13 functions we found giving $a$-number 
    $13$ was
    \[
       f=-\frac{1}{xy} + \frac{x^5-x^3+2x^2-2x+1}{x^2}.
    \]
    
    Similarly, with $p=7$ and $d_P=6$, $d_Q=8$ our bounds are $21 \le a_Y \le 37$.
    Among a sample of 5001 random functions satisfying the required constraints we found
    $4318$ with $a_Y=21$, $668$ with $a_Y=22$, $14$ with $a_Y=23$, and $1$ with $a_Y=24$. The function
    producing a cover $Y$ with $a$-number 24 was
    \[
       f=\frac{-x^4-3x^3+2x^2-x+1}{x^4y} + \frac{-2x^7+3x^5-3x^4+x^3+x^2+x+3}{x^4}.
    \]
\end{example}


\newcommand{\etalchar}[1]{$^{#1}$}
\providecommand{\bysame}{\leavevmode\hbox to3em{\hrulefill}\thinspace}
\providecommand{\MR}{\relax\ifhmode\unskip\space\fi MR }
\providecommand{\MRhref}[2]{%
  \href{http://www.ams.org/mathscinet-getitem?mr=#1}{#2}
}
\providecommand{\href}[2]{#2}

\end{document}